\documentclass[a4paper,10pt, reqno]{amsart}

\usepackage{amsmath,amssymb,latexsym, amsfonts}
\usepackage{verbatim}
\usepackage{times}
\usepackage{mathbbol}
\usepackage{hyperref}
\usepackage{mathabx}
\usepackage{bbm}

\newcount\mycount

\newcommand{\lie}{\mathcal{L}}
\newcommand{\BB}{\mathsf{B}}
\newcommand{\SSS}{\mathsf{S}}
\newcommand{\bb}{\mathsf{b}}
\newcommand{\hh}{\mathsf{h}}
\newcommand{\DD}{\mathsf{D}}
\newcommand{\EE}{\mathsf{E}}
\newcommand{\dd}{\mathsf{d}}
\newcommand{\sss}{\mathsf{s}}
\newcommand{\ttt}{\mathsf{t}}
\newcommand{\TT}{\mathsf{T}}
\newcommand{\NN}{\mathsf{N}}
\newcommand{\DCU}{\mathcal{D}^-(U)}

\numberwithin{equation}{section}

\newcommand{\cyc}{{\scriptscriptstyle\mathrm{cyc}}}
 
\theoremstyle{plain}

\newtheorem{theorem}{Theorem}[section]
\newtheorem{proposition}[theorem]{Proposition}
\newtheorem{prop}[theorem]{Proposition}
\newtheorem{lemma}[theorem]{Lemma}

\newtheorem{corollary}[theorem]{Corollary}

\theoremstyle{definition}
\newtheorem{definition}[theorem]{Definition}
\newtheorem{example}[theorem]{Example}

\newtheorem{rem}[theorem]{Remark}

\newcommand{\ahha}{{\scriptscriptstyle{A}}}

\newcommand{\emme}{{\scriptscriptstyle{M}}}

\newcommand{\uhhu}{{\scriptscriptstyle{U}}}
\newcommand{\pehhe}{{\scriptscriptstyle{P}}}

\newcommand{\ga}{\alpha} 
\newcommand{\gb}{\beta}  

\newcommand{\gd}{\delta} 
\newcommand{\gD}{\Delta} 
\newcommand{\gve}{\varepsilon} 
  
\newcommand{\gvf}{\varphi}

\newcommand{\go}{\omega}

\newcommand{\gs}{\sigma}

\newcommand{\Hom}{\operatorname{Hom}}
\newcommand{\Der}{\operatorname{Der}}

\newcommand{\Tor}{{\rm Tor}}
\newcommand{\Ext}{{\rm Ext}}

\newcommand{\id}{{\rm id}}

\newcommand{\pr}{{\rm pr} \,}

\newcommand{\due}[3]{{}_{{#2 }} {#1}_{{ #3}}\,}    % Zweifachindex

\newcommand{\pl}{\partial}

\newcommand{\rmref}[1]{{\rm (}\ref{#1}{\rm )}}

\newcommand{{\Hl}}{{H^{\ell}}} 
\newcommand{{\mHop}}{{m_{H^{\rm op}}}} 
\newcommand{{\Hop}}{{H^{\rm op}}} 
\newcommand{{\mUop}}{{m_{U^{\rm op}}}} 
\newcommand{{\mUopp}}{{m_{\scriptscriptstyle{U^{\rm op}}}}} 
\newcommand{{\Uop}}{{U^{\rm op}}}
\newcommand{{\mVop}}{{m_{V^{\rm op}}}} 
\newcommand{{\Vop}}{{V^{\rm op}}}  
\newcommand{{\Ae}}{{A^{\rm e}}}
\newcommand{{\Be}}{{B^{\rm e}}}
\newcommand{{\Aop}}{{A^{\rm op}}}
\newcommand{{\Aope}}{({A^{\rm op}})^{\rm e}}
\newcommand{{\Aopl}}{{A^{\rm op}_\pl}}

\newcommand{{\Bop}}{{B^{\rm op}}}
\newcommand{{\Bope}}{({B^{\rm op}})^{\rm e}}
\newcommand{{\Bpl}}{{B_\pl}}

\newcommand{{\op}}{{{\rm op}}}
\newcommand{{\coop}}{{{\rm coop}}}
\newcommand{{\sop}}{{*^{\rm op}}}

\newcommand{\amoda}{A^{\rm e}\mbox{-}\mathbf{Mod}}                  %
\newcommand{\moda}{A^\mathrm{op}\mbox{-}\mathbf{Mod}}         %
\newcommand{\umod}{U\mbox{-}\mathbf{Mod}}                     %  Modul-Kategorien
\newcommand{\modu}{U^\mathrm{op}\mbox{-}\mathbf{Mod}}         %

\newcommand{\hmu}{H_M^\bull(U)}
\newcommand{\hhmu}{H^M_\bull(U)}

\newcommand{\lact}{\smalltriangleright}                  %
\newcommand{\ract}{\smalltriangleleft}
\newcommand{\blact}{\blacktriangleright}  %
\newcommand{\bract}{\blacktriangleleft}   %

\newcommand{{\gog}}{{G \rightrightarrows G_0}}

\newcommand{{\rra}}{\rightrightarrows}

\newcommand{{\lra}}{\ \longrightarrow \ }
\newcommand{{\lla}}{\ \longleftarrow \ }
\newcommand{{\lma}}{\ \longmapsto \ }

\newcommand{{\bull}}{{\scriptscriptstyle{\bullet}}}
\newcommand{{\qqquad}}{{\quad\quad\quad}}
\newcommand{\Aopp}{{\scriptscriptstyle{\Aop}}}

\newcommand{\ubs}[2]{\underset{\scriptscriptstyle{#1}}{\underbrace{#2}}}

\begin{document}

\title{Batalin-Vilkovisky Structures on $ {\Ext} $ and $ {\Tor} $}

\author{Niels Kowalzig} 
\author{Ulrich Kr\"ahmer}

\begin{abstract}
This article studies the algebraic structure of 
homology theories defined by a left Hopf algebroid 
$U$ over a possibly
noncommutative base algebra $A$, such as for example 
Hochschild, Lie algebroid (in particular Lie algebra and Poisson), or
group and \'etale groupoid (co)homology.  
Explicit formulae for 
the canonical Gerstenhaber algebra structure 
on $ \mathrm{Ext}_U(A,A)$ are given.
The main technical result constructs a 
Lie derivative satisfying a generalised 
Cartan-Rinehart homotopy formula
whose essence is that
$\mathrm{Tor}^U(M,A)$ becomes 
for suitable right $U$-modules $M$ 
a Batalin-Vilkovisky module over 
$ \mathrm{Ext}_U(A,A)$, or in the words of 
Nest, Tamarkin, Tsygan and others, that 
$ \mathrm{Ext}_U(A,A)$ and
$\mathrm{Tor}^U(M,A)$ form a differential calculus. 
As an illustration, we show how the well-known operators from differential geometry in the classical Cartan homotopy formula can be obtained.
Another application consists in generalising       
Ginzburg's result that
the cohomology ring of a Calabi-Yau algebra is a Batalin-Vilkovisky
algebra to twisted Calabi-Yau algebras. 
\end{abstract}

\address{University of Glasgow,
School of Mathematics \& Statistics, University 
Gardens, Glasgow G12 8QW, Scotland}

\email{niels.kowalzig@glasgow.ac.uk}
\email{ulrich.kraehmer@glasgow.ac.uk}



\maketitle

\tableofcontents

\section{Introduction}
\subsection{Differential calculi}
By its definition in terms of
(co)chain complexes or derived functors,
the cohomology or homology of 
a mathematical object is typically only a graded module over some base ring.  
Thus an obvious task is to exhibit its full algebraic structure, 
and to understand which features of the 
original object this structure reflects.

For the (co)homology of associative algebras, 
this has been studied, 
amongst others, by Rinehart \cite{Rin:DFOGCA},
Gerstenhaber \cite{Ger:TCSOAAR}, 
Goodwillie \cite{Goo:CHDATFL}, 
Getzler \cite{Get:CHFATGMCICH} and
Nest, Tamarkin and
Tsygan, see e.g.~\cite{NesTsy:OTCROAA, TamTsy:NCDCHBVAAFC, TamTsy:CFAIT, Tsy:CH}. 
The ultimate answer is that Hochschild cohomology and homology form
what Nest, Tamarkin and Tsygan call a differential calculus:  

\begin{definition}
\label{golfoaranci}
Let $k$ be a commutative ring.
\begin{enumerate}
\item
A {\em Gerstenhaber algebra} 
over $k$ is a graded commutative $k$-algebra
$(V,\smallsmile)$  
$$
        V=\bigoplus_{p \in \mathbb{N}} V^p,\quad
        \ga \smallsmile \gb=(-1)^{pq}\gb \smallsmile \ga
        \in V^{p+q},\quad 
        \ga \in V^p,\gb \in V^q,
$$ 
with a graded Lie bracket 
$
        \{\cdot,\cdot\} : V^{p+1} \otimes_k V^{q+1} \rightarrow V^{p+q+1}
$ 
on the \emph{desuspension} 
$$
        V[1]:=\bigoplus_{p \in \mathbb{N}} V^{p+1}
$$
of $V$ 
for which all operators $\{\gamma,\cdot\}$ satisfy the graded Leibniz rule
$$
        \{\gamma,\ga \smallsmile \gb\}=
        \{\gamma,\ga\} \smallsmile \gb + (-1)^{pq} \ga \smallsmile
        \{\gamma,\gb\},\quad
        \gamma \in V^{p+1},\ga \in V^q.
$$
\item 
A {\em Gerstenhaber module} over  
$V$ is a graded $(V,\smallsmile)$-module 
$(\Omega,\smallfrown)$,
$$
        \Omega=\bigoplus_{n \in \mathbb{N}} \Omega_n,\quad
        \ga \smallfrown x \in \Omega_{n-p},\quad
        \ga \in V^p,x \in \Omega_n,
$$
with a representation of the graded Lie algebra
$(V[1],\{\cdot,\cdot\})$ 
$$
        \lie : V^{p+1} \otimes_k \Omega_n \rightarrow \Omega_{n-p},\quad
        \ga \otimes_k x \mapsto \lie_\ga (x), 
$$
which satisfies for 
$\ga \in V^{p+1},\gb \in V^q, x \in \Omega$ 
the mixed Leibniz rule
$$
 \gb \smallfrown \lie_\ga(x) =  \{\gb,\ga\} \smallfrown  x +(-1)^{pq} \lie_\ga(\gb \smallfrown x).
$$ 
\item 
Such a module  
is {\em Batalin-Vilkovisky} if there is a $k$-linear differential
$$
        \BB : \Omega_n \rightarrow \Omega_{n+1},\quad \BB  \BB=0,
$$
such that $\lie_\ga$ is for $\ga \in V^p$ given by  the homotopy formula 
$$
        \lie_\ga(x) =\BB(\ga \smallfrown x)-(-1)^p \ga \smallfrown 
        \BB(x). 
$$
\item 
A pair $(V,\Omega)$ 
of a Gerstenhaber algebra and of a Batalin-Vilkovisky module 
over it is also called a {\em differential calculus}.
\end{enumerate}
\end{definition}

Be aware that the term ``Gerstenhaber module''
is used in several 
different ways in the literature. The above one is based on the   
requirement that the operators 
$$
\iota_\ga := \ga \smallfrown \cdot : \Omega \rightarrow \Omega 
$$
form a Gerstenhaber algebra quotient of $V$ with bracket given by
$$
        \{\iota_\alpha,\iota_\beta\}:=[\iota_\alpha,\lie_\beta], 
$$
where $[\cdot,\cdot]$ denotes the graded commutator.
This agrees 
(up to slightly different sign conventions) with the
one used in \cite{DeSHekKac:CSOTLCACATVC}. One
will often additionally find that the mixed Leibniz rule 
\begin{equation}
\label{mysterious}
        \lie_{\ga \smallsmile \gb}=\lie_\ga \iota_\gb+
        (-1)^i \iota_\ga \lie_\gb,\quad
        \ga \in V^i,\gb \in V,
\end{equation} 
is demanded.
This is necessary for $V \oplus \Omega$ to become naturally a  
(square zero) extension of $V$ as a Gerstenhaber algebra. For
Batalin-Vilkovisky modules, Equation (\ref{mysterious}) 
is satisfied automatically,
so the definition of these is essentially unequivocal.

The definition of a Gerstenhaber algebra itself also admits
a modification in which the operators $\{\cdot,\gamma\}$, rather 
than $\{\gamma,\cdot\}$, are assumed to satisfy the graded Leibniz
rule. This had been the convention in Gerstenhaber's original paper
\cite{Ger:TCSOAAR}, cf.~Remark~\ref{esgehtvoran} below.  

\subsection{Aims and objectives}
The main aim of this paper is to further highlight 
the ubiquity of such Batalin-Vilkovisky structures 
by giving conditions for 
$$
        V :=\mathrm{Ext}_U(A,A),\quad
        \Omega := \mathrm{Tor}^U(M,A)
$$ 
to form a differential calculus 
when $U$ is 
a left Hopf algebroid (a $\times_\ahha$-Hopf algebra) 
over a possibly noncommutative $k$-algebra $A$;
we will recall some background on bialgebroids and Hopf algebroids 
in \S\ref{prelim} below. Here we only 
remind the reader that the rings governing 
most parts of classical
homological algebra all carry this structure, so that our results apply
for example to Hochschild and Lie-Rinehart (in particular 
Lie algebra, de Rham, Lie algebroid and Poisson) (co)homology 
as well as to that of any Hopf algebra (e.g.~group 
(co)homology).

Besides for the case of 
Hochschild (co)homology with canonical
coefficients $M=A$ that has been referred to above, 
our results are also already known for
Lie-Rinehart (co)homology due to the work of 
Rinehart and of Huebschmann
\cite{Rin:DFOGCA,Hue:DFLRAATMC}.
However, the Hopf algebroid generalisation is, in our opinion, not only
interesting because of new special cases to which it applies, 
but also leads to conceptually clearer statements and proofs. 
A case in point is that the cohomology coefficients are left $U$-modules while the homology ones are right $U$-modules, and this distinction is lost in many concrete examples such as group, Lie algebra, Poisson or Hochschild homology. 

For such reasons, we hope that the paper 
is of interest also to people working in
different but analogous settings  
in algebra, geometry and topology, 
see e.g.~\cite{BehFan:GABVSOLI,GinTra:DOABVSINCG,Men:BVAACCOHA, Men:CMCMIALAM, DotShaVal:GGAOTFTAHBVA}
and the
references therein.

\subsection{The Gerstenhaber algebra}
The Hopf algebroid (in fact, the underlying bialgebroid) structure on
$U$ leads to a monoidal structure on the category 
$\umod$ of left $U$-modules, and it is this monoidal
structure which is responsible for the Gerstenhaber algebra structure on 
$\mathrm{Ext}^\bull_U(A,A)$ that we consider here.  
This can be viewed as a special case of Menichi's
operadic construction \cite{Men:BVAACCOHA} 
that, in turn, closely follows Gerstenhaber's original work
on Hochschild cohomology
\cite{Ger:TCSOAAR}, or of 
Shoikhet's generalisation \cite{Sho:HATANMC} 
of Schwede's homotopy theory
approach to the Hochschild case \cite{Schw:AESIOTLBIHC}.
%% Both imply that the derived endomorphisms 
%% $ \mathrm{Ext}_\mathcal{C}(\mathbb{1},\mathbb{1})$  
%% of the unit object $\mathbb{1}$  
%% of a mildly restricted abelian monoidal
%% category $\mathcal{C}$ carry a
%% natural Gerstenhaber algebra structure. 

The aim of \S\ref{ud} is to 
give 
explicit formulae for 
$\smallsmile$ and $\{\cdot,\cdot\}$
in terms of the canonical cochain complex 
$$
        \delta : C^\bull(U,A) :=
        \Hom_\Aop\big(({U^{\otimes_\Aopp \bull}})_\ract, A\big)
        \rightarrow C^{\bull+1}(U,A) 
$$
that arises from the
bar resolution of $A$. We refer to the main text for the
notation used here and below. In particular, see
\S\ref{durchgegangen} for the definition of the four actions
$\lact,\ract,\blact,\bract$ of the base algebra $A$ on $U$.

On the level of cochains  
$ \varphi \in C^p(U,A),\psi \in C^q(U,A)$ 
the cup product turns out to be
\begin{equation}
\label{ostkreuz2}
        (\varphi \smallsmile \psi)(u^1, \ldots, u^{p+q}) = 
        \varphi\big(u^1, \ldots, u^{p-1}, 
        \psi(u^{p+1}, \ldots, u^{p+q}) \blact u^p\big).
\end{equation}
We then define along the classical lines
{\em Gerstenhaber products} 
$\circ_i$ by
\begin{equation*}
\begin{split}
&\ (\varphi \circ_i \psi)(u^1, \ldots, u^{p+q-1}) \\
&\ \quad := \varphi(u^1, \ldots, u^{i-1}, \DD_\psi(u^{i}, \ldots, u^{i+q-1}), u^{i+q}, \ldots, u^{p+q-1}), 
\end{split}
\end{equation*}
for $i = 1, \ldots, p$,
where the operator
$$
        \DD_\varphi: 
        U^{\otimes_\Aopp p} \to U,\quad 
        (u^1, \ldots, u^p) \mapsto 
        \varphi(u^1_{(1)}, \ldots, u^p_{(1)}) \lact 
        u^1_{(2)} \cdots u^p_{(2)}
$$
replaces the classical insertion operations used by Gerstenhaber. 
The $\circ_i$ are used to construct the Gerstenhaber bracket 
as usual as
\begin{equation}
\label{zugangskarte2}
{\{} \varphi,\psi \}
:= \varphi \bar\circ \psi - (-1)^{|p||q|} \psi \bar\circ \varphi
\end{equation}
with
$$
        \varphi \bar\circ \psi := 
        (-1)^{|p||q|}\sum^{p}_{i=1}
        (-1)^{|q||i|} \varphi \circ_i \psi,\quad 
        |n|:=n-1. 
$$

In \S\ref{ud} we will prove:

\begin{theorem}\label{pitandbull}
If $U$ is a bialgebroid 
over $A$, then the maps \rmref{ostkreuz2} and 
\rmref{zugangskarte2} induce a Gerstenhaber algebra structure on 
$H^\bull(U,A):=H^\bull(C^\bull(U,A), \delta )$.
\end{theorem}

When $U$ is a left Hopf algebroid and 
$U_\ract \in \moda$ is projective, the bar resolution is a 
projective resolution, so $H^\bull(U,A) \simeq \mathrm{Ext}_U(A,A)$
and the above result yields Gerstenhaber brackets on various 
$ \mathrm{Ext} $-algebras.   
Even for Hopf
algebras (i.e., for $A=k$) this has been discussed still fairly 
recently, see e.g.~\cite{FarSol:GSOTCOHA,Tai:IHBCOIDHAAGCOTYP,Men:CMCMIALAM}. 

%Note that even 
%if the canonical bracket vanishes, there may, of course,  
%still be different nontrivial ones --- classifying the Gerstenhaber
%algebra structures over a given graded commutative algebra 
%such as the derived endomorphisms of the unit object
%of a monoidal category is a task analogous 
%to the one of classifying the Poisson
%structures on a given manifold, so even if there is a
%canonical choice, one should expect there are many further ones. 

\subsection{The Gerstenhaber module}

In \cite{KowKra:CSIACT} we have 
studied the fact that for a left Hopf algebroid 
$U$ a left $U$-comodule 
structure on a right $U$-module $M$
induces
a para-cyclic $k$-module
structure on the canonical chain complex 
$$
        C_\bull(U,M) := M \otimes_\Aopp  ({}_\blact U_ \ract)^{\otimes_\Aopp \bull}
$$
that computes $\mathrm{Tor}^U(M,A)$
when $U$ is a right $A$-projective.

The question whether this leads to 
a Batalin-Vilkovisky module structure on the simplicial homology 
$H_\bull(U,M)$ of this para-cyclic object  
hinges on the compatibility between
the left $U$-comodule and the right $U$-module structure on $M$.  
In full generality, we define
for  $\gvf \in C^p(U,A)$
the {\em cap product}
\begin{equation}\label{ichhabe}
    \iota_\varphi(m,u^1, \ldots, u^n)  = 
        (m, u^1, \ldots, u^{n-p-1}, 
        \varphi(u^{n-|p|}, \ldots, u^n) \blact u^{n-p}),
  \end{equation}
and the {\em Lie derivative} (see the main text for all necessary details) 
\begin{equation}\label{dnv}
        \lie_\varphi := 
        \sum^{n-|p|}_{i=1} 
        (-1)^{\theta^{n,p}_{i}} 
        \ttt^{n - |p| - i} \, \DD'_\varphi \, \ttt^{i+p} 
        + \sum^{p}_{i=1} 
        (-1)^{\xi^{n,p}_{i}} 
        \ttt^{n-|p|} \, \DD'_\varphi \, \ttt^i,  
  \end{equation}
where $\theta$ and $\xi$ are sign functions, $\DD'_\gvf$ is $\DD_\gvf$
applied on the last $p$ components of an element in $C_n(U,M)$, and
$\ttt$ is the cyclic operator of the para-cyclic module $C_\bull(U,M)$
as in \rmref{landliebe}.

In general, these do not induce a Gerstenhaber module structure 
on $H_\bull(U,M)$, but only on the
homology  
${\hhmu}$ of the universal cyclic quotient 
$C_\bull^\cyc(U,M)$, see \S\ref{sting}. 
% In order to clearly distinguish the two from each other we refer to
% the earlier as to the \emph{homology of $U$ with coefficients in $M$}
% and to the latter as to the \emph{Hochschild homology of $U$ with
% coefficients in $M$}. 
A sufficient condition for the two to coincide is 
that $M$ is a stable anti Yetter-Drinfel'd module in which case the
para-cyclic $k$-module is cyclic, see again
\S\ref{sting} and \S\ref{lufthansaschmerz} below. However, 
a more general case 
that is ubiquitous in examples is
the following:

\begin{definition}\label{quasitsyglic}
A para-cyclic $k$-module $(C_\bull,\dd_\bull,\sss_\bull,\ttt_\bull)$
is {\em quasi-cyclic} if we have  
$$C_\bull= \mathrm{ker}\, (\mathrm{id}-\ttt_\bull^{\bull+1}) \oplus
\mathrm{im}\, (\mathrm{id}-\ttt_\bull^{\bull+1}).$$   
\end{definition}
We refer to \S\ref{sting} for the detailed explanation of
this condition and of its consequences. 
% We will in particular need it 
% when establishing a Hopf
% algebroid generalisation of the Cartan-Rinehart homotopy formula 
% (see Theorem~\ref{calleelvira}) 
% which is the main technical achievement in the present paper: 
In complete generality, we introduce
for any module-comodule $M$ (see Definition~\ref{modcomoddef}) 
the 
set $C^\bull_M(U) \,{\subseteq}\,C^\bull(U,A)$ 
consisting of those cochains for which the operators 
$\iota_\varphi$ and $\lie_\varphi$ descend to $C_\bull^\cyc(U,M)$. This
turns out to be a subcomplex whose cohomology will be denoted by 
$\hmu$. Then we prove:

\begin{theorem}\label{gestrichen}
For all modules-comodules $M$ over a left Hopf algebroid $U$, 
\rmref{ostkreuz2} and 
\rmref{zugangskarte2} induce a Gerstenhaber algebra structure on 
$\hmu$, and 
\rmref{ichhabe} and \rmref{dnv} induce a 
$\hmu$-Gerstenhaber module
structure on ${\hhmu}$.    
\end{theorem}

\subsection{The Batalin-Vilkovisky module}
Once this is established, we introduce 
the operator
$$
\SSS_\varphi := \sum^{n-p}_{j=0} \, 
\sum^j_{i=0} (-1)^{\eta^{n,p}_{j,i}} \ttt \sss_{n-|p|} \,
\ttt^{n-p-i} \, \DD'_\varphi \, \ttt^{n+i-|j|},
$$
where $\eta$ is again a sign function, and prove 
that for $ \varphi \in C^\bull_M(U)$ the 
{\em Cartan-Rinehart homotopy formula}
$$
\lie_\varphi = [\BB+\bb, \SSS_\varphi + \iota_\varphi] - \iota_{\gd\varphi} - \SSS_{\gd\varphi}
$$
is satisfied. 
Here $\bb$ and $\BB$ are the simplicial resp.~cyclic 
differentials on $C^\cyc_\bull(U,M)$ 
and $\delta$ is the cosimplicial differential on 
$C^\bull_M(U)$. This implies our main result:

\begin{theorem}\label{pen1}
For all module-comodules $M$ over a left Hopf algebroid 
$U$, the pair 
$(\hmu,\hhmu)$ 
carries a canonical structure of a differential calculus.
\end{theorem}

In the simplest case where $M$ is an SaYD module, we already mentioned that $C_\bull^\cyc(U,M)$ coincides with $C_\bull(U,M)$, and therefore we obtain:

\begin{corollary}
\label{pen2}
If $M$ is a stable anti Yetter-Drinfel'd module over a left Hopf
algebroid $U$ and if $U_\ract \in \moda$ is projective, then 
the pair $\big(\mathrm{Ext}_U(A,A),\mathrm{Tor}^U(M,A)\big)$ 
carries a canonical structure of a differential calculus.
\end{corollary}

For the Hochschild (co)homology of commutative associative algebras,
the 
earliest account of the set of operators $\bb, \BB, \iota, \lie$, and $\SSS$ is due to Rinehart \cite{Rin:DFOGCA}, where
these operators are called (in the same order) $\Delta$, $\bar d$,
$c$, $\theta$, and $f$, respectively.
For noncommutative algebras, the Lie derivative appeared for $1$-cocycles in \cite[p.~124]{Con:NCDG}, where it is denoted by $\delta^*$, and in \cite{Goo:CHDATFL}, where additionally the operators $\iota$ and $\SSS$ are introduced, denoted by $e$ and $E$, respectively. 
Finally, these operators were generalised from $1$-cocycles to arbitrary cochains both in \cite{Get:CHFATGMCICH}, where they are denoted by $\mathbf{b}$ and $\mathbf{B}$, as well as in \cite{GelDalTsy:OAVONCDG, NesTsy:OTCROAA, NesTsy:TFTCFAA, Tsy:CH}, the notation of which we take over.

\subsection{Applications}
A prominent example that forces one to go beyond SaYD 
modules is that of the Hochschild homology of
an algebra $A$ with coefficients in 
$M=A_\sigma$ for some automorphism 
$\sigma $ of $A$, that is, $M$ is $A$ as a $k$-module with 
$A$-bimodule structure given by $a \blact b \bract c:=ab \sigma (c)$. 
Whenever $ \sigma $ is semisimple, 
the resulting para-cyclic $k$-module is quasi-cyclic, and in the final
section of the paper we prove that  
this implies the following generalisation of   
a result of Ginzburg \cite{Gin:CYA} 
from Calabi-Yau algebras (which form the case 
in which $\sigma$ is inner) 
to twisted Calabi-Yau algebras (see Definition~\ref{mathar}), such as
the standard quantum groups  
\cite{BroZha:DCATHCFNHA},  
Koszul algebras whose Koszul dual is Frobenius as, for example, 
Manin's quantum plane 
\cite{VdB:ARBHHACFGR}, or the Podle\'s quantum
2-sphere \cite{Kra:OTHCOQHS}:

\begin{theorem}\label{ginzburger}
If $A$ is a twisted Calabi-Yau algebra with semisimple 
modular automorphism,
then the Hochschild cohomology 
$H^\bull(A,A)$ of $A$ is a Batalin-Vilkovisky algebra.  
\end{theorem}

Besides this application, we also explain in the penultimate section of
the paper how one can use our formulae to obtain the classical operators in Cartan's {\em magic} formula in differential geometry, i.e., the
Lie derivative, the insertion operator, and the de Rham differential 
in the setting of
Lie-Rinehart algebras (or Lie algebroids, and in particular the tangent bundle of a smooth manifold) by taking for $U$ the 
\emph{jet space} $J\!L$, which is the dual of the universal enveloping
algebra $V\!L$ of a Lie-Rinehart algebra $(A,L)$.
\
\\
\
\\
%\subsection*{Acknowledgements}
\thanks{ {\bf Acknowledgements.}   \!
It is our pleasure to thank Ryszard Nest, Boris Shoikhet, and
Boris Tsygan for inspiring 
discussions and explaining to us some aspects of
their work. Furthermore, we thank the referee for their careful
reading and suggestions.
N.K.~acknowledges funding by the Excellence Network of the University of Granada (GENIL) and would like to thank the University of Glasgow for hospitality and support. 
U.K.~furthermore acknowledges 
funding by the
Polish Government Grant N201 1770 33,
the Marie Curie PIRSES-GA-2008-230836
network and the 
Royal Society/RFBR joint project 
JP101196/11-01-92612,
% Royal Society - Russian Foundation for Basic Research
% Grant of Misha Feigin and Dima Vasiliev, 
and thanks 
ITEP Moscow for hospitality.

\section{Preliminaries}\label{prelim}
In this section we recall preliminaries on bialgebroids, Hopf
algebroids, and cyclic homology, mainly from our two papers 
\cite{KowKra:DAPIACT,KowKra:CSIACT} as we use therein 
the same notation
and conventions as here.
For more detailed information on bialgebroids and 
Hopf algebroids and references to the original 
sources, we recommend B\"ohm's 
survey \cite{Boe:HA}.

\subsection{Bialgebroids}\label{durchgegangen}
Throughout this paper, $A$ and $U$ are 
(unital associative) $k$-algebras, and we assume that there is 
a fixed $k$-algebra map
\begin{equation*}
\label{eta}
        \eta : \Ae := A \otimes_k \Aop \rightarrow 
        U.
\end{equation*}
This induces forgetful functors 
$$
\umod \rightarrow \amoda,\quad
\modu \rightarrow \amoda
$$ 
that
turn left $U$-modules $N$ respectively right $U$-modules $M$ 
into $A$-bimodules with actions
$$
        a \lact n \ract b:=\eta(a \otimes_k b)n,\quad
        a \blact m \bract b:=m \eta(b \otimes _k a),\quad
        a,b \in A,n \in N,m \in M.
$$%\end{equation}
In particular, left and right multiplication in $U$ 
defines $A$-bimodule structures of both these types on 
$U$ itself. Unless explicitly stated otherwise, we 
a priori consider $U$ as an $A$-bimodule using the actions   
$\lact,\ract$ arising from left multiplication in $U$. 
For example, in
(\ref{cerveza}) below the actions 
$\lact,\ract$ are used to define $U \otimes_\ahha U$, 
and later we will require $U$ to be right $A$-projective
meaning that $U_\ract \in \moda$ is projective.

Generalising the standard result for bialgebras (which 
is the case $A=k$), Schauenburg has proved 
\cite{Schau:DADOQGHA} that
the monoidal structures on $\umod$ for which the forgetful 
functor to $\amoda$ is strictly
monoidal (where $\amoda$ is monoidal via $\otimes_\ahha$) correspond to
what is known as \emph{(left) bialgebroid} (or $\times_\ahha$-bialgebra) structures
on $U$. We refer, e.g., to our earlier paper 
\cite{KowKra:DAPIACT} for a detailed definition 
(which is due to Takeuchi \cite{Tak:GOAOAA}). Let us only
recall that a bialgebroid has a coproduct and a counit 
\begin{equation}
\label{cerveza}
        \Delta : U \rightarrow U \otimes_\ahha U,\quad
        \varepsilon : U \rightarrow A,
\end{equation}
which turn $U$ into a coalgebra in $\amoda$.
One of the subtleties to keep in mind is 
that unlike for $A=k$ the counit $ \varepsilon $ is not necessarily a
ring homomorphism but only yields a left $U$-module structure on 
$A$ with action of $u \in U$ 
on $a \in A$ given by $u a := \varepsilon (u \bract a)$.
Furthermore, 
$\Delta$ is required to corestrict to a map from $U$ to the
Sweedler-Takeuchi product $U
\times_\ahha U$, which is the $\Ae$-submodule of  $U \otimes_\ahha U$
whose elements $\sum_i u_i \otimes_\ahha v_i$ fulfil 
\begin{equation}
\label{tellmemore}
 \textstyle
\sum_i a \blact u_i \otimes_\ahha v_i
		  = \sum_i u_i \otimes_\ahha v_i \bract a, \
		  \forall a \in A.
\end{equation}
In the sequel, we will freely use Sweedler's notation  
$ \Delta (u)=:u_{(1)} \otimes_\ahha u_{(2)}$.

\subsection{Hopf algebroids}
In the same paper \cite{Schau:DADOQGHA}, 
Schauenburg generalised
the notion of a Hopf algebra to the bialgebroid setting. What he called 
$\times_\ahha$-Hopf algebras will be called \emph{left Hopf algebroids} here. 
Again, we refer to \cite{KowKra:DAPIACT} for the definition, examples
and more background information, and only recall that  
the crucial piece of 
structure (in addition to a bialgebroid one) 
is the so-called \emph{translation map}
\begin{equation}
\label{pm}
		  U \rightarrow 
		  {}_\blact U \otimes_\Aopp U_\ract,
\end{equation}
for which we use the Sweedler-type notation
$$
u \mapsto u_+ \otimes_\Aopp u_-.
$$ 
\begin{example}
For a Hopf algebra 
over $A=k$, the translation map is given by 
$$u \mapsto u_{(1)} \otimes_k S(u_{(2)}),$$ where
$S$ is the antipode, and its relevance is
already discussed in great detail by Cartan and 
Eilenberg \cite{CarEil:HA}. 
\end{example}

We will make permanent use of the following identities that hold for the map \rmref{pm}, 
see \cite[Proposition~3.7]{Schau:DADOQGHA}: 
\begin{proposition}
Let $U$ be a left Hopf algebroid over $A$. For all $u, v \in U$, $a,b \in A$ one has 
\begin{eqnarray}
\label{Sch3}
u_+ \otimes_\Aopp  u_- & \in 
& U \times_\Aopp U, \\
\label{Sch1}
u_{+(1)} \otimes_\ahha u_{+(2)} u_- &=& u \otimes_\ahha 1 \in U_\ract \otimes_\ahha {}_\lact U, \\
\label{Sch2}
u_{(1)+} \otimes_\Aopp u_{(1)-} u_{(2)}  &=& u \otimes_\Aopp  1 \in  {}_\blact U
\otimes_\Aopp  U_\ract, \\ 
\label{Sch38}
u_{+(1)} \otimes_\ahha u_{+(2)} \otimes_\Aopp  u_{-} &=& u_{(1)} \otimes_\ahha
u_{(2)+} \otimes_\Aopp  u_{(2)-},\\
\label{Sch37}
u_+ \otimes_\Aopp  u_{-(1)} \otimes_\ahha u_{-(2)} &=& 
u_{++} \otimes_\Aopp
u_- \otimes_\ahha u_{+-}, \\
\label{Sch4}
(uv)_+ \otimes_\Aopp  (uv)_- &=& u_+v_+
\otimes_\Aopp  v_-u_-, 
\\ 
\label{Sch47}
u_+u_- &=& s (\varepsilon (u)), \\
\label{Sch48}
\varepsilon(u_-) \blact u_+  &=& u, \\
\label{Sch5}
(s (a) t (b))_+ \otimes_\Aopp  (s (a) t (b) )_- 
&=& s (a) \otimes_\Aopp  s (b), 
\end{eqnarray}
where in \rmref{Sch3} we mean the Sweedler-Takeuchi product
\begin{equation*}
\label{petrarca}
		  U \times_\Aopp  U:=
		  \left\{\textstyle\sum_i u_i \otimes_\Aopp  v_i \in 
		  {}_\blact U \otimes_\Aopp  U_\ract\mid
		  \sum_i u_i \ract a \otimes_\Aopp  v_i=
		  \sum_i u_i \otimes_\Aopp  a \blact
		  v_i
		  \right\},
\end{equation*}
which is an algebra by factorwise
multiplication, but with opposite 
multiplication on the second factor, and where in 
(\ref{Sch47}) and (\ref{Sch5}) we use the \emph{source} and \emph{target} maps 
\begin{equation}
\label{basmati}
        s,t : A \rightarrow U,\quad s(a):=\eta (a \otimes_k 1),\quad
        t(b):=\eta (1 \otimes_k b). 
\end{equation}
\end{proposition}

For us, the relevance of the 
translation map stems mostly from the fact that it turns the
category $\modu$ of right $U$-modules 
into a module category over the monoidal category $\umod$. 
Explicitly, the product of $N \in \umod$ with $M \in \modu$ is the
tensor product of the underlying $A$-bimodules with right action given
by
$$
        (n \otimes_\ahha m)u:=
        u_-n \otimes_\ahha mu_+,\quad
        u \in U,m \in M,n \in N.
$$ 

\subsection{Module-comodules and anti Yetter-Drinfel'd modules}
Throughout this paper, $M$ will denote a right $U$-module,  
and in fact one which is simultaneously a comodule:
\begin{definition}\label{modcomoddef}
By a \emph{module-comodule} 
(with compatible induced left $A$-action)
over a bialgebroid $U$ we shall mean 
a right $U$-module $M \in \modu$ 
for which the underlying left $A$-module ${}_\blact M$ 
is also equipped with a left $U$-coaction
$$
     \Delta_\emme:   M \rightarrow 
        U_\ract \otimes_\ahha {}_\blact  M, \quad
        m \mapsto m_{(-1)} \otimes_\ahha m_{(0)}.
$$  
\end{definition}

Recall, e.g.~from \cite{Boe:HA}, 
that $\Delta_\emme$ is then
an $\Ae$-module morphism 
$ 
M \rightarrow 
        U_\ract \times_\ahha {}_\blact  M,
$
where $  U_\ract \times_\ahha {}_\blact  M$ is the $\Ae$-submodule of
$U_\ract \otimes_\ahha {}_\blact  M$ whose elements $\sum_i u_i
\otimes_\ahha m_i$ fulfil 
\begin{equation}
\label{auchnochnicht}
\textstyle
\sum_i a \blact u_i \otimes_\ahha m_i
		  = \sum_i u_i \otimes_\ahha m_i \bract a, \
		  \forall a \in A.
\end{equation}

The following particular class of module-comodules 
was introduced in  
\cite{HajKhaRanSom:SAYDM} for Hopf algebras and 
in \cite{BoeSte:CCOBAAVC} for left Hopf algebroids:
\begin{definition}
A module-comodule over a left Hopf algebroid
is called an \emph{anti Yetter-Drinfel'd module (aYD)} 
if the full $\Ae$-module structure 
${}_\blact M_\bract$ of the module coincides with that underlying 
the comodule, and if one has
\begin{equation*}
\label{huhomezone}
		 (mu)_{(-1)} \otimes_\ahha (mu)_{(0)}= 
		  u_- m_{(-1)} u_{+(1)} \otimes_\ahha m_{(0)} u_{+(2)}
\end{equation*}
for all $m \in M,u \in U$. A module-comodule 
is called \emph{stable (SaYD)} if one has
$$
m_{(0)}m_{(-1)} = m.
$$
\end{definition}

\subsection{The (para-)cyclic $k$-modules 
$C_\bull(U,M)$ and $C^\cyc_\bull(U,M)$}
\label{sting}
The Batalin-Vilkovisky modules 
that we are going to study in this paper are 
obtained as the simplicial homology of 
para-cyclic $k$-modules of the following form  
\cite{KowKra:CSIACT}:

\begin{proposition}
\label{kamille}
For every right module $M$ over a bialgebroid $U$ 
there is a well-defined simplicial 
$k$-module structure on   
$$
        C_\bull(U,M) := 
        M \otimes_\Aopp  
        ({}_\blact U_ \ract)^{\otimes_\Aopp
          \bull}
$$
whose face and degeneracy maps are given by
\begin{equation}\label{landliebe}
\!\!\!
\begin{array}{rcll}
        \dd_i(m,x)  
&\!\!\!\!\! =& \!\!\!\!\!\left\{ \!\!\!
\begin{array}{l}
        (m,u^1, \ldots,\varepsilon(u^n) \blact u^{n-1}),
\\
(m,\ldots,u^{n-i} u^{n-i+1},
\ldots,u^n) 
\\
(mu^1,u^2,\ldots,u^n) 
\end{array}\right.  & \!\!\!\!\!\!\!\!\!\!\!\! \,  \begin{array}{l} \mbox{if} \ i \!=\! 0, \\ \mbox{if} \ 1
\!  \leq \! i \!\leq\! n-1, \\ \mbox{if} \ i \! = \! n, \end{array} \\
\
\\
\sss_j(m,x) &\!\!\!\!\! =&\!\!\!\!\!  \left\{ \!\!\!
\begin{array}{l} (m,u^1,  
\ldots,u^n,1)
\\
(m,\ldots,u^{n-j}, 1,  
u^{n-j+1},\ldots,u^n)  
\\
(m,1,u^1,\ldots,u^n) 
\end{array}\right.   & \!\!\!\!\!\!\!\!\!\!\!  \begin{array}{l} 
\mbox{if} \ j\!=\!0, \\ 
\mbox{if} \ 1 \!\leq\! j \!\leq\! n-1, \\  \mbox{if} \ j\! = \!n, \end{array} \\
\
\end{array}
\end{equation}
Here and in what follows, 
we denote the elementary tensors in 
$C_\bull(U,M)$ by
$$
        (m,x):=(m,u^1,\ldots,u^n),\quad
m \in M,u^1,\ldots,u^n \in U.
$$
For a module-comodule $M$ over 
a left Hopf algebroid $U$, the $k$-module $C_\bull(U,M)$  
becomes a para-cyclic $k$-module via   
\begin{equation} 
\ttt_n(m,x) = 
(m_{(0)} u^1_+,u^2_+,\ldots,u^n_+,
u^n_- \cdots u^1_- m_{(-1)}).
\end{equation}
This para-cyclic $k$-module 
is cyclic if $M$ is a stable anti Yetter-Drinfel'd module.
\end{proposition}

Recall that this means that the operators $(\dd_i,\sss_j,\ttt_n)$ satisfy all
the defining relations of a cyclic $k$-module in the sense of Connes 
(see e.g.~\cite{Con:NCDG} or 
\cite{Lod:CH} for the definition of a cyclic $k$-module), except for the one that requires that 
$$
        \TT_n := \ttt_n^{n+1}
$$ 
equals the identity (we do not
even require it to be an isomorphism) which, 
as mentioned in the proposition,
is only satisfied when $M$ is an SaYD module.

The relations between the operators $(\dd_i,\sss_j,\ttt_n)$ imply that 
$\TT_n$ commutes with all of them, 
so they descend to well-defined operators on 
the coinvariants 
$$
        C^\cyc_\bull(U,M) := C_\bull(U,M)/\mathrm{im}\,
        (\mathrm{id}-\TT_\bull),
$$
and hence 
this becomes a cyclic $k$-module.

In this paper, we will not study the cyclic homology of this
object, but rather the simplicial homology 
of both $C_\bull(U,M)$ and $C_\bull^\cyc(U,M)$: 

\begin{definition}
For any bialgebroid $U$ and any $M \in \modu$, 
we denote the simplicial 
homology of $C_\bull(U,M)$, that is, the  
homology with respect to the
boundary map 
\begin{equation}\label{scotstoun}
        \bb := \sum_{i=0}^n (-1)^i \dd_i,  
  \end{equation}
by $H_\bull(U,M)$ and
call it the 
\emph{homology of $U$ with
coefficients in $M$}.  
For a module-comodule over a 
left Hopf algebroid, we denote the 
simplicial homology of $C_\bull^\cyc(U,M)$ 
by ${\hhmu}$.  
\end{definition}

In general,
$H_\bull(U,M)$ differs from 
${\hhmu}$,
see~\cite{HadKra:BHOQG} for an
example. However, if $C_\bull(U,M)$ is 
quasi-cyclic in the sense of Definition~\ref{quasitsyglic}, we can
apply \cite[Proposition~2.1]{HadKra:THOQSLTPT}:

\begin{proposition}
If $C_\bull$ is a quasi-cyclic $k$-module, then 
the canonical quotient map 
$$
C_\bull \rightarrow 
C_\bull/\mathrm{im}( \mathrm{id}-\ttt_\bull^{\bull+1})
$$ 
is a quasi-isomorphism of the chain complexes  
that are defined by the
underlying simplicial $k$-module structures of 
$
C_\bull$ and
$C_\bull/\mathrm{im}( \mathrm{id}-\ttt_\bull^{\bull+1})$, respectively. 
\end{proposition}

This means that if $C_\bull(U,M)$ happens to be quasi-cyclic, then  
classes in ${\hhmu}$ 
can be represented by cycles in 
$C_\bull(U,M)$
that are invariant under   
$\TT_\bull$.

Mostly, we will now work on the reduced (normalised) complexes
of $C_\bull(U,M)$ resp.~of
$C_\bull^\cyc(U,M)$ 
by the subcomplex spanned by the
images of the degeneracy maps of these simplicial 
$k$-modules.
Being slightly sloppy, we will denote operators that descend from the
original complexes to these quotients by the same symbols if no
confusion can arise. 
Furthermore, we shall drop in all what follows 
the subscript on $\ttt$ and $\TT$ 
indicating the degree of the element on which it acts. 

\subsection{The operators $\NN,\sss_{-1}$ and $\BB$}
On every para-cyclic $k$-module, 
one defines the {\em norm operator}, 
the {\em extra  degeneracy},
and the \emph{cyclic differential} 
\begin{equation}
\label{extra}
        \NN := \sum_{i=0}^n (-1)^{in} \ttt^i, \qquad         
        \sss_{-1} := \ttt \, \sss_n,\qquad
        \BB=(\mathrm{id}-\ttt) \, \sss_{-1} \, \NN.
\end{equation}
Recall that $\BB$ coincides on the 
reduced complex 
$\bar C_\bull(U,M)$
with the map (induced by) $\sss_{-1} \, \NN$, so we 
are also slightly sloppy about this and denote the latter by 
$\BB$ as well, as we, in fact, will only consider the induced map on the
reduced complex.

It follows from the para-cyclic relations 
that one has 
\begin{equation}\label{mixedcomp}
        \BB^2 =(\mathrm{id} -\TT )
        (\mathrm{id} - \ttt)\sss_{-1} \sss_{-1} \NN,\quad
        \bb \BB + \BB \bb=\mathrm{id}-\TT,
\end{equation} 
so in general
$\BB$ does not turn 
$H_\bull(U,M)$, but only
${\hhmu}$,
into a cochain complex.

In the case of an SaYD module $M$ one can 
give a compact expression for $\BB$: 
one first computes directly
with the help of \rmref{Sch1}, \rmref{Sch2}, \rmref{Sch38}, and \rmref{Sch37} 
the powers of $\ttt$:

\begin{lemma}
\label{kohinor}
If $M$ is an SaYD module, the $i^{\scriptscriptstyle{\rm th}}$ power 
for $1 \le i \le n$ of the cyclic operator $\ttt$ can be expressed as
\begin{eqnarray*}
%% \ttt(m,x) &\!\!\!\!\!=&\!\!\!\!\! 
%% (m_{(0)} u^1_+,u^2_+,\ldots,u^n_+,u^n_- \cdots u^1_- m_{(-1)}), \\
%% \ttt^2(m,x) &\!\!\!\!\!=&\!\!\!\!\! 
%% (m_{(0)} u^1_{+(2)} u^2_+,u^3_+, 
%% \ldots,u^n_+,u^n_- \cdots u^1_- m_{(-1)},u^1_{+(1)}), \\
%% & \vdots &\\
%% \ttt^n(m,x) &\!\!\!\!\!=&\!\!\!\!\! 
%% (m_{(0)} u^1_{+(2)} \cdots u^{n-1}_{+(2)} u^n_+,u^n_- \cdots u^1_-
%% m_{(-1)},u^1_{+(1)},\ldots,u^{n-1}_{+(1)}),
\ttt^i(m,x) &\!\!\!\!\!=&\!\!\!\!\! (m_{(0)} u^1_{+(2)} \cdots u^{i-1}_{+(2)} u^i_+, u^{i+1}_+, \ldots, u^n_+, u^n_- \cdots u^1_- 
m_{(-1)}, u^1_{+(1)},\ldots, u^{i-1}_{+(1)}),
\end{eqnarray*}
where we abbreviated here, as elsewhere, 
$(m,x) = (m,u^1, \ldots, u^n)$.
\end{lemma}

Then a further direct computation gives:

\begin{lemma}
If $M$ is an SaYD module, the action of $\BB=\sss_{-1}\NN$ on 
$\bar C_\bull(U,M)$ can be expressed as   
\begin{equation}
\label{alhambra}
\begin{split}
 \sss_{-1}\NN(m,x) = 
\sum_{i=0}^{n} (-1)^{in} 
(m_{(0)} &u^1_{+(2)} \cdots u^{i}_{+(2)}, u^{i+1}_+, \\
&\ldots, u^{n}_+, u^n_- \cdots u^1_- m_{(-1)}, u^1_{+(1)}, \ldots,  u^{i}_{+(1)}).
\end{split}
\end{equation}
\end{lemma} 

\begin{example}
For $n=1$, the above expression reduces to 
\begin{equation*}
\sss_{-1}\NN(m,u) = (m_{(0)},u_+,u_- m_{(-1)}) - (m_{(0)} u_{+(2)},u_- m_{(-1)},u_{+(1)}). 
\end{equation*}
In particular, for a Hopf algebra over $A=k$ this reads
\begin{equation*}
\sss_{-1}\NN(m,u) = (m_{(0)},u_{(1)},S(u_{(2)}) 
m_{(-1)}) - (m_{(0)} u_{(2)},S(u_{(3)}) m_{(-1)},u_{(1)}). 
\end{equation*}
\end{example}

\section{The Gerstenhaber algebra}
\label{ud}

Unless stated explicitly otherwise, 
$U$ is throughout this section 
an arbitrary left $A$-bialgebroid.
We will first give explicit formulae for a canonical 
DG coalgebra structure $\Delta^{\pehhe}$ 
on the chain complex $(P,\bb')$ that is obtained when applying
the bar construction 
for the comonad $U \otimes_\Aopp \cdot$
to the unit object $A \in \umod$. 
Applying 
$\mathrm{Hom}_U(\cdot,A)$ to $P$ yields a cochain complex 
$(C^\bull(U,A),\delta)$. On the underlying graded vector space   
we define the structure of a (nonsymmetric) operad with 
multiplication. This, in particular, 
defines a DG algebra structure $(C^\bull(U,A),\smallsmile,\delta)$
and a
Gerstenhaber algebra structure on its cohomology 
$H^\bull(U,A)$.  
The fact that this DG algebra  
coincides with the one obtained by dualising the DG coalgebra 
structure on $P$ proves that
as long as $U$ is a right $A$-projective left Hopf algebroid,
$H^\bull(U,A)$ 
is the cohomology ring $ \mathrm{Ext}_U(A,A)$ 
that we studied  
in \cite{KowKra:DAPIACT}.

We will throughout use the convention in which DG
algebras are cochain complexes while DG coalgebras and DG
modules over DG algebras are chain complexes.

\subsection{The bar resolution $P$}\label{baraufloesung}
The bar construction 
for $U \otimes_\Aopp \cdot$ applied
to $A \in \umod$ yields the chain complex 
$(P_\bull,\bb')$ of left $U$-modules, where 
$$
        P_n:=({}_\blact U_\ract)^{\otimes_{\Aopp} n+1}
$$ 
is a $U$-module via left multiplication in the first tensor
component,  and $\bb'$ is given by
\begin{align*}
        \bb' (u^0,\ldots,u^n) 
:=& \ \sum_{i=0}^{n-1} (-1)^i
        (u^0, \ldots,u^iu^{i+1},\ldots,u^n)\\ 
 &\ + (-1)^n (u^0,\ldots,u^{n-2},
        \varepsilon(u^n) \blact u^{n-1}).
\end{align*}
Note that the tensor product over $\Aop$ is chosen in such a way 
that 
$$
        (u^0,\ldots,a \blact u^i,u^{i+1},\ldots,u^n)=
        (u^0,\ldots,u^i,u^{i+1} \ract a,\ldots,u^n)
$$
holds, which is necessary for $\bb'$ to be well-defined. 
%For that type of
%tensor product we will use 
%throughout the paper this notation $(x,\ldots,y)$, while 
%$x \otimes_\ahha \cdots \otimes_\ahha y$ will be used 
%whenever 
%$$x \otimes_\ahha \cdots 
%\otimes_\ahha  t \ract a \otimes_\ahha z \otimes_\ahha  \cdots 
%\otimes_\ahha  y=x \otimes_\ahha \cdots 
%\otimes_\ahha  t \otimes_\ahha a \lact z \otimes_\ahha \cdots 
%\otimes_\ahha  y.
%$$ 
We recall \cite[Lemma~2]{KowKra:DAPIACT}:

\begin{lemma}\label{alterhit}
If $U$ is a left Hopf algebroid and 
$U_\ract \in \moda$ is projective, then $(P_\bull,\bb')$ is a projective resolution of 
$A \in \umod$. 
\end{lemma}

\subsection{The DG coalgebra structure on $P$}   
As $\umod$ is monoidal, so is the category of chain 
complexes of $U$-modules and our 
aim is to turn $P$ into a coalgebra in this category.
 
\begin{definition}
We define
$$
        \Delta^{\pehhe} : P \rightarrow 
       P \otimes_\ahha P,\quad
        \Delta^\pehhe (u^0,\ldots,u^n):=  
        \sum_{i=0}^n \Delta^\pehhe_{ni}(u^0,\ldots,u^n),
$$
where for $i=0,\ldots,n$ the maps
$
        \Delta^\pehhe_{ni} : P_n \rightarrow 
        P_i \otimes_\ahha P_{n-i}
$
are given by
$$
        (u^0,\ldots,u^n) \mapsto 
        (u^0_{(1)},\ldots,u^i_{(1)}) \otimes_\ahha 
          (u^0_{(2)}\cdots u^i_{(2)},u^{i+1},\ldots,
          u^n).
$$
\end{definition}

We  verify by direct computation:
\begin{lemma}
$\Delta^\pehhe$ is coassociative. 
\end{lemma} 
\begin{proof}
For $j=0,\ldots,i$, we have 
\begin{align*}
  & \       ((\Delta^\pehhe_{i j} \otimes_\ahha 
        \mathrm{id}_{P_{n-i}})  \, \Delta^\pehhe_{ni})
        (u^0,\ldots,u^n)\\
 = & \ (u^0_{(1)},\ldots,u^j_{(1)})
 \otimes_\ahha   
        (u^0_{(2)} \cdots u^j_{(2)},
        u^{j+1}_{(1)},\ldots,
        u^i_{(1)})
 \\ & \ 
\otimes_\ahha   
          (u^0_{(3)}\cdots u^j_{(3)}
          u^{j+1}_{(2)} \cdots u^i_{(2)},u^{i+1},\ldots,
          u^n),
\end{align*}
and for $j=0,\ldots,n-i$, we have 
\begin{align*}
& \ ((\mathrm{id}_{P_{i}} \otimes_\ahha 
        \Delta^\pehhe_{n-i j})  \, \Delta^\pehhe_{ni})
       (u^0,\ldots,u^n)\\
 = & \   (u^0_{(1)},\ldots,u^i_{(1)})
 \otimes_\ahha 
  (u^0_{(2)} \cdots u^i_{(2)}, u^{i+1}_{(1)},
    \ldots,u^{i+j}_{(1)}) \\ & \ \otimes_\ahha 
    (u^0_{(3)} \cdots u^i_{(3)} u^{i+1}_{(2)} \cdots
    u^{i+j}_{(2)},u^{i+j+1},\ldots,u^n).  
\end{align*}
So for $\Delta^\pehhe$ to be coassociative, we need
\begin{align*}
 & \ \sum_{i=0}^n \sum_{j=0}^i  (u^0_{(1)},\ldots,u^j_{(1)})\otimes_\ahha    
        (u^0_{(2)} \cdots u^j_{(2)},
        u^{j+1}_{(1)},\ldots,
        u^i_{(1)}) \\
& \qquad\qquad \otimes_\ahha   
          (u^0_{(3)}\cdots u^j_{(3)}
          u^{j+1}_{(2)} \cdots u^i_{(2)},u^{i+1},\ldots,
          u^n)\\
= & \   \sum_{r=0}^n  \sum_{s=0}^{n-r} (u^0_{(1)},\ldots,u^r_{(1)})
 \otimes_\ahha
  (u^0_{(2)} \cdots u^r_{(2)}, u^{r+1}_{(1)},
    \ldots,u^{r+s}_{(1)}) \\
& \qquad\qquad \otimes_\ahha
    (u^0_{(3)} \cdots u^r_{(3)} u^{r+1}_{(2)} \cdots
    u^{r+s}_{(2)},u^{r+s+1},\ldots,u^n),  
\end{align*}
which is seen to be correct by some basic substitution in the indices,
writing first 
$$
\sum_{i=0}^n \sum_{j=0}^i = 
  \sum_{j=0}^n \sum_{i=j}^n,
$$
and then substituting 
$j$ by $r$ and $i$ by $s=i-j$.
\end{proof} 

\begin{proposition}
\label{karlmarxstadt}
If we define
$$
        \varepsilon^\pehhe:=\varepsilon :
        P_0 = U \rightarrow A
$$
and $\varepsilon^\pehhe|_{P_n}=0$ for $n>0$,
then $(P,\bb',\Delta^\pehhe,\varepsilon^\pehhe)$ is a differential
graded coalgebra.   
\end{proposition}
\begin{proof}
Both the counit property and the Leibniz rule 
\begin{equation}
\label{kropotkin}
        \Delta^\pehhe \, \bb'=
        (\bb' \otimes_\ahha \mathrm{id}_P +
        \mathrm{id}_P \otimes_\ahha \bb') \, \Delta^\pehhe
\end{equation}
are easily verified. We only remark that 
the above Equation (\ref{kropotkin}) is  
meant to be interpreted
using the Koszul sign convention, meaning that  
we have for all $c \in P_p,d \in P_q$
\begin{equation*}%\label{koszulsign}
(\mathrm{id}_P \otimes_\ahha \bb')(c \otimes_\ahha d)=
(-1)^p c \otimes_\ahha \bb'(d),
\end{equation*}
but $ (\bb' \otimes_\ahha \mathrm{id}_P)(c \otimes_\ahha d)=
\bb'(c) \otimes_\ahha d$, as $ \mathrm{id}_P$ is of degree 0.
\end{proof} 

\subsection{Comparison of $P$ and $P \otimes_\ahha P$}
Recall that so far it is sufficient to assume $U$ 
to be a left $A$-bialgebroid which is the algebraic 
underpinning of the fact that $\umod$ is monoidal with unit object
$A$. 
Using, for example, the standard spectral sequence of 
the bicomplex $P_\bull \otimes_\ahha P_\bull$, one easily verifies 
that the tensor product $P \otimes_\ahha P$ has homology 
$A \otimes_\ahha A \simeq A$; so it is, like $P$, a resolution of $A$.  
However, only when $U$ is a left Hopf algebroid, $P$ and $P
\otimes_\ahha P$ are necessarily quasi-isomorphic since 
in this case the tensor product of two projectives
in $\umod$ is projective \cite[Theorem 5]{KowKra:DAPIACT}. 
Proposition~\ref{karlmarxstadt} tells us that
$$
        \Delta^{\pehhe} : P \rightarrow P \otimes_\ahha P,\quad 
        \mathrm{id}_P \otimes_\ahha \varepsilon^P : 
        P \otimes_\ahha P \rightarrow P
$$ 
are morphisms of chain complexes that are one-sided inverses of
each other.  
In the left Hopf algebroid case the following proposition 
provides a
homotopy that shows that the maps become in this situation  
quasi-inverse to each other.
Note that this proposition 
is true for all left Hopf algebroids, assuming no projectivity of $U$
over $A$ (although, of course, without that $P$ is not 
necessarily a projective resolution).
 
\begin{proposition}\label{eisenhuettenstadt}
If $U$ is a left Hopf algebroid over $A$, then the maps
$$
\hh_n : \bigoplus\limits_{i+j = n} P_i \otimes_\ahha P_j \to 
\bigoplus\limits_{k+l = n+1} P_k \otimes_\ahha P_l
$$
given by
\begin{footnotesize}
\begin{equation*}
\begin{split}
& (u^0, \ldots, u^i) \otimes_\ahha (v^0, \ldots, v^j) \\
\mapsto & \sum^i_{r=0} (-1)^i (u^0_{+(1)}, \ldots, u^r_{+(1)}) \otimes_\ahha (u^0_{+(2)} \cdots u^r_{+(2)}, u^{r+1}_+, \ldots, u^i_+, u^i_- \cdots u^0_- v^0, v^1, \ldots, v^j) 
\end{split}
 \end{equation*}
\end{footnotesize}
define a homotopy equivalence 
$$ 
        \Delta^\pehhe \, (\mathrm{id}_P \otimes_\ahha
        \varepsilon^P) \sim \mathrm{id}_{P \otimes_\ahha P},
$$ 
so 
$\Delta^{\pehhe}$ and $\mathrm{id}_P \otimes_\ahha \varepsilon^P$ 
are mutual quasi-inverses and we have 
$P \simeq P \otimes_\ahha P$ as objects 
in the derived category $\DCU$.   
\end{proposition}

\begin{proof}
In degree $n=0$, the homotopy is
$$
        \hh_0 : u \otimes_\ahha v \mapsto 
        u_{+(1)} \otimes_\ahha (u_{+(2)},u_-v)=
        u_{(1)} \otimes_\ahha (u_{(2)+},u_{(2)-}v)
$$
and using the bialgebroid axioms as well as \rmref{Sch3}--\rmref{Sch5}, 
we get
\begin{align*}
        ((\mathrm{id}_U \otimes_\ahha \bb') \, \hh_0)(u \otimes_\ahha v)
&=
        u_{(1)} \otimes_\ahha (u_{(2)+}u_{(2)-}v-
        \varepsilon (u_{(2)-}v) \blact u_{(2)+})\\
&=
        u_{(1)} \otimes_\ahha \varepsilon (u_{(2)}) \lact v-
        u_{(1)} \otimes_\ahha  
        \varepsilon (\varepsilon (v) \blact u_{(2)-}) \blact u_{(2)+}\\
&=
        u_{(1)}\ract \varepsilon (u_{(2)})  \otimes_\ahha v-
        u_{(1)} \otimes_\ahha  
        \varepsilon (u_{(2)-}) \blact u_{(2)+} \ract \varepsilon (v)\\
&=
        u \otimes_\ahha v -
       u_{(1)} \otimes_\ahha 
        u_{(2)} \ract \varepsilon (v)\\
&= (\mathrm{id}_{U \otimes_\ahha U}-\Delta^{\pehhe} \, (\mathrm{id}_U
\otimes_\ahha \varepsilon^P))(u \otimes_\ahha v).
\end{align*}
Analogously, one computes that one has also for $n>0$  
\begin{align*}
& \quad \hh_{n-1} \, (\bb'  \otimes_\ahha \mathrm{id}_P+\mathrm{id}_P
\otimes_\ahha 
\bb') + (\bb'  \otimes_\ahha \mathrm{id}_P+\mathrm{id}_P \otimes_\ahha
\bb') \, \hh_n \\
=& \quad  \mathrm{id}_P - \Delta^{\pehhe} \, (\mathrm{id}_P
\otimes_\ahha \varepsilon^P).\qedhere
\end{align*}
\end{proof}

This fact demonstrates, on the one hand, the
homological difference between the bialgebroid and the left Hopf
algebroid case, and it also illustrates, on the other hand, that 
the cup and cap products we define below are
indeed the derived versions of the composition and contraction 
product 
that we dealt with abstractly in \cite{KowKra:DAPIACT}.

\subsection{$C^\bull(U,N)$ and the cup product}
We retain the assumption that 
$U$ is an $A$-bialgebroid and further denote by 
$P$ the DG coalgebra
defined in the previous sections.

\begin{definition}\label{xion1}
We define for all $N \in \umod$ the cochain complex
$$
        \hat{C}^\bull(U,N):=
        \mathrm{Hom}_U(P_\bull,N)
$$ 
with coboundary map $\hat\delta := \mathrm{Hom}_U(\bb',N)$, that is,
$$
        \hat\delta  : \hat{C}^p(U,N) \rightarrow 
        \hat{C}^{p+1}(U,N),\quad
        \hat\delta \hat\varphi := \hat\varphi \bb'.
$$ 
Furthermore, we define the {\em cup product}
$        \smallsmile \ :
        \hat{C}^\bull(U,A) \otimes_k
        \hat{C}^\bull(U,N) \rightarrow 
        \hat{C}^\bull(U,N)$
by
$$
        (\hat \varphi  \smallsmile \hat\psi)(c)
        :=\hat\psi (\hat\varphi (c_{(1)}) \lact c_{(2)})=
        \hat\varphi (c_{(1)}) \lact \hat\psi (c_{(2)}),
$$
where $c_{(1)} \otimes_\ahha c_{(2)}$ is 
$\Delta ^\pehhe(c)$ in Sweedler notation. 
\end{definition}

Note that for $N=A$ 
the cup product 
becomes simply 
the convolution product  
\begin{equation}
\label{xion2}
        (\hat \varphi  \smallsmile \hat \psi)(c)
        =\hat\varphi (c_{(1)})
        \hat\psi (c_{(2)}),
\end{equation}
and that 
Proposition~\ref{karlmarxstadt} implies:

\begin{corollary} 
\label{medtner}
$(\hat{C}^\bull(U,A),\hat\delta,\smallsmile)$ is a differential graded
algebra and $(\hat{C}^\bull(U,N),\hat\delta,\smallsmile)$ is a
differential graded left module over $\hat{C}^\bull(U,A)$. 
\end{corollary}

By $U$-linearity of $\hat\psi \in \hat{C}^\bull(U,A)$ 
we obtain in a standard fashion the isomorphism
\begin{equation}
\label{gesundbrunnen}
        \hat{C}^p(U,N) \overset{\simeq}{\longrightarrow} 
        C^p(U,N) :=
        \Hom_\Aop({U^{\otimes_\Aopp p}}_\ract, N), \quad 
        \hat\psi \mapsto
        \psi := \hat\psi(1,\cdot).
\end{equation}
The inverse map is given by
$$
        \varphi \mapsto \big\{ \hat\varphi : (u^0,\ldots,u^p) \mapsto 
        u^0 \varphi (u^1,\ldots,u^p) \big\}.
$$
Under this isomorphism, the differential $\hat\delta$  
is transformed into 
$$
\gd: C^\bull(U,N) \to C^{\bull+1}(U,N)
$$
given by
\begin{equation}
\label{spaetkauf}
\begin{split}
        \gd\varphi(u^1, \ldots, u^{p+1}) 
&:= u^1 \varphi(u^2, \ldots, u^{p+1}) 
\\ & \quad 
+ \sum^{p}_{i=1} (-1)^i 
        \varphi(u^1, \ldots, u^i u^{i+1}, \ldots, u^{p+1}) \\
& \quad + (-1)^{p+1} 
        \varphi(u^1, \ldots, \varepsilon(u^{p+1}) \blact u^p).
\end{split}
\end{equation}
Observe that by duality, $C^\bull(U,A)$ carries the structure of a
cosimplicial $k$-module. This will be used in Definition~\ref{thomson}
when defining the
associated reduced complex $\bar C^\bull(U,A)$.

Finally, the cup product can be expressed on $C^\bull(U,A)$ 
as follows:
\begin{lemma}
\label{tiefenacht}
The cup product (\ref{xion2}) assumes on 
$\varphi \in C^p(U,A),\psi \in C^q(U,A)$
the form
\begin{equation}
\label{ostkreuz3}
        (\varphi \smallsmile \psi)
        (u^1, \ldots, u^{p+q}) = 
        \varphi\big(u^1, \ldots, u^{p-1}, 
        \psi(u^{p+1}, \ldots, u^{p+q}) \blact u^p\big).
\end{equation}
\end{lemma} 
\begin{proof}
For $U$-linear $\hat\varphi: P_p \to A$ and $\hat\psi: P_q \to A$, the
explicit meaning of \rmref{xion2} is on an element 
$P_n \ni c := (u^0,
\ldots, u^n)$ 
\begin{equation*}
        (\hat\varphi \smallsmile \hat\psi)(c) = 
        \left\{\begin{array}{ll} 
            \hat\varphi(u^0_{(1)}, \ldots, u^p_{(1)})
            \hat\psi(u^0_{(2)} \cdots u^p_{(2)}, u^{p+1}, \ldots, u^n) 
& \mbox{if } p+q = n,
\\ 0 & \mbox{otherwise.}
\end{array} \right.
\end{equation*}
Using the $U$-linearity of the cochains, the Sweedler-Takeuchi property \rmref{tellmemore}, the fact that all
$A$-actions on $U$ commute, and the property of the 
tensor product in question, we obtain 
\begin{equation*}\begin{split}
& \ \hat\varphi(u^0_{(1)}, \ldots, u^p_{(1)})
        \hat\psi(u^0_{(2)} \cdots u^p_{(2)}, u^{p+1}, \ldots, u^n) \\
& \quad = \hat\varphi(u^0_{(1)}, \ldots, u^p_{(1)}) 
        \varepsilon\big(u^0_{(2)} \cdots u^p_{(2)} 
        \bract \hat\psi(1, u^{p+1}, \ldots, u^n)\big) \\
& \quad = \hat\varphi\big(u^0_{(1)} \ract 
        \varepsilon\big(u^0_{(2)} \bract 
        \varepsilon(u^1_{(2)} \bract \cdots \bract 
        \varepsilon(u^p_{(2)})\ldots)\big), u^1_{(1)}, \ldots,  
        \hat\psi(1, u^{p+1}, \ldots, u^{p+q}) \blact u^p_{(1)} \big) \\
& \quad =  \hat\varphi\big(u^0, u^1, \ldots, u^{p-1}, 
        \hat\psi(1, u^{p+1}, \ldots, u^{p+q}) \blact u^p\big).
\end{split}
\end{equation*}
Applying now the isomorphism
\rmref{gesundbrunnen} yields the claim. 
\end{proof} 

In the following, we will mostly be working with this
alternative complex $(C^\bull(U,A),\delta)$ and 
small Greek letters will usually denote cochains 
therein.
%, while we explicitly write $\check \varphi $ for cochains in 
% the original $C^\bull(U,A)$. 
%

\subsection{The comp algebra structure on $C^\bull(U,A)$}
\label{krauslig}
For the construction of the Gerstenhaber bracket, we 
associate to any $p$-cochain
$\varphi \in  C^p(U,A)$ the operator
\begin{equation}
\label{huetor1}
        \DD_\varphi: 
        U^{\otimes_\Aopp p} \to U, \qquad
        (u^1, \ldots, u^p) \mapsto 
        \varphi(u^1_{(1)}, \ldots, u^p_{(1)}) \lact 
        u^1_{(2)} \cdots u^p_{(2)}.
\end{equation}
For zero cochains, i.e., elements in $A$, this becomes
 the map $A \to U, a \mapsto s(a)$, where $s$ is the source map in \rmref{basmati}.

These operators provide the correct substitute 
of the insertion 
operations
used by Gerstenhaber to define what he called 
a \emph{pre-Lie system} in \cite{Ger:TCSOAAR} and
a \emph{(right) comp algebra} in \cite{GerSch:ABQGAAD}. 
Indeed, we can now define, in analogy with \cite{Ger:TCSOAAR},
the {\em Gerstenhaber products} 
$$
        \circ_i : C^p(U,A) \otimes_k C^q(U,A) \rightarrow 
        C^{p+q-1}(U,A),\quad
        i = 1, \ldots, p,
$$
by
\begin{equation}
\label{maxdudler}
\begin{split}
        &\ (\varphi \circ_i \psi)(u^1, \ldots, u^{p+q-1}) 
\\ &\ \quad 
:= \varphi(u^1, \ldots, u^{i-1}, 
        \DD_\psi(u^{i}, \ldots, u^{i+q-1}), 
        u^{i+q}, \ldots, u^{p+q-1}),
\end{split}
\end{equation}
and for zero cochains we define $a \circ_i \psi = 0$ for all $i$ and
all $\psi$, whereas 
$$
\gvf \circ_i a :=  \varphi(u^1, \ldots, u^{i-1}, 
        s(a),u^{i}, \ldots, u^{p-1}).
$$ 
These Gerstenhaber products satisfy the following associativity
relations: 

\begin{lemma}
For  $\varphi \in C^p(U,A)$, 
$\psi \in C^q(U,A)$, and $\chi \in C^r(U,A)$
we have
\begin{equation*}
\label{danton}
(\varphi \circ_i \psi) \circ_j \chi = 
\begin{cases}
(\varphi \circ_j \chi) \circ_{i+r-1} \psi \qquad \mbox{if} \ j < i, \\
\varphi \circ_i (\psi \circ_{j-i +1} \chi) \qquad \mbox{if} \ \, i \leq j < q + i, \\
(\varphi \circ_{j-q+1} \chi) \circ_{i} \psi \qquad \mbox{if} \ j \geq q + i.
\end{cases}
\end{equation*}. 
\end{lemma} 
\begin{proof}
Straightforward computation.
\end{proof}

The structure of a right comp algebra is completed by adding 
the {\em distinguished element} (analogously to \cite[p.~62]{GerSch:ABQGAAD}) 
\begin{equation}
\label{distinguished, I said}
        \mu := \varepsilon \, m_\uhhu \in C^2(U,A),
\end{equation}
where $m_\uhhu$ is the multiplication map of $U$.

\begin{rem}
The associativity of $ m_U $ implies  
$
        \mu \circ_1 \mu = \mu \circ_2 \mu.
$
Furthermore, one has
\begin{equation}
\label{elekta}
\DD_\mu = m_\uhhu,
\end{equation}
as will be used later.
\end{rem}

%In addition, our comp algebra is \emph{strict} 
%as the element $ \varepsilon \in C^1(U,A)$ satisfies
%$$
%        \varepsilon \circ_1 \varphi = \varphi,\quad
%        \varphi \circ_i \varepsilon = \varphi,\quad i=1,\ldots,p 
%$$
%for all $ \varphi \in C^p(U,A)$. 

\begin{rem}
Equivalently, this structure turns 
$O(n):=C^n(U,A)$ into a
nonsymmetric operad
in the category of 
$k$-modules, see e.g.~\cite[\S
5.8.13]{LodVal:AO} or \cite{MarShnSta:OIATAP, Men:BVAACCOHA}, with composition
$$
       O(n) \otimes_k O(i_1) \otimes_k \cdots 
        \otimes_k O(i_n) \rightarrow O(i_1+\cdots+i_n)
$$ 
given by 
$$
        \varphi \otimes_k \psi_1 \otimes_k \cdots 
        \otimes_k \psi_n \mapsto  
        \varphi \bigl(\DD_{\psi_1}(\cdot),\DD_{\psi_2}(\cdot),\ldots,
        \DD_{\psi_n}(\cdot)\bigr).
$$
Together with $ \mu $, the operad $O$ becomes an operad with multiplication
whose unit is $ \mathrm{id}_\ahha \in C^0(U,A).$
\end{rem}
 
\subsection{The Gerstenhaber algebra $H^\bull(U,A)$}
Recall that $|n| = n-1$.

\begin{definition}
For two cochains $\varphi \in C^p(U,A),\psi \in C^q(U,A)$
we define
\begin{equation*}
%\label{zugangskarte}
        \varphi \bar\circ \psi := 
(-1)^{|p||q|}        
\sum^{p}_{i=1}
        (-1)^{|q||i|} \varphi \circ_i \psi \in C^{|p+q|}(U,A)
\end{equation*}
and their \emph{Gerstenhaber bracket} by
\begin{equation}
\label{zugangskarte}
{\{} \varphi,\psi \}
:= \varphi \bar\circ \psi - (-1)^{|p||q|} \psi \bar\circ \varphi.
\end{equation}
  \end{definition}

Furthermore, one verifies by straightforward computation:

\begin{lemma}
\label{platt}
For $\varphi \in C^p(U,A)$ 
and $\psi \in C^q(U,A)$, 
we have   
\begin{equation*}
        \varphi \smallsmile \psi = 
        (\mu \circ_1 \varphi) \circ_{p+1} \psi = 
        (\mu \circ_2 \psi) \circ_1 \varphi
\end{equation*}
and
\begin{equation}
\label{erfurt}
        \delta \varphi = \{\mu,  \varphi\}.
\end{equation}
%the second equation following from \rmref{danton}.
\end{lemma}

We can now state the main theorem of this section (cf.~Theorem
\ref{pitandbull}), which follows from Gerstenhaber's results. First,
let us agree about notation:
\begin{definition}
For a bialgebroid $U$ and every 
$N \in \umod$
we denote the cohomology of $C^\bull(U,N)$ 
by $H^\bull(U,N)$ and call this the \emph{cohomology of $U$ with
  coefficients in $N$}.   
\end{definition}

\begin{rem}\label{hgleichext}
If $U$ is a right $A$-projective left
Hopf algebroid 
so that $P$ is, in view of Lemma~\ref{alterhit}, 
a projective resolution of $A \in \umod$, then 
we have $H^\bull(U,N) \simeq \mathrm{Ext}_U(A,N)$, 
but in general we use the symbol $H^\bull(U,N)$ 
for the cohomology of the explicit cochain complex 
$C^\bull(U,N)$. 
\end{rem}

\begin{theorem}
\label{pitandbull2}
If $U$ is a bialgebroid over $A$,
then the maps \rmref{ostkreuz3} and \rmref{zugangskarte} induce a Gerstenhaber algebra structure on 
$H^\bull(U,A)$.
\end{theorem}

\begin{proof}%[Proof of Theorem~\ref{pitandbull}]
It is a general fact that by using the above formulae 
for $ \delta , \smallsmile $ as definitions,
any right comp algebra 
becomes a DG algebra on whose cohomology 
$\{\cdot,\cdot\}$ induces a Gerstenhaber algebra structure, see 
e.g.\ \cite{GerSch:ABQGAAD, McCSmi:ASODHCC} and the references therein. 
\end{proof}
% \begin{proof}
% We only prove \rmref{ostkreuz}, as \rmref{westkreuz} is a direct
% consequence thereof using \rmref{danton}, and the proof of
% \rmref{suedkreuz} can be---{\em mutatis mutandum}---transferred from
% the proof for associative algebras in 
% \cite[Theorem~3]{Ger:TCSOAAR}. 

% For $U$-linear $\hat\varphi: P_p \to A$ and $\hat\psi: P_q \to A$, the explicit meaning of \rmref{xion2} is on an element $P_n \ni p := (u^0, \ldots, u^n)$
% \begin{equation*}
% (\hat\varphi \smallsmile \hat\psi)(p) = \left\{\begin{array}{lll} 
% \hat\varphi(u^0_{(1)}, \ldots, u^p_{(1)})\hat\psi(u^0_{(2)} \cdots u^p_{(2)}, u^{p+1}, \ldots, u^n) & \mbox{if} & p+q = n,
% \\ 0 & \mbox{else.} &
% \end{array} \right.
% \end{equation*}
% Then, by $U$-linearity, the Takeuchi property, the fact that all $A$-actions on $U$ commute, 
% and in the tensor product in question
% \begin{equation*}
% \begin{split}
% & \ \hat\varphi(u^0_{(1)}, \ldots, u^p_{(1)})\hat\psi(u^0_{(2)} \cdots u^p_{(2)}, u^{p+1}, \ldots, u^n) \\
% & \quad = \hat\varphi(u^0_{(1)}, \ldots, u^p_{(1)}) \varepsilon\big(u^0_{(2)} \cdots u^p_{(2)} \bract \hat\psi(1, u^{p+1}, \ldots, u^n)\big) \\
% & \quad = \hat\varphi\Big(u^0_{(1)} \ract \varepsilon\big(u^0_{(2)} \bract \varepsilon(u^1_{(2)} \bract \ldots \bract \varepsilon(u^p_{(2)})\ldots)\big), u^1, \ldots,  \hat\psi(1, u^{p+1}, \ldots, u^{p+q}) \blact u^p_{(1)} \Big) \\
% & \quad =  \hat\varphi\big(u^0, u^1, \ldots, u^{p-1}, \hat\psi(1, u^{p+1}, \ldots, u^{p+q}) \blact u^p\big),
% \end{split}
% \end{equation*}
% which yields \rmref{ostkreuz} using the isomorphism \rmref{gesundbrunnen}.
% \end{proof}

\begin{rem}
The fact that 
the cup product is graded commutative up to homotopy follows
abstractly using 
the ``Hilton-Eckmann trick'', see, e.g., \cite{Sua:THHAFTACOCP} or 
\cite[Theorem~3]{KowKra:DAPIACT} for the concrete bialgebroid
incarnation. In Gerstenhaber's approach 
it follows from
\begin{equation*}%\label{suedkreuz}
        (-1)^{|q|}\varphi \bar\circ \gd \psi 
        - (-1)^{|q|}\gd(\varphi \bar\circ \psi) 
        + \gd \varphi \bar\circ \psi = 
        \psi \smallsmile \varphi 
        - (-1)^{pq} \varphi \smallsmile \psi, 
\end{equation*}
which means that 
$
\gd(\varphi \bar\circ \psi) = (-1)^{q}\big(\psi \smallsmile \varphi - (-1)^{pq} \varphi \smallsmile \psi \big)
$ 
if $\varphi$ and $\psi$ are cocycles, so their graded commutator is a
coboundary.
\end{rem}

\begin{rem}
If $A$ is commutative and $ \eta $ factorises through the multiplication
map of $A$, that is, if the source and target maps of $U$ coincide so that 
$a \lact u=u \ract a$ holds for all $a \in A,u \in U$, 
then the tensor flip 
$$
        \tau : U \otimes_\ahha U \rightarrow U \otimes_\ahha U,\quad
        u \otimes_\ahha v \mapsto v \otimes_\ahha u
$$ 
is well defined. Consequently, it makes sense to then speak about 
cocommutative left Hopf algebroids, meaning that 
$\tau \circ \Delta = \Delta $. For example, this holds for 
the example of the 
universal enveloping algebra of a Lie-Rinehart algebra, see 
\S\ref{ganzhuebsch}.
In this case an explicit computation shows
that the Gerstenhaber bracket $\{\cdot,\cdot\}$ vanishes which is
clear also for abstract reasons, see \cite{Tai:IHBCOIDHAAGCOTYP}. 
\end{rem}

% \begin{proof}[Proof of Theorem~\ref{pitandbull}]
% The properties of $\{\cdot,\cdot\}$ 
% are verified exactly as it is proven that each operad with
% multiplication is a cosimplicial space whose cohomology is a Gerstenhaber algebra \cite{GerSch:ABQGAAD, GerVor:HGAAMSO, McCSmi:ASODHCC}, observing that for any $\gvf \in C^p(U,A)$ 
% the cosimplicial pieces computing the complex \rmref{gesundbrunnen} are given by
% \begin{equation}
% \label{sandisk}
% \gd_i \varphi = \left\{
% \begin{array}{ll}
% \mu \circ_1 \varphi & \mbox{if} \quad i = 0\\
% \varphi \circ_i \mu & \mbox{if} \quad 1 \leq i \leq p\\
% \mu \circ_2 \varphi & \mbox{if} \quad i = p+1
% \end{array}
%  \right.
% \end{equation}
% as well as
% $$
% \gs_{i-1}\varphi = \varphi \circ_i \id_\ahha \qquad \mbox{for} \ 1 \leq i \leq p, 
% $$
% as can be verified by some bialgebroid computations similar to those in the proof of Lemma~\ref{rheinsberg} as well as the identity $\DD'_\mu = m_\uhhu$, the multiplication in $U$.
% Now $\sum^{p+1}_{i=0} (-1)^i \gd_i = \gd$, where the right hand side is given as in \rmref{spaetkauf}. 
% \end{proof}
% Observe that combining \rmref{zugangskarte} and \rmref{sandisk}, the differential \rmref{spaetkauf} can be expressed in the familiar form
% $$
% \gd \gvf = \{\gvf, \mu\}.
% $$ 

\begin{rem}\label{esgehtvoran}
Before moving on we also quickly remark 
that the reader may find formulae for Gerstenhaber brackets in
the literature that use a slightly different sign convention. Some
confusion that arises from this 
can be avoided by using the notion of the 
\emph{opposite} $(V,\smallsmile_\op,\{\cdot,\cdot\}_\op)$
of a Gerstenhaber algebra $(V,\smallsmile,\{\cdot,\cdot\})$: 
this is defined by
$$
        u \smallsmile_\op v:=v \smallsmile u,\quad
        \{u,v\}_\op := -\{v,u\},
$$ 
and 
it is verified straightforwardly that this indeed is a
Gerstenhaber algebra again. When defining a Gerstenhaber algebra from
a right comp algebra, the same changes can be made on the level of 
the comp algebra itself. The differential then has to be rescaled 
on degree $p$ by a factor of $(-1)^p$ in order to obtain a 
DG algebra again. 
\end{rem}

\section{The Gerstenhaber module}\label{tchaikovskij}
This section introduces the structures on 
homology that correspond to the cup product and the Gerstenhaber
bracket on $H^\bull(U,A)$: 
the cap product between 
$H^\bull(U,A)$ and  
$H_\bull(U,M)$ and then a Hopf algebroid generalisation 
of the Lie derivative that has been defined by Rinehart on 
Lie-Rinehart and Hochschild (co)homology. 
This, for module-comodules $M$ over a left Hopf algebroid $U$,  
will be defined only on ${\hhmu}$ rather than on 
$H_\bull(U,M)$ , and dually it will be necessary to replace 
$H^\bull(U,A)$ by a Gerstenhaber algebra 
$\hmu$ that is the cohomology of a 
suitable comp subalgebra $C^\bull_M(U) \,{\subseteq}\,C^\bull(U,A)$ . 

\subsection{$C_\bull(U,M)$ and 
the cap product}
The first steps in this section are completely dual to those in 
the previous one. First of all, we define the 
homology of a bialgebroid with coefficients in a right module.
The following is the counterpart of Definition~\ref{xion1}:
 
\begin{definition}\label{neon22}
For any bialgebroid $U$ and any 
$M \in \modu$ we define
$$
        \hat C_\bull(U,M):=
        M \otimes_U 
        P_\bull,
%,\quad
%        \bb := \mathrm{id}_M \otimes_U \bb'.
$$
which becomes a chain complex of $k$-modules 
with boundary map 
$\hat \bb := \mathrm{id}_M \otimes_U \bb'$. 
Using the coalgebra structure $\Delta ^\pehhe$ of $P$, we furthermore introduce 
the {\em cap product}
\begin{equation*}
\label{cap}
        \smallfrown \  :
        \hat C^p(U,A) \otimes_k 
        \hat C_n(U,M) \rightarrow 
        \hat C_{n-p}(U,M)
\end{equation*}
by
\begin{equation}
\label{annodom1900}
        \hat\varphi \smallfrown 
       (m \otimes_U c):=
        m \otimes_U c_{(1)} 
        \ract \hat\varphi (c_{(2)}).
\end{equation}
\end{definition}

Analogously to (\ref{gesundbrunnen}), we have an isomorphism of
$k$-modules 
\begin{equation}
\label{kitajgorod}
        \hat{C}_n(U,M) \overset{\simeq}{\longrightarrow} 
        C_n(U,M) =
        M \otimes_\Aopp  U^{\otimes_\Aopp n},
\end{equation}
given by
$$
        m \otimes_U (u^0,\cdots,u^n) \mapsto
        (mu^0,u^1,\ldots,u^n).
$$
Here and in what follows, we are again 
using the notation 
$$(m,u^1,\ldots,u^n):= 
m \otimes_\Aopp u^1 \otimes_\Aopp \cdots \otimes_\Aopp u^n
$$ 
to better
distinguish the tensor product over $\Aop$ from that one over $A$.

\begin{rem}
As a straightforward computation shows, 
the simplicial differential 
$\bb$ from (\ref{scotstoun}) 
differs from the one
induced by 
$\hat \bb$ only by a sign factor: if we suppress the  isomorphism
\rmref{kitajgorod}, then we have on $C_n(U,M)$  $$
        \bb = (-1)^n \hat\bb,
$$
so the two boundary maps yield the same homology $H_\bull(U,M)$.
\end{rem}

\begin{rem}
In analogy with Remark~\ref{hgleichext}, 
if $U$ is a right $A$-projective 
Hopf algebroid, then we have  
$H_\bull(U,M) \simeq \mathrm{Tor}^U(M,A)$.
\end{rem}

Let us compute what happens to the cap product 
under the isomorphisms 
\rmref{gesundbrunnen} and \rmref{kitajgorod}:

\begin{lemma}
The cap product of  
$\varphi \in C^p(U,A)$ with 
$(m, x) \in C_n(U,M)$ is given by
\begin{equation}
\label{alles4}
        \varphi \smallfrown (m,x)  = 
        (m, u^1, \ldots, u^{n-p-1}, 
        \varphi(u^{n-|p|}, \ldots, u^n) \blact u^{n-p}),
\end{equation}
where we again use the abbreviation 
$(m,x)=(m, u^1, \ldots, u^n)$ 
as in Proposition~\ref{kamille}.
\end{lemma} 
\begin{proof}
For $\hat\varphi \in \hat{C}^p(U,A)$ (recall that these are the $U$-linear
cochains), we have by a computation similar to that in the proof of Lemma \ref{tiefenacht}
\begin{footnotesize}
\begin{equation}
\label{groel}
\begin{split}
& \ \hat\varphi \smallfrown \big(m \otimes_U (u^0, \ldots, u^{n})\big) \\ 
& \quad = \hat\varphi 
        (u^0_{(2)} \cdots u^{n-p}_{(2)}, u^{n-|p|},\ldots, u^n) 
        m \otimes_\uhhu  (u^0_{(1)}, \ldots, u^{n-p}_{(1)}) \\
& \quad = \varepsilon \big(u^0_{(2)} \cdots u^{n-p}_{(2)} \bract 
        \hat\varphi(1, u^{n-|p|}, \ldots, u^n)\big) 
        m \otimes_\uhhu  
        \big(u^0_{(1)}, \ldots, u^{n-p}_{(1)}\big) \\
& \quad = m \otimes_\uhhu  \big(
        u^0_{(1)} \ract \varepsilon 
        (u^0_{(2)} \cdots u^{n-p}_{(2)}), \ldots, u^{n-p-1}_{(1)}, 
        \hat\varphi(1, u^{n-|p|}, \ldots, u^n) \blact u^{n-p}_{(1)}\big) \\
& \quad = m \otimes_\uhhu  \big(u^0, \ldots, u^{n-p-1}, 
        \hat\varphi(1, u^{n-|p|}, \ldots, u^n) \blact u^{n-p}\big).
\end{split}
\end{equation}
\end{footnotesize}
The claim follows by applying the isomorphisms (\ref{gesundbrunnen})
and (\ref{kitajgorod}).
\end{proof}

In the sequel we will 
carry out extensive computations concerning
algebraic relations satisfied by the operators 
$$ 
        \iota_\varphi := \varphi \smallfrown \cdot : C_n(U,M)
        \rightarrow C_{n-p}(U,M).
$$
As a first illustration, we formulate 
the following analogue of Corollary~\ref{medtner} in this
notation. This could still be nicely written out using
$\smallfrown$, but the computations 
in the subsequent sections will be
too complex for that.

\begin{prop}
$(C_\bull(U,M),\bb,\smallfrown)$ is a left DG module over 
$(C^\bull(U,A),\delta,\smallsmile)$, i.e.,
\begin{eqnarray}
\label{mulhouse1}
        \iota_\varphi \, \iota_\psi &=&
        \iota_{\varphi \smallsmile \psi},\\
        \label{mulhouse2}
        [\bb, \iota_\varphi] &=& \iota_{\gd\varphi},
\end{eqnarray} 
where $[\cdot,\cdot]$ denotes the 
graded commutator, that is, we
explicitly have for $\varphi \in C^p(U,A)$
$$
        [\bb,\iota_\varphi]=\bb \, \iota_\varphi -
        (-1)^p \iota_\varphi \, \bb,
$$ 
as $\iota_\varphi$ is of degree $p$ while $\bb$ is of degree
$1$.
\end{prop} 

\begin{proof}
This follows instantly from 
the DG coalgebra axioms 
when using the original 
presentation \rmref{annodom1900} 
for the cap product.
\end{proof} 

Consequently, $(H_\bull(U,M),\smallfrown)$ 
is a left module over 
the ring $(H^\bull(U,A),\smallsmile)$.

%\subsection{The cyclic quotient and the operators $\DD_\varphi'$, $\DD^{i^{\rm th}}_\varphi$}

\subsection{The comp module structure on $C_\bull(U,M)$} 
\label{lufthansaschmerz}
A finer analysis, parallel to the one carried out for $C^\bull(U,A)$ in 
\S\ref{krauslig},
shows that $C_\bull(U,M)$ carries 
a structure that we will refer to as that of a  
comp module 
over $C^\bull(U,A)$: 

\begin{definition}
A {\em comp module} over a comp algebra $C^\bull$ is
a sequence of $k$-modules $C_\bull$ together with $k$-linear
operations 
$$
        \bullet_i : 
        C^p \otimes_k C_n \to C_{n-|p|},\quad i = 1, \ldots, n-|p|
$$
satisfying
for $\gvf \in C^p$, $\psi \in C^q$, $y \in
C_n$, and $j = 1, \ldots n-|q|$
\begin{equation}
\label{SchlesischeStr}
\gvf \bullet_i \big(\psi \bullet_j y\big) = 
\begin{cases} 
\psi \bullet_j \big(\gvf \bullet_{i + |q|}  y\big) \quad & \mbox{if} \ j < i \leq n - |p|-|q|, \\
(\gvf \circ_{j-i+1} \psi) \bullet_{i}  y \quad & \mbox{if} \ j - |p| \leq i \leq j, \\
\psi \bullet_{j-|p|} \big(\gvf \bullet_{i}  y\big) \quad & \mbox{if} \ 1 \leq i < j - |p|.
\end{cases}
\end{equation}
 \end{definition}

Of course, the middle line in \rmref{SchlesischeStr} can also be read
from right to left so as to get an idea how an element $\gvf \circ_i
\psi$ acts on $C_\bull$ via $\bullet_i$.

In our case, we define for $i = 1, \ldots, n-|p|$ 
$$
 \bullet_i : 
C^p(U,A) \otimes_k C_n(U, M) \to C_{n-|p|}(U,M)
$$
by
\begin{equation}\label{sehrgeehrter}
  \gvf \bullet_i(m,x) 
:= (m, u^1, \ldots, u^{i-1}, \DD_\varphi(u^i, \ldots, u^{i+|p|}), u^{i+p}, \ldots, u^n).
%\quad i = 1, \ldots, n-|p|,   
  \end{equation}
Observe that for zero cochains, i.e., for elements in $A$, this means that 
$$
a \bullet_i(m,x) := (m, u^1, \ldots, u^{i-1}, s(a), u^{i}, \ldots, u^n), \quad i=1, \ldots, n+1,
$$
where $s$ is the source map from \rmref{basmati}.

One verifies by straightforward computation:
\begin{lemma}\label{uddingston}
The operations \rmref{sehrgeehrter} turn  
$C_\bull(U,M)$ into a comp module over $C^\bull(U,A)$. 

%% Let $\gvf \in C^p(U,A)$, $\psi \in C^q(U,A)$, and $(m,x) \in
%% C_n(U,M)$. Then, for $j = 1, \ldots n-|q|$, one has
%% \begin{equation}
%% \label{SchlesischeStr}
%% \gvf \bullet_i \big(\psi \bullet_j (m,x)\big) = 
%% \begin{cases} 
%% \psi \bullet_j \big(\gvf \bullet_{i + |q|}  (m,x)\big) \quad & \mbox{if} \ j < i \leq n - |p|-|q|, \\
%% (\gvf \circ_{j-i+1} \psi) \bullet_{i}  (m,x) \quad & \mbox{if} \ j - |p| \leq i \leq j, \\
%% \psi \bullet_{j-|p|} \big(\gvf \bullet_{i}  (m,x)\big) \quad & \mbox{if} \ 1 \leq i < j - |p|.
%% \end{cases}
%% \end{equation}
\end{lemma}

\begin{rem}
Despite the similarity, the associativity relations \rmref{SchlesischeStr}
are quite different from those that hold for the $\circ_i$ 
in a comp algebra. For example, 
there seems to be no way  
to express the cap product $\smallfrown$  
in terms of $ \mu $ and $\bullet_i$ 
by a formula analogous to the 
one given in Lemma~\ref{platt} for 
the cup product $\smallsmile$. However,  
Lemma~\ref{platt3} below will provide 
a counterpart of the second part of 
Lemma~\ref{platt}.
\end{rem}

For later use, let us also note down the 
following relations: 
\begin{lemma}
Let $\gvf \in C^p(U,A)$, $\psi \in C^q(U,A)$, and $(m,x) \in C_n(U,M)$. For $i = 1, \ldots, n-|p+q|$ one has
\begin{eqnarray}
\label{schokopudding}
(\gvf \smallsmile \psi) \bullet_i (m,x) &=& \mu \bullet_i \big(\gvf \bullet_i (\psi \bullet_{i+p} (m,x))\big), \\
\label{orecchiette2}
\varphi \bullet_i \big(\psi \smallfrown (m,x)\big) &=& \psi \smallfrown \big(\gvf \bullet_i (m,x)\big).
\end{eqnarray}
\end{lemma}
\begin{proof}
Eq.~\rmref{schokopudding} is easily proven by means of the
Sweedler-Takeuchi property \rmref{tellmemore} and
\rmref{elekta}. Eq.~\rmref{orecchiette2} follows from the fact that
the coproduct of $U$ is an $\Ae$-module homomorphism.
\end{proof}

Similar as for the cap product with a fixed cochain, we introduce 
a new notation for the operator 
$\varphi \bullet_i \cdot$, where 
$ \varphi \in C^p(U,A)$, 
in order to keep the
presentation of our computations below 
as compact as possible:  
whenever $p \leq n$ and for $i = 1, \ldots, n-|p|$, we
define
\begin{equation*}
\label{daygum}
 \DD^{\scriptscriptstyle{i{\rm th}}}_\varphi :
C_n(U, M) \to C_{n-|p|}(U,M), \quad
  (m, x) \mapsto 
\varphi \bullet_i (m,x).
\end{equation*}
%\begin{equation}
%\label{daygum}
%\begin{array}{rcl}
 %\DD^{\scriptscriptstyle{i{\rm th}}}_\varphi : %C_n(U, M) &\to& C_{n-|p|}(U,M), \\
%  (m, x) &\mapsto& (m, u^1, \ldots, u^{i-1}, \%DD_\varphi(u^i, \ldots, u^{i+|p|}), u^{i+p}, \l%dots, u^n).
%\end{array}
%\end{equation}
In particular, we will make frequent use of the short hand
notation  
$$
        \DD'_\gvf := \DD^{\scriptscriptstyle{(n-|p|){\rm th}}}_\varphi.
$$ 
For example, in this notation we have:
\begin{lemma}
For
any $\varphi \in C^p(U,A)$,
for $0 \leq p < n$ 
we have 
on $C_n(U,M)$
\begin{eqnarray}
\label{parkraumbewirtschaftung1}
\dd_0 \, \DD'_\varphi &=& \iota_\gvf, \\
\label{parkraumbewirtschaftung-1}
\dd_{i} \, \DD'_\varphi 
&=& 
\DD'_\varphi \, \dd_{i+|p|}, 
\qquad \mbox{for} \quad i = 2, \ldots, n-|p|, \\
\label{parkraumbewirtschaftung0}
\sss_{j} \, \DD'_\varphi 
&=& 
\DD'_\varphi \, \sss_{j+|p|}, 
\qquad \mbox{for} \quad j = 1, \ldots, n-|p|.
\end{eqnarray}
\end{lemma} 
\begin{proof}
Using \rmref{landliebe}, \rmref{alles4}, and with $\DD_\gvf$ as in \rmref{huetor1}, 
Eq.~\rmref{parkraumbewirtschaftung1} follows directly from the
identity
\begin{equation*}
\label{whiteinch}
\gve \DD_\varphi = \gvf,
\end{equation*}
which we prove now: one verifies 
in a straightforward manner that
$$
\bar\Delta: U^{\otimes_\Aopp p} \to (U^{\otimes_\Aopp p})_\ract \otimes_\ahha \due U \lact {}, \quad (u^1, \ldots, u^p) \mapsto (u^1_{(1)}, \ldots, u^p_{(1)}) \otimes_\ahha u^1_{(2)} \cdots u^p_{(2)}
$$
defines a right $U$-comodule structure on $(U^{\otimes_\Aopp p})_\ract$. Using source and target maps from \rmref{basmati} and denoting by $m_\uhhu$ the multiplication in $U$, we can then write
$$
        \gve \DD_\gvf 
        = 
        \gve m_\uhhu\big(s \gvf \otimes 
        \id\big)\bar\Delta 
        = 
        \gve m_\uhhu\big(s \gvf \otimes 
        s \gve \big)\bar\Delta 
        =  
        (\gvf \otimes \gve)\bar\Delta 
        = 
        \gvf \mUopp (\id \otimes t \gve)\bar\Delta  
        = 
        \gvf,
$$
which holds by $A$-linearity of a bialgebroid counit, the right $A$-linearity of $\gvf$ and the fact that $\bar\Delta$ is a coaction.

Eqs.~\rmref{parkraumbewirtschaftung-1} and \rmref{parkraumbewirtschaftung0} follow by straightforward computation, 
using the fact that the involved 
face and degeneracy maps 
can be written as
\begin{equation*}
\begin{array}{rl}
\begin{array}{rcl}
\dd_i(m,x) &=& \mu \bullet_{n-i} (m,x), \\
\sss_j(m,x)&=& (\gve 1_\uhhu) \bullet_{n-|i|} (m,x), 
\end{array}
& \mbox{for} \ i, j = 1, \ldots, n-1,
\end{array}
\end{equation*}
where $(m,x) \in C_\bull(U,M)$,
and then applying the properties \rmref{SchlesischeStr}.
\end{proof} 

\subsection{The comp algebra $C^\bull_M(U)$}
When
$U$ is a left Hopf
algebroid (not just a bialgebroid as before) 
and $M$ is a module-comodule,
the para-cyclic structure on $C_\bull(U,M)$ 
given in Proposition~\ref{kamille}
relates the products $\bullet_i$ 
to each other:

\begin{lemma}
\label{precedinglemma}
For any $\varphi \in C^p(U,A)$,
we have for $0 \leq p \leq n$ and $(m, x) \in C_\bull(U,M)$
\begin{equation}
%\label{lallamada}
\label{naumburg}
\gvf \bullet_{i} \big(\ttt(m,x)\big) = \begin{cases}
 \ttt \big( \gvf \bullet_{i+1} (m,x)\big) & \mbox{for} \ i = 1, \ldots, n-p, \\
\ttt \big( \iota_\varphi \, \sss_{-1}(m,x)\big)     & \mbox{for} \ i = n-|p|.
\end{cases}
\end{equation}
\end{lemma}
\begin{proof}
The case for $1 \leq i \leq n-p$ is a simple computation using \rmref{Sch38} and \rmref{Sch5}:
\begin{align*}
& \ 
\gvf \bullet_{i} \big( \ttt  (m, u^1, \ldots, u^n)\big) \\
=& \ 
\big(m_{(0)}u^1_+, u^2_+, \ldots, u^{i}_+, \DD_\gvf(u^{i+1}_+,\ldots,  u^{i+p}), u^{i+p+1}_+, \ldots, u^n_+, u^n_-  \cdots u^1_- m_{(-1)}\big) \\
=& \ 
\big(m_{(0)}u^1_+, u^2_+, \ldots, \big(\DD_\gvf(u^{i+1},\ldots,
  u^{i+p})\big)_+, \ldots \\
& \ \ldots , u^n_+, u^n_- \cdots \big(\DD_\gvf(u^{i+1},\ldots, u^{i+p})\big)_- \cdots u^1_- m_{(-1)}\big) \\
=& \ 
\ttt \big( \gvf \bullet_{i+1} (m, u^1, \ldots, u^n)\big). %\qedhere
\end{align*}
As for the case $i = n-|p|$, one first observes that no
aYD condition (i.e., compatibility of $U$-action and $U$-coaction) is
needed for the explicit computation, which we leave to the
reader.
\end{proof}

% In our alternative notation for 
% $\bullet_i$, the above reads
% $$
% \ttt \DD^{\scriptscriptstyle{(i+1){\rm th}}}_\varphi
% = \DD^{\scriptscriptstyle{i{\rm th}}}_\varphi
% \ttt.
% $$

The comp module structure of $C_\bull(U,M)$ 
does not descend, for general module-comodules $M$ 
over left Hopf algebroids, to the universal cyclic quotient 
$C_\bull^\cyc(U,M)$. Since we will have to work from some point on on
the latter, we define:

\begin{definition}
If $U$ is a left Hopf algebroid and 
$M$ is a module-comodule, we define
$$
        C^\bull_M(U):=\big\{ \varphi \in C^\bull(U,A) \mid 
        \DD^{\scriptscriptstyle{i{\rm th}}}_\varphi (\mathrm{im}(\mathrm{id} - 
        \TT) ) \, {\subseteq} \, \mathrm{im}(\mathrm{id} - \TT) \forall i\big\}.
$$  
\end{definition}

Obviously, one has 
$C^\bull_M(U)=C^\bull(U,A)$ whenever $M$
is an SaYD module.
Observe furthermore 
that the middle relation in \rmref{SchlesischeStr} 
immediately implies:

\begin{lemma}
$C^\bull_M(U) \,{\subseteq}\,
C^\bull(U,A)$ is a comp subalgebra. 
\end{lemma} 
  
In particular, it is a DG
subalgebra, so it makes sense to talk about its cohomology:
\begin{definition}
The cohomology of $C^\bull_M(U)$ will
be denoted by $\hmu$.
  \end{definition}

Applying Eq.~\rmref{naumburg} repeatedly, one obtains that on 
$C^\cyc_\bull(U,M)$ all operators $\DD^{\scriptscriptstyle{i{\rm
      th}}}_\varphi$ can be expressed in terms of 
$\DD'_\varphi$ and the cyclic operator. More precisely, 
Lemma~\ref{uddingston} 
respectively Eq.~\rmref{orecchiette2} imply:

\begin{lemma}
If $M$ is a module-comodule over a left Hopf algebroid $U$, then
for any $\varphi \in C^p_M(U)$ and 
$\psi \in C^q_M(U)$ we have
\begin{equation}
\label{orecchiette1}
        \DD^{\scriptscriptstyle{i{\rm th}}}_\varphi = 
        \ttt^{n-|p|-i} \DD'_\varphi \ttt^{i+p}, 
        \qquad i = 1, \ldots, n-|p|, 
\end{equation}
and
\begin{equation}
\label{orecchiette3}
\ttt^{n-|p+q|-i} \DD'_\varphi \ttt^{i+p} \iota_\psi = \iota_\psi \ttt^{n-|p|-i} \DD'_\varphi \ttt^{i+p}
\end{equation}
as operators 
on $C^\cyc_\bull(U,M)$.
\end{lemma}

We conclude this subsection with another technical lemma:

\begin{lemma}
Let $M$ be a module-comodule over a left Hopf algebroid $U$ and
$\varphi \in C^p_M(U)$ as well as 
$\psi \in C^q_M(U)$.
\begin{enumerate}
\item
If $\psi$ is a cocycle, then the equation
\begin{equation}
\label{cabras} 
\dd_1 \DD'_\psi = \sum^q_{i=1} (-1)^{i+q} \DD'_\psi \dd_i + (-1)^q \dd_1 \ttt \DD'_\psi \ttt^n
\end{equation}
holds for $0 < q < n$
on $C_n^\cyc(U,M)$.
\item
For $0 \leq p \leq n$, the identities
\begin{equation}
\label{parkraumbewirtschaftung2b}
\DD'_\varphi = \ttt \, \iota_\varphi \, \sss_{-1} \, \ttt^n
\end{equation}
and
\begin{equation}
\label{parkraumbewirtschaftung2a}
\iota_\gvf \, \sss_{-1} 
= \ttt^{n-|p|} \, \DD'_\varphi \, \ttt
\end{equation}
hold on $C^\cyc_\bull(U,M)$.
\end{enumerate}
\end{lemma}

\begin{proof}
All statements are either obvious or follow by a straightforward
computation. For example, \rmref{cabras} is proven with the help of
\rmref{orecchiette1} and
\rmref{spaetkauf}. Eqs.~\rmref{parkraumbewirtschaftung2b} and
\rmref{parkraumbewirtschaftung2a} follow directly from
\rmref{naumburg} as 
we have $\mathrm{id}-\TT=0$ on $C^\cyc_\bull(U,M)$.
\end{proof}

\subsection{The Lie derivative}
\label{nmd}
Now we
define a Hopf algebroid generalisation of the
\emph{Lie derivative}
that will subsequently be shown to define 
a Gerstenhaber module structure on 
${\hhmu}$. Throughout, $U$ 
is a left Hopf algebroid and $M$ is a module-comodule.

\begin{definition}
For $\varphi \in C^p(U,A)$, 
we define 
$$
        \lie_\varphi : C_n(U,M) \rightarrow 
        C_{n-|p|}(U,M)
$$
in degree $n$ with $p < n+1$ to be 
%
% \begin{equation}
% \label{messagedenoelauxenfantsdefrance1}
% \lie_\varphi := \sum^{n-p+1}_{i=0} (-1)^? \ttt^{n - i - p + 1} \, \DD'_\varphi \, \ttt^{i+p} + \sum^{p-1}_{i=1} (-1)^? \ttt^{n-p+1} \,  \DD'_\varphi \, \ttt^i,
% \end{equation}
% which we re-sum as 
%
\begin{equation}
\label{messagedenoelauxenfantsdefrance2}
        \lie_\varphi := 
        \sum^{n-|p|}_{i=1} 
        (-1)^{\theta^{n,p}_{i}} 
        \ttt^{n - |p| - i} \, \DD'_\varphi \, \ttt^{i+p} 
        + 
        \sum^{p}_{i=1} 
        (-1)^{\xi^{n,p}_{i}} 
        \ttt^{n-|p|} \, \DD'_\varphi \, \ttt^i,
\end{equation}
where the signs are given by
\begin{equation*}
\label{nerv2}
        \theta^{n,p}_{i} 
        := |p|(n-|i|), 
        \qquad \xi^{n,p}_{i} 
        := n|i| + |p|.
\end{equation*}
In case $p = n+1$, we set
$$
        \lie_\varphi := 
        (-1)^{|p|}\iota_\varphi \, \BB,
$$
and for $p > n+1$, we define 
$\lie_\varphi := 0$.
\end{definition}

We will speak of the first sum in the Lie derivative as of the {\em untwisted part}
and of the second sum as of the {\em twisted part}, a terminology which will 
become vivid in \S\ref{alletklaa}.

Clearly, 
$ \lie_\varphi $ descends for 
$ \varphi \in C^\bull_M(U)$ 
to a well-defined operator on 
$C^\cyc_\bull(U,M)$.
In particular, this applies to the 
distinguished element $ \mu $ 
from \rmref{distinguished, I
said}. For this specific cochain, we obtain the following
counterpart to the second half of Lemma~\ref{platt}:

\begin{lemma}\label{platt3}
The differential of $C^\cyc_\bull(U,M)$ is given by
\begin{equation}
\label{sachengibts}
        \bb = -\lie_\mu. 
\end{equation}
\end{lemma}
\begin{proof}
Using \rmref{elekta}, one obtains $\DD'_\mu = \dd_1$ 
and correspondingly for the Lie derivative by the relations of a para-cyclic module:
\begin{equation*}
\begin{split}
\lie_\mu &= \sum^{n-1}_{i=1} (-1)^{n-i+1} \ttt^{n - 1 - i} \, \dd_1 \, \ttt^{i+2} 
+ \sum^{2}_{i=1} (-1)^{n(i-1)+1} \ttt^{n-1} \, \dd_1 \, \ttt^i, \\
&= \sum^{n-1}_{i=1} (-1)^{n-i+1} \dd_{n-i} \, \ttt^{n+1} - \dd_n \, \ttt^n + (-1)^{n+1} \dd_n \ttt^{n+1} \\
&= \sum^{n-1}_{j=1} (-1)^{j+1} \dd_{j} \, \ttt^{n+1} - \dd_0 \, \ttt^{n+1} + (-1)^{n+1} \dd_n \ttt^{n+1} = - \bb
\end{split}
\end{equation*}
on the quotient $C^\cyc_\bull(U,M)$.
\end{proof}

\subsection{The case of 1-cochains}
\label{alletklaa}
For the reader's convenience, 
we treat some special cases in
detail that will help understanding the 
general
formula for $\lie_\varphi $ and how it has been derived.

First of all, consider a 1-cochain $ \varphi \in C^1(U,A)$. 
By extending scalars from $k$ to the ring 
$k \Lbrack r \Rbrack$ of formal power series
in an indeterminate $r$, we 
define for any $k\Lbrack r \Rbrack$-linear map 
$$
        \DD : C_n(U,M)\Lbrack r \Rbrack \rightarrow C_n(U,M)\Lbrack r
        \Rbrack
$$
the operators
$$
        \ttt^\DD := \DD \, \ttt,\quad
        \TT^\DD := 
        (\ttt^\DD)^{n+1}.
$$ 
We apply this with $\DD$ being the exponential series
$$
        \exp(r\varphi) := 
        \sum_{i \geq 0} \frac{1}{i!} 
        (r\DD'_\varphi)^i. 
$$
Thinking of a 1-cocycle $ \varphi $ 
as of a generalised vector field, of 
$\exp(r\varphi)$ as of its flow, and of
$$
        \Omega_\varphi := \mathrm{id}-\TT^{\exp(r\varphi)} 
$$
as of a curvature along an integral curve motivates the fact that  
a short computation yields
\begin{equation*}
        \lie_\varphi = 
            {\textstyle{\frac{d}{dr}}}
            \Omega_\varphi |_{r=0} 
\end{equation*}
 for $n>0$, which in this case is explicitly given by
\begin{equation*}
\label{nakissa}
\lie_\varphi = \sum^n_{i=0} \ttt^{n-i} \, \DD'_\varphi \, \ttt^{i+1} = \sum^n_{i=1} \ttt^{n-i} \, \DD'_\varphi \, \ttt^{i+1} 
+ \ttt^{n} \, \DD'_\varphi \, \ttt.
\end{equation*}

Next, let us study $\lie_\varphi $ in greater detail on $C_\bull^\cyc(U,M)$.
Note first that, when descending to the quotient 
$C_n^\cyc(U,M)$, 
the untwisted part in \rmref{messagedenoelauxenfantsdefrance2} can be written as follows:
\begin{equation*}
\begin{split}
\sum^{n-|p|}_{i=1} (-1)^{\theta^{n,p}_{i}} \ttt^{n - |p| - i} \, \DD'_\varphi \, \ttt^{i+p} (m,x)  
&= \sum^{n-|p|}_{i=1} (-1)^{\theta^{n,p}_{i}}  \DD^{\scriptscriptstyle{i{\rm th}}}_\gvf (m,x) \\
&= \sum^{n-|p|}_{i=1} (-1)^{\theta^{n,p}_{i}} \gvf \bullet_i (m,x).
\end{split}
\end{equation*}
If we now introduce the operator
\begin{equation*}
\label{huetor2}
        \EE_\varphi: U \to U, \quad 
        u \mapsto \varphi(u_-) \blact u_+,
\end{equation*}
then $\lie_\varphi $ can be further rewritten as follows:
\begin{prop}
For every module-comodule $M$ over a left Hopf algebroid $U$, the Lie
derivative $\lie_\varphi $ for $\varphi \in C^1_M(U)$ 
assumes on $C_n^\cyc(U,M)$ the form 
\begin{equation}
\label{ramblas}
\begin{split}
\lie_\varphi(m, x) &= \sum_{i=1}^n \big(m,u^1, \ldots,
\DD_\varphi(u^i), \ldots, u^n\big) \\
&\quad + \big(m_{(0)}, u^1_+, \ldots, u^{n-1}_+, \varphi(u^n_- \cdots u^1_- m_{(-1)}) \blact u^n_+\big).
\end{split}
\end{equation}
This can be alternatively written as
\begin{footnotesize}
\begin{equation}
\label{lefevrerand}
\begin{split}
\lie_\varphi(m, x) &= \big(\varphi(m_{(-1)}) m_{(0)}, u^1, \ldots, u^n \big) 
+ \sum_{i=1}^n \big(m, u^1, \ldots, \DD_\varphi(u^i), \ldots, u^n\big) \\
& \qquad + \sum_{j=1}^n \big(m, u^1, \ldots, \EE_\varphi(u^j), \ldots, u^n\big) \\
&\qquad - \sum_{k=1}^{n} \big(m_{(0)}, u^1_+, \ldots, u^{k-1}_+, \gd\varphi(u^k_-, u^{k-1}_- \cdots u^1_- m_{(-1)}) \blact u^k_+, 
u^{k+1}, \ldots, u^n\big). 
\end{split}
\end{equation}
\end{footnotesize}
\end{prop}
\begin{proof}
The explicit form for the untwisted part of $\lie$, i.e., the first summand in \rmref{ramblas} was explained above, 
whereas the twisted part follows by a straightforward computation using the powers of $\ttt$ in Lemma \ref{kohinor}. Eq.~\rmref{lefevrerand} follows by using Eq.~\rmref{spaetkauf} for $p = 1$
% , i.e.,
% $$
% \gd\varphi(u, v) = \varepsilon\big(\varphi(v) \blact u\big) - \varphi(uv) + \varphi\big(\varepsilon(v) \blact u \big).
% $$
as well as \rmref{Sch3} and \rmref{Sch48}.
\end{proof}

\begin{example}
In degree $n=1$, the above reads 
\begin{equation*}
\begin{split}
\lie_\varphi(m, u) &= \big(\varphi(m_{(-1)}) m_{(0)}, u\big) + \big(m, \varphi(u_{(1)}) \lact  u_{(2)}\big) +
\big(m, \varphi(u_-) \blact u_+\big) \\
& \quad - \big(m_{(0)}, \gd\varphi(u_-, m_{(-1)}) \blact u_+\big),
\end{split}
\end{equation*}
and in degree $n=2$ it becomes
\begin{equation*}
\begin{split}
\lie_\varphi(m, u, v) &= \big(\varphi(m_{(-1)}) m_{(0)}, u, v\big) + \big(m, \varphi(u_{(1)}) \lact  u_{(2)}, v\big) 
+ \big(m,  u, \varphi(v_{(1)}) \lact  v_{(2)}\big) \\
&\quad 
+ \big(m, \varphi(u_-) \blact u_+, v\big) + \big(m, u, \varphi(v_-) \blact v_+\big) \\
&\quad 
- \big(m_{(0)}, \gd\varphi(u_-, m_{(-1)}) \blact u_+, v\big) 
-  \big(m_{(0)}, u_+, \gd\varphi(v_-, u_- m_{(-1)}) \blact v_+ \big).
\end{split}
\end{equation*}
\end{example}

\begin{example}
In case $\varphi$ is a $1$-co{\em cycle}, one has the cocycle condition
\begin{equation}
\label{marcello}
\varphi(uv) = \varepsilon\big(\varphi(v) \blact u\big) + \varphi\big(\varepsilon(v) \blact u \big),
\end{equation}
which implies 
$
\varphi(1) = 0.
$
The Lie derivative in degree zero then reads, as before
$$
\lie_\varphi(m) = \varphi(m_{(-1)}) m_{(0)} = \varphi(m_{(-1)}) \blact m_{(0)} , 
$$
whereas in degree $n$ reduces to
\begin{equation}
\label{lefevre}
\begin{split}
\lie_\varphi(m, x) &= \big(\varphi(m_{(-1)}) m_{(0)}, u^1, \ldots, u^n\big) \\
&\quad + \sum_{i=1}^n \big(m, u^1, \ldots, \DD_\varphi(u^i), \ldots, u^n\big) \\
&\quad + \sum_{j=1}^n \big(m, u^1, \ldots, \EE_\varphi(u^j), \ldots, u^n\big).
\end{split}
\end{equation}
In particular, in degree one this reads 
$$
\lie_\varphi(m, u) = \big(\varphi(m_{(-1)}) m_{(0)}, u\big) + \big(m, \varphi(u_{(1)}) \lact  u_{(2)}\big) +
\big(m, \varphi(u_-) \blact u_+\big).
$$
Observe that in \rmref{lefevre} the single 
summands where the $\EE_\varphi$ appear are not well-defined but only their sum is (a similar comment applies to \rmref{lefevrerand}). 
To exemplify this, consider in degree $2$ the map
$$
(u, v) \mapsto \big(\EE_\varphi(u), v\big) + \big(u, \EE_\varphi(v) \big).
$$
Using \rmref{marcello} and \rmref{Sch5}, one has
$$
(a \blact u, v) \mapsto \big(\EE_\varphi(u), v \ract a \big) + \big(u,  v \ract \varphi(s(a))\big) + \big(a \blact u, \EE_\varphi(v)\big),
$$
and it is easy to see that $(u, v \ract a)$ has the same image.\end{example}

\subsection{The case of an SaYD module}

In the case of stable anti Yetter-Drinfel'd modules, one can find an expression for $\lie_\varphi $ 
on $C^\cyc_\bull(U,M)$ analogous to the one given in 
\rmref{ramblas} for the special case of 
$1$-cochains. This is achieved by the following result:
\begin{prop}
If $M$ is an SaYD module and 
$\gvf \in C^p_M(U)$, one has
on $C^\cyc_\bull(U,M)$
\begin{footnotesize}
\begin{equation*}
\label{encarnita}
\begin{split}
&\lie_\gvf(m, x) =\\
& \quad \sum_{i=1}^{n-|p|} (-1)^{\theta^{n,p}_{i}} 
\big(m,u^1,\ldots, \DD_\gvf(u^i,\ldots, u^{i+|p|}), \ldots, u^n\big) \\
%&\quad  + (-1)^? (m_{(0)}, u^1_+, \ldots, u^{n-p}_+, \psi(u^{n-p+2}_+, \ldots, u^n_+, u^n_- \cdots u^1_- m_{(-1)}) \blact u^{n-p+1}_+), \\
&\quad  + \sum^{|p|}_{i=0} (-1)^{\xi^{n,p}_{i+1}} \big(m_{(0)}u^1_{+(2)} \cdots u^{i}_{+(2)}, u^{i+1}_+, \ldots, u^{n-p+i}_+, \\
& \qquad \qquad \qquad \gvf(u^{n-|p|+i+1}_+, \ldots, u^n_+, u^n_- \cdots u^1_- m_{(-1)}, u^1_{+(1)}, \ldots, u^{i}_{+(1)}) \blact u^{n-|p|+i}_+\big). 
\end{split}
\end{equation*}
\end{footnotesize}
\end{prop}
\begin{proof}
Straightforward computation using Lemma \ref{kohinor} as well as Schauenburg's relations \rmref{Sch3}--\rmref{Sch5}, the fact that the two $\Ae$-module structures originating from the $U$-action and $U$-coaction coincide for SaYD modules, and 
the Sweedler-Takeuchi condition \rmref{auchnochnicht} for comodules.
\end{proof}

\begin{example}
For $p =2$ and $n=3$, this reads:
\begin{equation*}
\begin{split}
\lie_\gvf(m, u, v, w) &= - \big(m,\DD_\gvf(u,v), w \big) +  \big(m, u, \DD_\gvf(v,w) \big) \\ 
& \quad - \big(m_{(0)}, u_+, \gvf(w_+, w_-v_-u_-m_{(-1)}) \blact v_+\big) \\ 
& \quad + \big(m_{(0)}u_{+(2)}, v_+, \gvf(w_-v_-u_-m_{(-1)}, u_{+(1)}) \blact w_+\big). 
\end{split}
\end{equation*}
\end{example}

\subsection{The DG Lie algebra module structure}

We now prove that 
the Lie derivative $\lie$ defines a DG 
Lie algebra representation of $(C^\bull_M(U)[1], \{.,.\})$: 

\begin{theorem}
\label{feinefuellhaltertinte}
For any two cochains $\varphi \in C_M^p(U)$ 
and $\psi \in C_M^q(U)$, 
we have on the quotient $C^\cyc_\bull(U,M)$
\begin{equation}
\label{weimar}
[\lie_\varphi, \lie_\psi] = \lie_{\{\varphi, \psi\}},
\end{equation}
where the bracket on the right hand side is the Gerstenhaber bracket
\rmref{zugangskarte}.
Furthermore, we have
\begin{equation}
\label{alles1}
[\bb, \lie_\varphi] + \lie_{\gd\varphi}= 0. 
\end{equation}
\end{theorem}

\begin{proof}
The proof relies 
on Eqs.~\rmref{parkraumbewirtschaftung0}, \rmref{parkraumbewirtschaftung2a}, and \rmref{orecchiette1}: 
assume that $1 \leq q \leq p$ and 
$p+q \leq n+1$, as the proof for zero cochains and the case $q = 0$,
$p = n+1$ can be carried out by similar, but easier
computations. Recall that throughout we consider the operators induced
on $C^\cyc_\bull(U,M)$ and hence may identify 
$\TT$ and $\mathrm{id} $.

Using \rmref{messagedenoelauxenfantsdefrance2}, we
explicitly compute the expressions for
$\lie_\varphi \lie_\psi $ and $\lie_\psi
\lie_\varphi $. The underbraced terms will afterwards be computed and
compared one
by one.
One has
\begin{small}
\begin{equation*}
\begin{split}
\lie_\varphi \lie_\psi 
&= \ubs{(1)}{\sum^{n-|p|-|q|}_{i=1} \, \sum^{n-|q|}_{j=1} (-1)^{\theta^{n-|q|,p}_i + \theta^{n,q}_j} \ttt^{n-|p|-|q|-i} \DD'_\varphi \ttt^{n-|q|+p+i-j} \DD'_\psi \ttt^{j+q}} \\ 
& \quad + \ubs{(2)}{\sum^{p}_{i=1} \, \sum^{n-|q|}_{j=1} (-1)^{\xi^{n-|q|,p}_i + \theta^{n,q}_j} \ttt^{n-|p|-|q|} \DD'_\varphi \ttt^{n-|q|+i-j} \DD'_\psi \ttt^{j+q}}  \\
&\quad + \ubs{(3)}{\sum^{n-|p|-|q|}_{i=1} \, \sum^{q}_{j=1} (-1)^{\theta^{n-|q|,p}_i + \xi^{n,q}_j} \ttt^{n-|p|-|q|-i} \DD'_\varphi \ttt^{n-|q|+p+i} \DD'_\psi \ttt^{j}}  \\
& \quad + \ubs{(4)}{\sum^{p}_{i=1} \, \sum^{q}_{j=1} (-1)^{\xi^{n-|q|,p}_i + \xi^{n,q}_j} \ttt^{n-|p|-|q|} \DD'_\varphi \ttt^{n-|q|+i} \DD'_\psi \ttt^{j}},  
\end{split}
\end{equation*}
\end{small}
along with
\begin{small}
\begin{equation*}
\begin{split}
-(-1)^{|p||q|} \lie_\psi \lie_\varphi 
&= \ubs{(5)}{\sum^{n-|p|-|q|}_{j=1} \, \sum^{n-|p|}_{i=1} (-1)^{\theta^{n,q}_j + \theta^{n,p}_i +1} \ttt^{n-|q|-|p|-j} \DD'_\psi \ttt^{n-|p|+q+j-i} \DD'_\gvf \ttt^{i+p}} \\ 
& \quad + \ubs{(6)}{\sum^{q}_{j=1} \, \sum^{n-|p|}_{i=1} (-1)^{\xi^{n-|p|,q}_j + \theta^{n-|q|,p}_i +1}  \ttt^{n-|q|-|p|} \DD'_\psi \ttt^{n-|p|+j-i} \DD'_\gvf \ttt^{i+p}}  \\
&\quad + \ubs{(7)}{\sum^{n-|q|-|p|}_{j=1} \, \sum^{p}_{i=1} (-1)^{\theta^{n,q}_j + \xi^{n,p}_i +1}  \ttt^{n-|q|-|p|-j} \DD'_\psi \ttt^{n-|p|+q+j} \DD'_\gvf \ttt^{i}}  \\
& \quad + \ubs{(8)}{\sum^{q}_{j=1} \, \sum^{p}_{i=1} (-1)^{\xi^{n-|p|,q}_j + \xi^{n,p}_i +1 + |p||q|}  \ttt^{n-|q|-|p|} \DD'_\psi \ttt^{n-|p|+j} \DD'_\gvf \ttt^{i}}.  
\end{split}
\end{equation*}
\end{small}
Furthermore, it follows 
from \rmref{SchlesischeStr} and \rmref{orecchiette3} 
that, for $i = 1, \ldots, p$, 
we have 
the identities
\begin{equation*}
%\begin{split}
\DD'_{\gvf \circ_i \psi} = \DD'_\gvf \DD^{\scriptscriptstyle{k{\rm th}}}_\psi =  \DD'_\gvf \ttt^{n-|q|-k} \DD'_\psi \ttt^{k+q}, \quad \mbox{where} \quad k = n-|p|-|q|, \ldots, n-|q|.
%\end{split}
\end{equation*}
Hence
\begin{small}
\begin{equation*}
\begin{split}
\lie_{\varphi \bar\circ \psi} 
&= (-1)^{|p||q|} {\sum^{n-|p|-|q|}_{i=1} \, \sum^{n-|q|}_{k= n -|p|-|q|} (-1)^{\theta^{n,|p+q|}_{i} + |q||k-n+|p+q||} \ttt^{n-|p|-|q|-i} \DD'_\varphi \DD^{\scriptscriptstyle{k{\rm th}}}_\psi \ttt^{i+|p+q|}} \\ 
& \quad + (-1)^{|p||q|}{\sum^{|p+q|}_{i=1} \, \sum^{n-|q|}_{k= n -|p|-|q|} (-1)^{\xi^{n,|p+q|}_{i} + |q||k-n+|p+q||} \ttt^{n-|p|-|q|} \DD'_\varphi \DD^{\scriptscriptstyle{k{\rm th}}}_\psi \ttt^{i+|p+q|}} \\ 
&= {\sum^{n-|p|-|q|}_{i=1} \, \sum^{n-|q|}_{k= n -|p|-|q|} (-1)^{\theta^{n,p}_{i} + \theta^{n,q}_i + |q|(|k|-n)} \ttt^{n-|p|-|q|-i} \DD'_\varphi \ttt^{n-|q|-k} \DD'_\psi \ttt^{k+q+i+|p+q|}} \\ 
& \quad + {\sum^{|p+q|}_{i=1} \, \sum^{n-|q|}_{k= n -|p|-|q|} 
(-1)^{\xi^{n,|p+q|}_{i} + |q|(|k|-n)} \ttt^{n-|p|-|q|} \DD'_\varphi \ttt^{n-|q|-k} \DD'_\psi \ttt^{k+q+i}} \\
&= \ubs{(9)}{\sum^{n-|p|-|q|}_{i=1} \, \sum^{|p|+i}_{l=i} (-1)^{\theta^{n,p}_{i} + \theta^{n,q}_i + |q|(l+i+|p|)} \ttt^{n-|p|-|q|-i} \DD'_\varphi \ttt^{|p|+i-l} t^{n-|q|+1} \DD'_\psi \ttt^{l+q}} \\ 
& \quad + \ubs{(10)}{\sum^{|p+q|}_{i=1} \, \sum^{n+1}_{l = n -|p|+1} (-1)^{\xi^{n,|p+q|}_{i} + \theta^{n,q}_l} \ttt^{n-|p|-|q|} \DD'_\varphi \ttt^{n-|l|} \DD'_\psi \ttt^{l+i}},
\end{split}
\end{equation*}
\end{small}
where we substituted $l:=k-n+|p|+|q|+i$ in the first summand of the last equation, $l:= k +q$ in the second summand, and used the fact that we descend to the quotient $C^\cyc_\bull(U,M)$.
Now it is easy to see that
$$
(9) = \sum^{n-|p|-|q|}_{i=1} \, \sum^{|p|+i}_{l=i} (-1)^{\theta^{n,p}_{i} + \theta^{n,q}_l + |q||p|}  \DD^{\scriptscriptstyle{i{\rm th}}}_\gvf  \DD^{\scriptscriptstyle{l{\rm th}}}_\psi.
$$
Likewise,
\begin{small}
\begin{equation*}
\begin{split}
-(-1)^{|p||q|} \lie_{\psi \bar\circ \gvf} &= \ubs{(11)}{ \sum^{n-|q|-|p|}_{j=1} \, \sum^{|q|+i}_{l=i} (-1)^{\theta^{n,q}_{j} + \theta^{n,p}_{l} + 1}  \DD^{\scriptscriptstyle{j{\rm th}}}_\psi  \DD^{\scriptscriptstyle{l{\rm th}}}_\gvf}  \\ 
& \quad + \ubs{(12)}{\sum^{|q+p|}_{j=1} \, \sum^{n+1}_{l= n -|q| +1} (-1)^{\xi^{n, |q+p|}_{j} + \theta^{n,p}_l + |p||q| + 1} \ttt^{n-|q|-|p|} \DD'_\psi \ttt^{n-|l|} \DD'_\gvf \ttt^{l+j}}.
\end{split}
\end{equation*}
\end{small}
We can now write on the quotient $C^\cyc_\bull(U,M)$
\begin{small}
\begin{equation*}
\begin{split}
(1) &= {\sum^{n-|p|-|q|}_{i=1} \, \sum^{n-|q|}_{j=1} (-1)^{\theta^{n-|q|,p}_{i} + \theta^{n,q}_j} \DD^{\scriptscriptstyle{i{\rm th}}}_\gvf \DD^{\scriptscriptstyle{j{\rm th}}}_\psi} \\
&=  \ubs{(13)}{\sum^{n-p-|q|}_{j=1} \, \sum^{n-|p|-|q|}_{i=j+1} (-1)^{\theta^{n-|q|,p}_{i} + \theta^{n,q}_j} \DD^{\scriptscriptstyle{i{\rm th}}}_\gvf \DD^{\scriptscriptstyle{j{\rm th}}}_\psi} 
+ \ubs{(14)}{\sum^{n-|q|}_{j=p+1} \, \sum^{j-p}_{i=1} (-1)^{\theta^{n-|q|,p}_{i} + \theta^{n,q}_j} \DD^{\scriptscriptstyle{i{\rm th}}}_\gvf \DD^{\scriptscriptstyle{j{\rm th}}}_\psi} \\
& \quad + \ubs{(15)}{\sum^{n-|p|-|q|}_{i=1} \, \sum^{|p|+i}_{l=i} (-1)^{\theta^{n-|q|,p}_{i} + \theta^{n,q}_l} \DD^{\scriptscriptstyle{i{\rm th}}}_\gvf \DD^{\scriptscriptstyle{l{\rm th}}}_\psi} 
\end{split}
\end{equation*}
\end{small}
and 
\begin{small}
\begin{equation*}
\begin{split}
(5) &= {\sum^{n-|q|-|p|}_{j=1} \, \sum^{n-|p|}_{i=1} (-1)^{\theta^{n,q}_{j} + \theta^{n,p}_i +1} \DD^{\scriptscriptstyle{j{\rm th}}}_\psi \DD^{\scriptscriptstyle{i{\rm th}}}_\gvf} \\
&=  \ubs{(16)}{\sum^{n-q-|p|}_{i=1} \, \sum^{n-|q|-|p|}_{j=i+1} (-1)^{\theta^{n,q}_{j} + \theta^{n,p}_i +1}  \DD^{\scriptscriptstyle{j{\rm th}}}_\psi \DD^{\scriptscriptstyle{i{\rm th}}}_\gvf} 
+ \ubs{(17)}{\sum^{n-|p|}_{i=q+1} \, \sum^{i-q}_{j=1} (-1)^{\theta^{n,q}_{j} + \theta^{n,p}_i +1}  \DD^{\scriptscriptstyle{j{\rm th}}}_\psi \DD^{\scriptscriptstyle{i{\rm th}}}_\gvf} \\
& \quad + \ubs{(18)}{\sum^{n-|q|-|p|}_{j=1} \, \sum^{|q|+j}_{l=j} (-1)^{\theta^{n,q}_{j} + \theta^{n,p}_l +1}  \DD^{\scriptscriptstyle{j{\rm th}}}_\psi \DD^{\scriptscriptstyle{l{\rm th}}}_\gvf}. 
\end{split}
\end{equation*}
\end{small}
We directly see that $(9) = (15)$, along with $(11) =(18)$.
Furthermore, by a simple observation one sees that
\begin{small}
\begin{equation*}
\begin{split}
(13) &= {\sum^{n-q-|p|}_{j=1} \, \sum^{n-|q|-|p|}_{i=j+1} (-1)^{\theta^{n-|q|,p}_{i} + \theta^{n,q}_j}  \DD^{\scriptscriptstyle{j{\rm th}}}_\psi \DD^{{\scriptscriptstyle(i+|q|){\rm th}}}_\gvf} \\
&= {\sum^{n-q-|p|}_{j=1} \, \sum^{n-|p|}_{k=j+q} (-1)^{\theta^{n-|q|,p}_{k-|q|} + \theta^{n,q}_j}  \DD^{\scriptscriptstyle{j{\rm th}}}_\psi \DD^{\scriptscriptstyle{k{\rm th}}}_\gvf} =: (19),
\end{split}
\end{equation*}
\end{small}
where in the second step we substituted $k:= i+|q|$. Reordering the double sums in $(19)$,
$$
\sum^{n-q-|p|}_{j=1} \, \sum^{n-|p|}_{k=j+q} = \sum^{n-|p|}_{k=q+1} \, \sum^{k-q}_{j=1}, 
$$
and by $\theta^{n-|q|,p}_{k-|q|} = \theta^{n,p}_{k}$, we conclude that $(13) = (19) = - (17)$. Analogously, one proves that $(14) = (16)$.

After a tedious, but straightforward re-ordering of summands one furthermore has

\begin{small}
\begin{equation*}
\begin{split}
(2) %&= {\sum^{p}_{i=1} \, \sum^{n-q+1}_{j=1} (-1)^? \ttt^{n-q-p+2} \DD'_\varphi \ttt^{n-q+1-j+i} \DD'_\psi \ttt^{j+q}} \\
&= {\sum^{p-2}_{i=0} \, \sum^{p-|q|-i}_{k=q+1} (-1)^{\xi^{n,p}_{k+i} + |q||i|} \ttt^{n-|p|-|q|} \DD'_\varphi \ttt^{i} \DD'_\psi \ttt^{k}} 
 \\ & \qquad 
+ {\sum^{p}_{i=1} \, \sum^{i-1}_{k=0} (-1)^{\xi^{n,p}_{k+i} + |q||i|} \ttt^{n-|p|-|q|} \DD'_\varphi \ttt^{i} \DD'_\psi \ttt^{n-|k|}} \\
& \qquad + {\sum^{n-|q|}_{i=p+1} \, \sum^{i-2}_{k=i-p-1} (-1)^{\xi^{n,p}_{|k|+i} +|q||i|} \ttt^{n-|p|-|q|} \DD'_\varphi \ttt^{i} \DD'_\psi \ttt^{n-k}},
\end{split}
\end{equation*}
\end{small}
whereas
\begin{small}
\begin{equation*}
\begin{split}
(10) 
%&= {\sum^{q+p-1}_{i=1} \, \sum^{n+1}_{k=n-p+2} (-1)^? \ttt^{n-q-p+2} \DD'_\varphi \ttt^{n+1-k} \DD'_\psi \ttt^{k+i}} \\
&= {\sum^{|p|}_{i=1} \, \sum^{p+|q|-i}_{k=0} (-1)^{\xi^{n,p}_{k+i} + |q||i|} \ttt^{n-|p|-|q|} \DD'_\varphi \ttt^{i} \DD'_\psi \ttt^{k}} 
+ {\sum^{p+|q|}_{k=1} (-1)^{\xi^{n,p}_{k}} \ttt^{n-|p|-|q|} \DD'_\varphi \DD'_\psi \ttt^{k}} \\
& \quad + {\sum^{|p|}_{i=2} \, \sum^{i-2}_{k=0} (-1)^{\xi^{n,p}_{|k|+i} + |q||i|} \ttt^{n-|p|-|q|} \DD'_\varphi \ttt^{i} \DD'_\psi \ttt^{n-k}}.
\end{split}
\end{equation*}
\end{small}
From these expressions one obtains 
after equally tedious but straightforward computations 
\begin{footnotesize}
\begin{equation*}
\begin{split}
(2) - (10) &= \ubs{(20)}{\sum^{|p|}_{i=0} \, \sum^{q}_{k=1} (-1)^{\xi^{n,p}_{k+i} + |q||i|} \ttt^{n-|p|-|q|} \DD'_\varphi \ttt^{i} \DD'_\psi \ttt^{k}} \\ 
& \quad + \ubs{(21)}{\sum^{n-|q|}_{i=p} \, \sum^{i-1}_{k=i-p} (-1)^{\xi^{n,p}_{k+i} + |q||i|} \ttt^{n-|p|-|q|} \DD'_\varphi \ttt^{i} \DD'_\psi \ttt^{n-|k|}},
\end{split}
\end{equation*}
\end{footnotesize}
and one verifies 
directly that $(20) = (4)$.

So far all terms in the expressions for $\lie_\gvf \lie_\psi$ and $\lie_{\gvf {\bar\circ}\psi}$ cancelled, except 
\begin{small}
\begin{eqnarray*}
(3)  &=& {\sum^{n-|p|-|q|}_{i=1} \, \sum^{q}_{j=1} (-1)^{\theta^{n-|q|, p}_{i} + \xi^{n,q}_{j}} \ttt^{n-|p|-|q|-i} \DD'_\varphi \ttt^{n+p-|q|+i} \DD'_\psi \ttt^{j}},  \\
(21) &=& {\sum^{n-|q|}_{i=p} \, \sum^{i-1}_{k=i-p} (-1)^{\xi^{n,p}_{k+i} + |q||i|} \ttt^{n-|p|-|q|} \DD'_\varphi \ttt^{i} \DD'_\psi \ttt^{n-|k|}}.
\end{eqnarray*}
\end{small}
Repeating the same type of arguments that led to $(21)$ analogously cancels all terms in $-(-1)^{|p||q|} \lie_\psi \lie_\gvf$ and $-(-1)^{|p||q|}\lie_{\psi {\bar\circ}\gvf}$, except
\begin{small}
\begin{eqnarray*}
(7)  &=& {\sum^{n-|q|-|p|}_{j=1} \, \sum^{p}_{i=1} (-1)^{\theta^{n, q}_{j} + \xi^{n,p}_{i} +1} \ttt^{n-|q|-|p|-j} \DD'_\psi \ttt^{n+q-|p|+j} \DD'_\gvf \ttt^{i}},  \\
(22) &:=& {\sum^{n-|p|}_{i=q} \, \sum^{i-1}_{k=i-q} (-1)^{\xi^{n,q}_{k+i} + |p|(|i|+|q|)} \ttt^{n-|q|-|p|} \DD'_\psi \ttt^{i} \DD'_\gvf \ttt^{n-|k|}}.
\end{eqnarray*}
\end{small}
Using \rmref{orecchiette3}, \rmref{parkraumbewirtschaftung2a}, and
\rmref{parkraumbewirtschaftung0}, and the relations of a cyclic
$k$-module we see that
\begin{small}
\begin{equation*}
\begin{split}
(3)  &= {\sum^{n-|p|-|q|}_{i=1} \, \sum^{q}_{j=1} (-1)^{\theta^{n-|q|, p}_{i} + \xi^{n,q}_{j}} \ttt^{n-|p|-|q|-i} \DD'_\varphi \ttt^{p+i} \iota_\psi \ttt \sss_n \ttt^{j-1}}  \\
 &= {\sum^{n-|p|-|q|}_{i=1} \, \sum^{q}_{j=1} (-1)^{\theta^{n-|q|, p}_{i} + \xi^{n,q}_{j}} \iota_\psi \ttt^{n-|p|-|i|} \DD'_\varphi \ttt^{|p|+i} \sss_0 \ttt^{j}}  \\
 &= {\sum^{n-|p|-|q|}_{i=1} \, \sum^{q}_{j=1} (-1)^{\theta^{n-|q|, p}_{i} + \xi^{n,q}_{j}} \iota_\psi \ttt \sss_{n-p+1} \ttt^{n-|p|-|i|} \DD'_\varphi \ttt^{|p|+i+j}}  \\
 &= {\sum^{n-|p|-|q|}_{i=1} \, \sum^{q}_{j=1} (-1)^{\theta^{n-|q|, p}_{i} + \xi^{n,q}_{j}} \ttt^{n-|p|-|q|} \DD'_\psi \ttt^{n-|p|-|i|} \DD'_\varphi \ttt^{|p|+i+j}} =: (23).
\end{split}
\end{equation*}
\end{small}
Substitution of $l:= n-|p|-|i|$ and subsequently of $k:=l-j$ produces
\begin{small}
\begin{equation*}
\begin{split}
(23)  &= {\sum^{n-|p|}_{l=q} \, \sum^{q}_{j=1} (-1)^{\xi^{n,q}_{j} + |p|(|l|+|q|)} \ttt^{n-|p|-|q|} \DD'_\psi \ttt^{l} \DD'_\varphi \ttt^{n-|l|+j}} \\
      &= {\sum^{n-|p|}_{l=q} \, \sum^{l-1}_{k=l-q} (-1)^{\xi^{n,q}_{l+k} + |p|(|l|+|q|)} \ttt^{n-|p|-|q|} \DD'_\psi \ttt^{l} \DD'_\varphi \ttt^{n-|k|}}, 
\end{split}
\end{equation*}
\end{small}
and this is directly seen to be $(22)$. Likewise, one shows that $(7)
= (21)$.

For Eq.~\rmref{alles1}, simply use \rmref{sachengibts} to express $\bb$, then apply \rmref{weimar} to the case where $\gvf := \mu$ and finally make use of \rmref{erfurt}:
\begin{equation*}
\{\bb, \lie_\gvf \} = - \{\lie_\mu, \lie_\gvf\} = - \lie_{\{\mu, \gvf\}} = - \lie_{\delta\gvf}.\qedhere
\end{equation*}
\end{proof}

\subsection{The Gerstenhaber module ${\hhmu}$}
By the identities \rmref{mulhouse2} and \rmref{alles1}, both
operators $\iota_\varphi$ and $\lie_\varphi$ 
descend to well defined operators on 
the Hochschild homology 
${\hhmu}$, provided that $\varphi$ is a cocycle. 
In this case, the following theorem together with Proposition \ref{feinefuellhaltertinte} proves that 
$\iota$ and $\lie$ turn ${\hhmu}$ 
into a module over the 
Gerstenhaber algebra $\hmu$, 
cf.~Def.~\ref{golfoaranci} (ii):

\begin{theorem}
\label{waterbasedvarnish}
If $M$ is a module-comodule over a left Hopf algebroid $U$, 
then for any two cocycles 
$\varphi \in C_M^p(U)$, $\psi \in C_M^q(U)$, 
the induced maps 
$$
\begin{array}{rrcl}
\lie_\varphi:  & \hhmu &\to& H^M_{\bull-|p|}(U), \\
\iota_\psi: & \hhmu &\to& H^M_{\bull-q}(U) 
\end{array}
$$
satisfy
\begin{equation}
\label{radicale1}
[\iota_\psi, \lie_\varphi] = \iota_{\{\psi, \gvf\}}.
\end{equation}
\end{theorem}

\begin{proof}%[Proof of Theorem \ref{waterbasedvarnish}]
Throughout we use relations that we have shown above 
to hold for operators on $C^\cyc_\bull(U,M)$, 
but as we now consider the induced
operators on homology, 
we will also assume tacitly that the
operators only act on cycles and 
that we compute modulo boundaries.

Assume $p+q \leq n+1$ (otherwise both sides in \rmref{radicale1} are zero). Without restriction we may assume that $0 < q < p$, the case of $p = q$ and that of zero cochains being skipped as the proof is similar, but somewhat simpler.
We now have
\begin{footnotesize}
\begin{equation*}
\begin{split}
\iota_\psi \lie_\varphi 
%&= \sum^{n-p+1}_{i=1} (-1)^? \iota_\psi \ttt^{n-p+1-i} \DD'_\varphi \ttt^{i+p} + \sum^{p}_{i=1} (-1)^? \iota_\psi \ttt^{n-p+1} \DD'_\varphi \ttt^{i} \\
&= \sum^{n-|p|-q}_{i=1} (-1)^{\theta^{n,p}_i} \iota_\psi \ttt^{n-|p|-i} \DD'_\varphi \ttt^{i+p} + \sum^{n-|p|}_{i=n-|p|-|q|} (-1)^{\theta^{n,p}_i} \iota_\psi \ttt^{n-|p|-i} \DD'_\varphi \ttt^{i+p} \\
& \qquad + 
\sum^{p}_{i=1} (-1)^{\xi^{n,p}_i} \iota_\psi \ttt^{n-|p|} \DD'_\varphi \ttt^{i} \\
&= \ubs{(1)}{\sum^{n-|p+q|}_{i=1} (-1)^{\theta^{n,p}_i} \iota_\psi \ttt^{n-|p|-i} \DD'_\varphi \ttt^{i+p}} + \ubs{(2)}{\sum^{q}_{k=1} (-1)^{|p|(|q|+|k|)} \iota_{\psi \circ_k \varphi}} 
+  \ubs{(3)}{\sum^{p}_{i=1} (-1)^{\xi^{n,p}_i} \iota_\psi \iota_\varphi \sss_{-1} \ttt^{i-1}}, 
\end{split}
\end{equation*}
\end{footnotesize}
using \rmref{maxdudler} and \rmref{orecchiette1} for the second term and \rmref{parkraumbewirtschaftung2a} for the third term. 
Observe that already
$$
(2) = \iota_{\psi \bar\circ \gvf}.
$$
On the other hand, we see that
\begin{footnotesize}
\begin{equation*}
-(-1)^{q|p|} \lie_\varphi \iota_\psi 
= \ubs{(4)}{\sum^{n-q-|p|}_{i=1} (-1)^{\theta^{n,p}_i +1} \ttt^{n-q-|p|-i} \DD'_\varphi \ttt^{i+p} \iota_\psi} 
+ \ubs{(5)}{\sum^{p}_{i=1} (-1)^{\xi^{n-q,|q||p|}_i} \ttt^{n-q-|p|} \DD'_\varphi \ttt^{i} \iota_\psi}.
\end{equation*}
\end{footnotesize}
By Equation \rmref{orecchiette2}, one immediately observes that $(1) = -(4)$, 
hence we are left to prove that 
\begin{equation}
\label{supergrip}
(3) + (5) =  - (-1)^{|q||p|} \iota_{\varphi \bar\circ \psi} = - \sum^{p}_{i=1} (-1)^{|q||i|} \iota_{\varphi \circ_i \psi}, 
\end{equation}
or, in our former terminology, only the ``twisted'' parts in the Lie derivative still matter.

By \rmref{parkraumbewirtschaftung2a}, we see that
\begin{footnotesize}
$$
(5) = \sum^{p}_{i=1} (-1)^{\xi^{n-q,|q||p|}_i} \iota_\varphi \sss_{-1} \ttt^{i-1} \iota_\psi 
= \ubs{(6)}{\sum^{p-1}_{i=1} (-1)^{\xi^{n-q,|q||p|}_i} \iota_\varphi \sss_{-1} \ttt^{i-1} \iota_\psi} 
+ \ubs{(7)}{(-1)^{\xi^{n-q,|q||p|}_p} \iota_\varphi \sss_{-1} \ttt^{p-1} \iota_\psi},
$$
\end{footnotesize}
and we continue with
\begin{footnotesize}
\begin{equation*}
\begin{split}
(6) 
&= \sum^{p-1}_{i=1} (-1)^{\xi^{n-q,|q||p|}_i} \iota_\varphi \sss_{-1} \ttt^{i-1} \dd_0 \DD'_\psi 
 = \sum^{p-1}_{i=1} (-1)^{\xi^{n-q,|q||p|}_i} \iota_\varphi \dd_i \sss_{-1} \ttt^{i-1} \DD'_\psi \\
&= \sum^{p-1}_{i=1} \, \sum^{n-q+2}_{{j=0}\atop{j \neq i}} (-1)^{\xi^{n-q,|q||p|}_i + |j-i|} \iota_\varphi \dd_j \sss_{-1} \ttt^{i-1} \DD'_\psi \\
&= \ubs{(8)}{\sum^{p-1}_{i=1} \, \sum^{n-|q|}_{{j=1}\atop{j \neq i}} (-1)^{\xi^{|n|-q,|q||p|}_i + j} \iota_\varphi \dd_j \sss_{-1} \ttt^{i-1} \DD'_\psi} \\
& \quad  
+ \ubs{(9)}{\sum^{p-1}_{i=1} (-1)^{\xi^{|n|-q,|q||p|}_i} \iota_\varphi \ttt^{i-1} \DD'_\psi + \sum^{p-1}_{i=1} (-1)^{\xi^{|n|-q,|q||p|}_{|i|}+1} \iota_\varphi \ttt^{i} \DD'_\psi},
\end{split}
\end{equation*}
\end{footnotesize}
where in the third line we used \rmref{mulhouse2} together with the fact that $\varphi$ is a cocycle, and that we deal here with the induced maps on ${\hhmu}$, i.e., $\bb \iota_\varphi = 0 = \iota_\varphi \bb$. Observe now that
\begin{footnotesize}
$$
{(9)} = (-1)^{|q||p|+1} \iota_\varphi \DD'_\psi  + (-1)^{n|p|} \iota_\varphi \ttt^{|p|} \DD'_\psi  = \ubs{(10)}{(-1)^{|q||p|+1} \iota_{\varphi \circ_p \psi}}  
+ \ubs{(11)}{(-1)^{n|p|} \iota_\varphi \ttt^{|p|} \DD'_\psi}.
$$
\end{footnotesize}
Furthermore,
\begin{footnotesize}
\begin{equation*}
(8) = \ubs{(12)}{\sum^{p-3}_{i=0} \, \sum^{n-|q|}_{{j=1}} (-1)^{\xi^{n-q,|q||p|}_{|i|} + |j|} \iota_\varphi \sss_{-1} \ttt^{i} \dd_j \DD'_\psi} 
    + \ubs{(13)}{\sum^{n-|q|-|p|}_{{j=1}} (-1)^{\xi^{|n|,q}_{|p|}+j} \iota_\varphi \sss_{-1} \ttt^{p-2} \dd_j \DD'_\psi},  
\end{equation*}
\end{footnotesize}
where by \rmref{cabras} and \rmref{parkraumbewirtschaftung-1} we have
\begin{footnotesize}
\begin{equation*}
\begin{split}
(12) 
&= {\sum^{p-3}_{i=0} \, \sum^{n}_{{j=q+1}} (-1)^{\xi^{n-q,qp}_{|i|} + |j + p|} \iota_\varphi \sss_{-1} \ttt^{i} \DD'_\psi \dd_j} 
+ {\sum^{p-3}_{i=0} \, \sum^{q}_{{j=1}} (-1)^{\xi^{n-q,qp}_{|i|} + |j + p|} \iota_\varphi \sss_{-1} \ttt^{i} \DD'_\psi \dd_j} \\
& \quad
+ {\sum^{p-3}_{i=0} (-1))^{\xi^{n-q,|q||p|}_{|i|} + |q|} \iota_\varphi \sss_{-1} \ttt^{i} \dd_1 \ttt \DD'_\psi \ttt^n} \\ 
&= \ubs{(14)}{\sum^{p-3}_{i=0} (-1)^{\xi^{n-q,qp}_{|i|} + p} \iota_\varphi \sss_{-1} \ttt^{i} \DD'_\psi \dd_0}
+ \ubs{(15)}{\sum^{p-3}_{i=0} (-1)^{\xi^{n-q,qp}_{|i|} + |p|} \iota_\varphi \sss_{-1} \ttt^{i} \dd_1 \ttt \DD'_\psi \ttt^n},  
\end{split}
\end{equation*}
\end{footnotesize}
where in the second line we used that the representatives in
${\hhmu}$ are cycles.
By a similar argument we get, still with \rmref{cabras},
\begin{footnotesize}
\begin{equation*}
\begin{split}
(13) &= {\sum^{n-|q|-|p|}_{{j=2}} (-1)^{\xi^{|n|,q}_{|p|} + j} \iota_\varphi \sss_{-1} \ttt^{p-2} \dd_j \DD'_\psi} 
+ {(-1)^{|n|p} \iota_\varphi \sss_{-1} \ttt^{p-2} \dd_1 \ttt \DD'_\psi \ttt^n} \\
& \quad
    + {\sum^{q}_{{j=1}} (-1)^{\xi^{|n|,|j|}_{|p|}} \iota_\varphi \sss_{-1} \ttt^{p-2} \DD'_\psi \dd_j} \\
&= \ubs{(16)}{(-1)^{|n|p} \iota_\varphi \sss_{-1} \ttt^{p-2} \dd_1 \ttt \DD'_\psi \ttt^n}    
+ \ubs{(17)}{\sum^{n}_{{j=n-|p|+1}} (-1)^{\xi^{|n|,j}_{|p|}} \iota_\varphi \sss_{-1} \ttt^{p-2} \DD'_\psi \dd_j} \\
& \quad
+ \ubs{(18)}{(-1)^{|n|p+1} \iota_\varphi \sss_{-1} \ttt^{p-2} \DD'_\psi \dd_0}.    
\end{split}
\end{equation*}
\end{footnotesize}
We now see that
\begin{footnotesize}
\begin{equation*}
\begin{split}
(14)&+(18)+(15)+(16) \\
&= \sum^{p-2}_{i=0} (-1)^{\xi^{n-q,qp}_{|i|} + p} \iota_\varphi \sss_{-1} \ttt^{i} \DD'_\psi \dd_0 
+ {\sum^{p-2}_{i=0} (-1)^{\xi^{n-q,qp}_{|i|} + |p|} \iota_\varphi \sss_{-1} \ttt^{i} \dd_1 \ttt \DD'_\psi \ttt^n} \\ 
&= \sum^{p-2}_{i=0} (-1)^{\xi^{n-q,qp}_{|i|} + p} \iota_\varphi \sss_{-1} \ttt^{i} \dd_{n-|q|} \DD'_\psi \ttt^n 
+ {\sum^{p-2}_{i=0} (-1)^{\xi^{n-q,qp}_{|i|} + |p|} \iota_\varphi \sss_{-1} \ttt^{i} \dd_1 \ttt \DD'_\psi \ttt^n} \\
&= {\sum^{p-2}_{i=0} (-1)^{\xi^{n-q,qp}_{|i|} + p} \iota_\varphi \sss_{-1} \ttt^{i} (\dd_0-\dd_1) \ttt \DD'_\psi \ttt^n} =: (19). 
\end{split}
\end{equation*}
\end{footnotesize}
Let us come back to the other half and compute (3): to this end, consider first
\begin{footnotesize}
\begin{equation*}
\begin{split}
\iota_\psi& \iota_\varphi(m, u^1, \ldots, u^n) \\
&= \big(m, u^1, \ldots, \psi(u^{n-|p+q|}, \ldots, \varphi(u^{n-|p|}, \ldots, u^n) \blact u^{n-p}) \blact u^{n-p-q}\big) \\
&= \big(m, u^1, \ldots, \varepsilon\big(\varphi(u^{n-|p|}, \ldots, u^n) \blact \DD_\psi(u^{n-|p+q|}, \ldots, u^{n-p})\big) \blact u^{n-p-q}\big) \\
&= \big(m, u^1, \ldots, \varphi(\DD_\psi(u^{n-|p+q|}, \ldots, u^{n-p})u^{n-|p|}, \ldots, u^n) \blact u^{n-p-q}\big) \\ 
&\quad + \sum^{n-1}_{i=n-|p|} \!\! (-1)^{i-n+p} \Big(m, u^1, \ldots, \varphi\big(\DD_\psi(u^{n-|p+q|}, \ldots, u^{n-p}), \ldots, u^{i}u^{i+1}, \ldots, u^n\big) \blact u^{n-p-q}\Big) \\ 
&\quad + (-1)^p \Big(m, u^1, \ldots, \varphi\big(\DD_\psi(u^{n-|p+q|}, \ldots, u^{n-p}), \ldots, \varepsilon(u^n) \blact u^{n-1}\big) \blact u^{n-p-q}\Big),  
\end{split}
\end{equation*}
\end{footnotesize}
which is true since $\varphi$ is a cocycle; that is, with the help of \rmref{orecchiette1},
$$
\iota_\psi\iota_\varphi = \sum^p_{i=0} (-1)^{i+p} \iota_\varphi \dd_i \ttt^p \DD'_\psi
t^{n-|p|}. 
$$
Hence, by \rmref{parkraumbewirtschaftung-1} and \rmref{parkraumbewirtschaftung0}, 
\begin{footnotesize}
\begin{equation*}
\begin{split}
(3) &= {\sum^{p}_{j=1} (-1)^{\xi^{n,p}_{j}} \iota_\psi \iota_\varphi \sss_{-1} \ttt^{j-1}} \\
&=  {\sum^{p}_{j=1} \, \sum^p_{i=0} (-1)^{\xi^{n,i}_{j}} \iota_\varphi \dd_i \sss_{-1} \ttt^{|p|} \DD'_\psi \ttt^{n-|p|+j}} \\
&=  \ubs{(20)}{\sum^{p-1}_{j=1} \, \sum^p_{i=0} (-1)^{\xi^{n,i}_{j}} \iota_\varphi \dd_i \sss_{-1} \ttt^{|p|} \DD'_\psi \ttt^{n-|p|+j}} 
 + \ubs{(21)}{\sum^p_{i=0} (-1)^{\xi^{n,i}_{p}} \iota_\varphi \dd_i \sss_{-1} \ttt^{|p|} \DD'_\psi},  
\end{split}
\end{equation*}
\end{footnotesize}
where we continue with
\begin{footnotesize}
\begin{equation*}
(21) = (-1)^{n|p| +1} \iota_\varphi \ttt^{|p|} \DD'_\psi 
+  \sum^{n-|q|}_{k=n-|q|-|p|+1} (-1)^{|n|p-|q|+k} \iota_\varphi \sss_{-1} \ttt^{p-2} \dd_k \DD'_\psi
+ (-1)^{|n||p|} \iota_\varphi \dd_p \sss_{-1} \ttt^{|p|} \DD'_\psi, 
\end{equation*}
\end{footnotesize}
and these three terms are precisely, by \rmref{parkraumbewirtschaftung-1} and \rmref{parkraumbewirtschaftung1} again, the terms $-(11)$, $-(16)$, and $-(7)$, respectively. We furthermore have
\begin{footnotesize}
\begin{equation*}
(20)
=  \ubs{(22)}{\sum^{p-1}_{j=1} \, \sum^p_{i=1} (-1)^{\xi^{n,i}_{j}} \iota_\varphi \dd_i \sss_{-1} \ttt^{p-1} \DD'_\psi \ttt^{n-|p|+j}} 
+  \ubs{(23)}{\sum^{p-1}_{j=1} (-1)^{n|j|+1} \iota_\varphi \ttt^{p-1} \DD'_\psi \ttt^{n-|p|+j}}, 
\end{equation*}
\end{footnotesize}
where
\begin{footnotesize}
\begin{equation*}
(23) = \ubs{(24)}{\sum^{p-1}_{j=2} (-1)^{n|j|+1} \iota_\varphi \ttt^{p-1} \DD'_\psi \ttt^{n-|p|+j}} 
- \ubs{(25)}{\iota_\varphi \ttt^{p-1} \DD'_\psi \ttt^{n-p+2}}, 
\end{equation*}
\end{footnotesize}
and we observe that $(25) = \iota_{\varphi \circ_1 \psi}$.

For better orientation let us state were we are at this point: we are left with the equations
\begin{eqnarray}
\label{31}
(19) &=& {\sum^{p-1}_{i=1} (-1)^{\xi^{n-q,qp}_{i}+p} \iota_\varphi (\dd_i - \dd_{i+1}) \sss_{-1} \ttt^{i} \DD'_\psi \ttt^n}, \\
\label{22}
(22) &=& \sum^{p-1}_{j=1} \, \sum^p_{i=1} (-1)^{\xi^{n,i}_{j}} \iota_\varphi \dd_i \sss_{-1} \ttt^{p-1} \DD'_\psi \ttt^{n-|p|+j}, \\
\label{24}
(24) &=& \sum^{p-1}_{j=2} (-1)^{n|j|+1} \iota_\varphi \ttt^{p-1} \DD'_\psi \ttt^{n-|p|+j}, 
\end{eqnarray}
and we are also missing the terms, cf.~\rmref{supergrip}, 
$$
-\sum^{p-1}_{i=2} (-1)^{|q||i|} \iota_{\varphi \circ_i \psi}.
$$ 
The proof proceeds now in recursive steps, which at each step reproduce formally the Equations \rmref{31}--\rmref{24}, but with lower degrees, and one of the $\iota_{\varphi \circ_i \psi}$. We only give the next step: start with
\begin{footnotesize}
\begin{equation*}
(22)
=  \ubs{(26)}{\sum^{p-2}_{j=1} \, \sum^p_{i=1} (-1)^{\xi^{n,i}_{j}} \iota_\varphi \dd_i \sss_{-1} \ttt^{p-1} \DD'_\psi \ttt^{n-|p|+j}} 
+  \ubs{(27)}{\sum^p_{i=1} (-1)^{\xi^{n,i}_{|p|}} \iota_\varphi \dd_i \sss_{-1} \ttt^{p-1} \DD'_\psi \ttt^{n}}, 
\end{equation*}
\end{footnotesize}
where
\begin{footnotesize}
\begin{equation*}
\begin{split}
(26) &=  \ubs{(28)}{\sum^{p-2}_{j=1} \, \sum^{n-|q|}_{i=n-|q|-|p|+1} (-1)^{\xi^{n,i}_{|j|} + q + p} \iota_\varphi \sss_{-1} \ttt^{p-2} \dd_i \DD'_\psi \ttt^{n-|p|+j}} \\
& \qquad +  \ubs{(29)}{\sum^{p-3}_{j=1} (-1)^{\xi^{n,p}_{j}} \iota_\varphi \sss_{-1} \ttt^{p-2} \dd_1 \ttt \DD'_\psi \ttt^{n-|p|+j}} 
 +  \ubs{(30)}{(-1)^{|n||p|} \iota_\varphi \sss_{-1} \ttt^{p-2} \dd_1 \ttt \DD'_\psi \ttt^{n-1}}. 
\end{split}
\end{equation*}
\end{footnotesize}
Then
\begin{footnotesize}
\begin{equation*}
\begin{split}
(28) &=  \ubs{(31a)}{\sum^{n-2}_{j=n-p+2} \, \sum^{n-j}_{i=0} (-1)^{\xi^{n,|i|}_{|p+j|}}  \iota_\varphi \sss_{-1} \ttt^{p-2} \DD'_\psi \ttt^{j} \dd_i} \\
& 
 \qquad +  \ubs{(31b)}{\sum^{n-2}_{j=n-p+2} \, \sum^{j-n+p-2}_{i=0} (-1)^{\xi^{n,|i|}_{p+j}}  \iota_\varphi \sss_{-1} \ttt^{p-2} \DD'_\psi \ttt^{j} \dd_{n-i}} \\
& \qquad + \ubs{(32)}{(-1)^{|n||p|} \iota_\varphi \sss_{-1} \ttt^{p-2} \DD'_\psi \ttt^{n-1} (\dd_0 - \dd_1)}.
\end{split}
\end{equation*}
\end{footnotesize}
Since the representatives of the elements we consider are in $\ker \bb$, we conclude
\begin{footnotesize}
\begin{equation*}
\begin{split}
(31a)+(31b) &=  {\sum^{p-3}_{j=1} \, \sum^{n-j}_{i=p-j} (-1)^{\xi^{n,i}_{|j+p|}}  \iota_\varphi \sss_{-1} \ttt^{p-2} \DD'_\psi \ttt^{n-|p|+j} \dd_i} \\
&=  {\sum^{p-3}_{j=1} \, \sum^{n-|p|}_{i=1} (-1)^{\xi^{n,|i+p|}_{j}}  \iota_\varphi \sss_{-1} \ttt^{p-2} \DD'_\psi \dd_i \ttt^{n-|p|+1+j}} =: (33).
\end{split}
\end{equation*}
\end{footnotesize}
Now, again by \rmref{cabras}, we have 
\begin{footnotesize}
\begin{equation*}
\begin{split}
(33) + (29) 
&= {\sum^{p-3}_{j=1} \, \sum^{n-|q|-|p|}_{i=1} (-1)^{\xi^{n,|i+p+q|}_{j}}  \iota_\varphi \sss_{-1} \ttt^{p-2} \dd_i \DD'_\psi \ttt^{n-p+2+j}} \\
&= {\sum^{p-3}_{j=1} \, \sum^{n-|q|}_{i=p} (-1)^{\xi^{n,i+q}_{j}}  \iota_\varphi \dd_i \sss_{-1} \ttt^{p-2} \DD'_\psi \ttt^{n-p+2+j}} \\
&= {\sum^{p-3}_{j=1} \, \sum^{p-1}_{i=0} (-1)^{\xi^{n,|i+q|}_{j}}  \iota_\varphi \dd_i \sss_{-1} \ttt^{p-2} \DD'_\psi \ttt^{n-p+2+j}} \\
& \qquad + {\sum^{p-3}_{j=1} (-1)^{nj} \iota_\varphi \dd_{n-q+2} \sss_{-1} \ttt^{p-2} \DD'_\psi \ttt^{n-p+2+j}} \\
&= \ubs{(34)}{\sum^{p-3}_{j=1} \, \sum^{p-1}_{i=1} (-1)^{\xi^{n,|i+q|}_{j}}  \iota_\varphi \dd_i \sss_{-1} \ttt^{p-2} \DD'_\psi \ttt^{n-p+2+j}}  
+ \ubs{(35)}{\sum^{p-3}_{j=1} (-1)^{\xi^{n,|q|}_{j}} \iota_\varphi \ttt^{p-2} \DD'_\psi \ttt^{n-p+2+j}} \\
& \qquad 
+ \ubs{(36)}{\sum^{p-3}_{j=1} (-1)^{nj} \iota_\varphi \ttt^{p-1} \DD'_\psi \ttt^{n-p+2+j}}, 
\end{split}
\end{equation*}
\end{footnotesize}
where in the third equation we used one more time $\bb \iota_\varphi = 0 = \iota_\varphi \bb$, which holds in our situation.
One furthermore has
\begin{footnotesize}
\begin{equation*}
(35) 
= \ubs{(37)}{\sum^{p-2}_{j=3} (-1)^{nj+q} \iota_\varphi \ttt^{p-2} \DD'_\psi \ttt^{n-p+1+j}} 
+ \ubs{(38)}{(-1)^q \iota_\varphi \ttt^{p-2} \DD'_\psi \ttt^{n-p+3}}, 
\end{equation*}
\end{footnotesize}
and we see that $(38) = -(-1)^{|q|} \iota_{\varphi \circ_2 \psi}$, that is, the second summand in \rmref{supergrip}. Moreover,
\begin{footnotesize}
\begin{equation*}
(27) + (19) 
= \ubs{(39)}{\sum^{p-2}_{i=1} (-1)^{\xi^{n-q,qp}_{i} +p} \iota_\varphi (\dd_i - \dd_{i+1}) \sss_{-1} \ttt^{i} \DD'_\psi \ttt^{n}} 
+ \ubs{(40)}{\sum^{p-2}_{i=1} (-1)^{np + |i|} \iota_\varphi \dd_i \sss_{-1} \ttt^{p-1} \DD'_\psi \ttt^{n}}, 
\end{equation*}
\end{footnotesize}
where
\begin{footnotesize}
\begin{equation*}
(40) = {\sum^{n-q}_{i=n-|q|-|p|+1} (-1)^{\xi^{|n|,|i+q|}_{p}} \iota_\varphi \sss_{-1} \ttt^{p-2} \dd_i \DD'_\psi \ttt^{n}}
= \ubs{(41)}{\sum^{n}_{i=n-p+3} (-1)^{|n||p|+i} \iota_\varphi \sss_{-1} \ttt^{p-2} \DD'_\psi \ttt^{n-1} \dd_i}.
\end{equation*}
\end{footnotesize}
Furthermore, we obtain
\begin{footnotesize}
\begin{equation*}
\begin{split}
(41) + (32) 
&= {\sum^{n-p+2}_{i=2} (-1)^{\xi^{|n|,i}_{p}}  \iota_\varphi \sss_{-1} \ttt^{p-2} \DD'_\psi \ttt^{n-1} \dd_i} \\ 
&= {\sum^{q}_{i=1} (-1)^{|n||p|+i} \iota_\varphi \sss_{-1} \ttt^{p-2} \DD'_\psi \dd_i \ttt^{n}} 
+ {\sum^{n-|p|}_{i=q+1} (-1)^{|n||p|+i} \iota_\varphi \sss_{-1} \ttt^{p-2} \DD'_\psi \dd_i \ttt^{n}} \\ 
&= \ubs{(42)}{(-1)^{|n||p|+1} \iota_\varphi \sss_{-1} \ttt^{p-2} \dd_1 \ttt \DD'_\psi \ttt^{n-1}} 
+ \ubs{(43)}{\sum^{n-|q|-|p|}_{i=2} (-1)^{|n||p|+|q+i|} \iota_\varphi \sss_{-1} \ttt^{p-2} \dd_i \DD'_\psi \ttt^{n}}, 
\end{split}
\end{equation*}
\end{footnotesize}
where for the first term in the last line we used \rmref{cabras}.
By $\bb \iota_\varphi = 0 = \iota_\varphi \bb$ again, one has
\begin{footnotesize}
\begin{equation*}
\begin{split}
(43) 
&= {\sum^{n-|q|}_{i=p} (-1)^{|n||p|+i+p+q} \iota_\varphi \dd_i \sss_{-1} \ttt^{p-2} \DD'_\psi \ttt^{n}} \\
&= \ubs{(44)}{\sum^{p-1}_{i=1} (-1)^{n|p|+i+q} \iota_\varphi \dd_i \sss_{-1} \ttt^{p-2} \DD'_\psi \ttt^{n}}
+ \ubs{(45)}{(-1)^{n|p|+q} \iota_\varphi \ttt^{p-2} \DD'_\psi \ttt^{n}} 
+ \ubs{(46)}{(-1)^{np} \iota_\varphi \ttt^{p-1} \DD'_\psi \ttt^{n}}. 
\end{split}
\end{equation*}
\end{footnotesize}
Finally, we see that $(42) = - (30)$, that $(36)+(46) = -(24)$, and that
\begin{footnotesize}
\begin{equation*}
(34) + (44) 
= {\sum^{p-1}_{j=2} \, \sum^{p-1}_{i=1} (-1)^{nj+i+q} \iota_\varphi \dd_i \sss_{-1} \ttt^{p-2} \DD'_\psi \ttt^{n-|p|+j}}  =: (47),
\end{equation*}
\end{footnotesize}
as well as
\begin{footnotesize}
\begin{equation*}
(37) + (45) 
= {\sum^{p-1}_{j=3}  (-1)^{nj+q} \iota_\varphi \ttt^{p-2} \DD'_\psi \ttt^{n-|p|+j}} =:(48).  
\end{equation*}
\end{footnotesize}
We are now left with the three terms
%\begin{footnotesize}
\begin{eqnarray}
\label{53}
(39) &=& {\sum^{p-2}_{i=1} (-1)^{\xi^{n-q,qp}_{i} +p}  \iota_\varphi (\dd_i - \dd_{i+1}) \sss_{-1} \ttt^{i} \DD'_\psi \ttt^{n}}, \\
\label{54}
(47) &=& {\sum^{p-1}_{j=2} \, \sum^{p-1}_{i=1} (-1)^{\xi^{n,|i|}_{|j|} +q}  \iota_\varphi \dd_i \sss_{-1} \ttt^{p-2} \DD'_\psi \ttt^{n-|p|+j}}, \\  
\label{55}
(48) &=& {\sum^{p-1}_{j=3}  (-1)^{nj+q} \iota_\varphi \ttt^{p-2} \DD'_\psi \ttt^{n-|p|+j}},  
\end{eqnarray}
%\end{footnotesize}
and these correspond (with alternating signs) to the Eqs.~\rmref{31}--\rmref{24}, but with one summand 
less and $p$ lowered by one, respectively. Also, we obtained $\iota_{\varphi \circ_2 \psi}$, 
see $(38)$, on the way. Repeating the same steps as above another $p-3$ times yields the missing terms 
$$
- \sum^{p-1}_{i=3} (-1)^{|q||i|} {\iota_{\varphi \circ_i \psi} = - \sum^{p-1}_{i=3} (-1)^{|q||i|} \iota_\varphi \ttt^{p-i} \DD'_\psi \ttt^{n-|p|+i}},  
$$
in \rmref{supergrip},
and cancels the rest. Observe that in \rmref{54} and \rmref{55} the factor $(-1)^q$ appears in contrast to \rmref{22} and \rmref{24}, but in correspondence to the sign rule in \rmref{supergrip}.
\end{proof}

\section{The Batalin-Vilkovisky module}
This section contains the 
both conceptually and computationally 
most involved aspect of our paper,
which is a Hopf algebroid generalisation of the 
Cartan-Rinehart homotopy formula. This is a relation on the
(co)chain level which implies on 
(co)homology the Batalin-Vilkovisky relation that expresses
$\lie_\varphi $ as the graded commutator of $\BB$ and $
\iota_\varphi$. In other words, establishing 
this formula will complete the proof that 
$\hmu$ and
$\hhmu$ form a 
differential calculus. 

\subsection{The operators $\SSS_\varphi$}
We begin by defining 
the generalisation 
of the operator 
denoted by $\SSS$ in the work  
Nest, Tsygan and Tamarkin \cite{NesTsy:OTCROAA, Tsy:CH,
  TamTsy:NCDCHBVAAFC, TamTsy:CFAIT}, by $\mathbf{B}$ in Getzler's work
\cite{Get:CHFATGMCICH}, and by $f$ in Rinehart's paper
\cite{Rin:DFOGCA}. This operator may be considered as a generalisation
of the cap product for the cyclic bicomplex. Throughout this section,
$U$ is assumed to be a left Hopf algebroid and $M$ is a
module-comodule (not necessarily an SaYD module).

\begin{definition}
Given $\varphi \in  C^p(U,A)$, we define 
$$
        \SSS_\varphi : C_n(U,M) \rightarrow 
        C_{n-p+2}(U,M)
$$
for $p \le n$ by
\begin{equation*}
\label{capillareal1}
        \SSS_\varphi := 
        \sum^{n-p}_{j=0} \, 
        \sum^j_{i=0} (-1)^{\eta^{n,p}_{j,i}} \sss_{-1} \,
        \ttt^{n-p-i} \, \DD'_\varphi \, \ttt^{n+i-|j|},
\end{equation*}
where the sign is given by
\begin{equation*}
\label{nerv1}
        \eta^{n,p}_{j,i} := 
        nj + |p|i.
%n(n-j-1) + |p|i.
\end{equation*}
For $p>n$, we put
$$
        \SSS_\varphi := 0.
$$
\end{definition}

\begin{rem}
Observe that the extra degeneracy \rmref{extra} is given here as $\sss_{-1} = \ttt \, \sss_{n-|p|}$.
\end{rem}

In general, inserting the explicit formula 
for $\ttt,\DD'_\varphi$ and 
$\sss_{-1}$ results in truly unpleasant expressions. 
However,  
in case $M$ is an SaYD module and 
hence $C_\bull(U,M)$ a cyclic module, 
these can be at least somewhat 
simplified:

\begin{prop}
If $M$ is an SaYD module over a left Hopf algebroid $U$, then   
$\SSS_\varphi$,
for $\varphi \in C^p(U,A)$, $p \leq n$, assumes the following form:
\begin{equation*}
\label{capillareal2}
\begin{split}
\SSS_\varphi(m, x) &= 
%& \sum_{i=1}^{n-p+1} (-1)^{i+1} (m_{(0)}, u^1_+, \ldots, \DD_\varphi(u^i_+, \ldots, u^{i+p-1}_+), \ldots, u^n_+, u^n_- \cdots u^1_- m_{(-1)}) \\
\sum_{i=0}^{n-p} \ \sum_{j=i+1}^{n-|p|} (-1)^{n(i+|p|) + |p|(j+i+1)} 
\big(m_{(0)} u^1_{+(2)} \cdots u^{i}_{+(2)}, u^{i+1}_+, \ldots, \\ 
& \qquad \qquad \DD_\varphi(u^{j}_+, \ldots, u^{j+|p|}_+),  
%& \qquad \qquad \qquad \qquad \qquad \qquad 
\ldots, u^n_+, u^n_- \cdots u^1_- m_{(-1)}, u^1_{+(1)}, \ldots, u^{i}_{+(1)}\big).
\end{split}
\end{equation*}
\end{prop}
\begin{proof}
Direct computation.
\end{proof} 
\begin{example}
For $n=1, \, p=1$, the above means:
\begin{equation*}
%\begin{split}
\SSS_\varphi(m, u) = (m_{(0)}, \varphi(u_{+(1)}) \lact u_{+(2)}, u_- m_{(-1)}),
%\end{split}
\end{equation*}
while it becomes for $n=2, p=1$:
\begin{equation*}
\begin{split}
\SSS_\varphi(m, u, v) &= (m_{(0)}, \varphi(u_{+(1)}) \lact u_{+(2)}, v_+,  v_-u_- m_{(-1)}) \\ 
&\quad + (m_{(0)}, u_+, \varphi(v_{+(1)}) \lact v_{+(2)}, v_-u_- m_{(-1)}) \\ 
&\quad + (m_{(0)}u_{+(2)}, \varphi(v_{+(1)}) \lact v_{+(2)}, v_-u_- m_{(-1)}, u_{+(1)}). 
\end{split}
\end{equation*}
For $n=3$ and $p=2$, we get
\begin{equation*}
\begin{split}
\SSS_\varphi(m, u, v, w) &= - (m_{(0)}, \varphi(u_{+(1)}, v_{+(1)}) \lact 
u_{+(2)}v_{+(2)}, w_+, w_-v_-u_-m_{(-1)}) \\
& \quad + (m_{(0)}, u_+, \varphi(v_{+(1)},  w_{+(1)}) \lact v_{+(2)}w_{+(2)}, w_-v_-u_-m_{(-1)}) \\
& \quad + (m_{(0)}u_{+(2)}, \varphi(v_{+(1)}, w_{+(1)}) \lact v_{+(2)}w_{+(2)}, w_-v_-u_-m_{(-1)}, u_{+(1)}).
\end{split}
\end{equation*}
\end{example}

\subsection{The relation $[\BB,\SSS_\varphi]=0$}

Our first result 
is that $\SSS_\varphi$ commutes with 
$\BB$. As this simplifies the 
formula for $\BB$, we will from now on 
be working on the
reduced chain complex $\bar C_\bull(U,M)$ 
resp.~$\bar C_\bull^{\cyc}(U,M)$, which dually
requires passing also to the reduced cochain complex:
\begin{definition}\label{thomson}
We denote by $\bar C^\bull(U,A)$ respectively 
$\bar C_M^\bull(U)$ the 
intersection of the kernels of the
codegeneracies in the cosimplicial $k$-modules  
$C^\bull(U,A)$ respectively 
$C^\bull_{M}(U)$.
\end{definition}

\begin{prop}
\label{pleiadians}
For any $\varphi \in \bar C^p(U,A)$ 
the identity
\begin{equation}
\label{alles3}
[\BB, \SSS_\varphi] = 0
\end{equation}
holds on the reduced chain complex 
$\bar C_\bull(U,M)$.
\end{prop}
\begin{proof}
Explicitly, the graded commutator reads 
on the reduced complex
$$
[\BB, \SSS_\varphi] =  
\ttt \, 
\sss_{n-p+2} \, \NN \, \SSS_\varphi - 
(-1)^{p-2} \, \SSS_{\varphi} \, \ttt \, \sss_n \, \NN.
$$ 
If $p > n+1$, the entire expression is already zero.
% and for $p = n+1$ the first summand 
% $ \ttt \, s_{n-p+2} \, \NN \,  \SSS_\varphi$ vanishes. 
Hence assume that $p \leq n+1$ and first consider the second summand: 
it suffices to show that the 
image of $\SSS_{\varphi} \, \ttt \, \sss_n$ 
on elements of degree $n$ is degenerate, and this can be seen as 
follows:
%\begin{footnotesize}
\begin{equation*}
\begin{split}
        \SSS_{\varphi} \, \ttt \, \sss_n 
&= 
        \sum^{n-p+1}_{j=0} \, \sum^j_{i=0} 
        (-1)^{\eta^{n,p}_{j,i}} \ttt \, \sss_{n-p+2} \, \ttt^{n-p-i+1} \, 
        \DD'_\varphi \, \ttt^{n+i-j+2} \, \ttt \, \sss_n \\
 &= 
        \sum^{n-p+1}_{j=0} \, \sum^j_{i=0} 
        (-1)^{\eta^{n,p}_{j,i}} \ttt \, \sss_{n-p+2} \, \ttt^{n-p-i+1} \, 
        \DD'_\varphi \, \ttt^{n+i-j+1} \, \sss_0 \, \ttt \\
 &= 
        \sum^{n-p+1}_{j=0} \, \sum^j_{i=0} 
        (-1)^{\eta^{n,p}_{j,i}} \ttt \, \sss_{n-p+2} \, \ttt^{n-p-i+1} \, 
        \DD'_\varphi \, \sss_{n-(j-i)+1} \, \ttt^{n+i-j+2} \\
&= 
        \sum^{n-p+1}_{j=0} \, \sum^j_{i=0} 
        (-1)^{\eta^{n,p}_{j,i}} \ttt \, \sss_{n-p+2} \, \ttt^{n-p-i+1}  
        \sss_{n-(j-i)-p+2} \,  \DD'_\varphi \, \ttt^{n+i-j+2} \\
&= 
        \sum^{n-p+1}_{j=0} \, \sum^j_{i=0} 
        (-1)^{\eta^{n,p}_{j,i}} \ttt \, \sss_{n-p+2} \, \ttt^{n-p-j+2}  
        \, \sss_{n-p+1} \, \ttt^{j-i-1} \, \DD'_\varphi \, 
        \ttt^{n+i-j+2} \\
&= 
        \sum^{n-p+1}_{j=0} \, \sum^j_{i=0} 
        (-1)^{\eta^{n,p}_{j,i}} \ttt \, \sss_{n-p+2} \, \ttt^{n-p-j}  
        \, \sss_{0} \, \ttt^{j-i} \, \DD'_\varphi \, 
        \ttt^{n+i-j+2},
\end{split}
\end{equation*}
%\end{footnotesize}
using the simplicial and cyclic relations as well as
\rmref{parkraumbewirtschaftung0} in the third line, along with the
fact that $j-i = 0, \ldots, n-p+1$. Now we distinguish the following
cases: we have on $\bar C^\cyc_\bull(U,M)$
\begin{small}
$$
\ttt \sss_{n-p+2} \ttt^{n-p-j}  \sss_{0} = 
\begin{cases} 
\begin{array}{ll}
\ttt \sss_{n-p+2} \ttt^{n-p+3}  \sss_{0} = \ttt \sss_{n-p+2} \sss_{n-p+3} \ttt^{n-p+3} & \mbox{if} \quad j = n-p+1, \\
\ttt \sss_{n-p+2} \sss_{0} & \mbox{if} \quad j = n-p, \\
\ttt \sss_{n-p+2} \ttt  \sss_{0} & \mbox{if} \quad j = n-p-1, \\
\ttt \sss_{n-p+2} \sss_{n-p-j} \ttt^{n-p-j} & \mbox{if} \quad j \leq n-p-2,
\end{array}
\end{cases}
$$
\end{small}
and a quick computation reveals that in all these cases one produces degenerate elements.

That the first summand $\ttt \sss_{n-p+2} \NN \SSS_\varphi$ is also degenerate follows by a similar argument, and this finishes the proof.
%
% To prove that $ \ttt_{n+1} \sss_n N \SSS_\psi$ is also degenerate, it suffices to show that $t^k_{n-p+2} \SSS_\psi$ 
% for $0 \leq k \leq n$ produces degenerate elements, 
% which is a tedious but straightforward computation as above which also uses \rmref{Sch38}; 
% we leave this to the terribly bored reader. Then $\sss_{n-p+2} \ttt^k_{n-p+2} \SSS_\psi$ produces elements with two $1$ appearing somewhere % % from which finally follows that $ \ttt_{n-p+3} \sss_{n-p+2} \ttt^k_{n-p+2} \SSS_\psi$ for $0 \leq k \leq n$ is degenerate, and hence $B\SSS_\psi$ is % as well.
\end{proof}

\subsection{The Cartan-Rinehart homotopy formula}
We are now in a position to state:

\begin{theorem}\label{calleelvira}
If $M$ is a module-comodule over a left Hopf algebroid $U$, then 
for any
cochain $\varphi \in \bar C_M^\bull(U)$ 
the homotopy formula
\begin{equation}
\label{sacromonte1}
\lie_\varphi = [\BB+\bb, \SSS_\varphi + \iota_\varphi] - \iota_{\gd\varphi} - \SSS_{\gd\varphi}
\end{equation}
holds on $\bar C^\cyc_\bull(U,M)$.
\end{theorem}

\begin{rem}
Observe that using \rmref{alles3} and \rmref{alles4}, this can be rewritten as
\begin{equation}
\label{sacromonte2}
\lie_\varphi = [\BB, \iota_\varphi] + [\bb, \SSS_\varphi] - \SSS_{\gd\varphi}.
\end{equation}
\end{rem}

\begin{rem}
Apart from the obvious classical Cartan homotopy \cite{Car:LTDUGDLEDUEFP}, 
this formula has been given in the context of associative algebras, i.e., in the classical cyclic homology of algebras, in \cite{Rin:DFOGCA} for the commutative case, 
in \cite{NesTsy:OTCROAA, Get:CHFATGMCICH} for the noncommutative situation, 
and in more restricted settings such as for $1$-cocycles in \cite{Goo:CHDATFL, Con:NCDG, Xu:NCPA}.
\end{rem}

\begin{proof}[Proof of Theorem \ref{calleelvira}]
We stress that throughout we work 
on $\bar C^\cyc_\bull(U,M)$. Rewrite first
\begin{equation*}
\begin{split}
[\BB, \iota_\varphi] + [\bb, \SSS_\varphi] - \SSS_{\gd\varphi} 
&= \BB\iota_\varphi - (-1)^p \iota_\varphi \BB + \bb \SSS_\varphi - (-1)^{p-2} \SSS_\varphi \bb  - \SSS_{\gd\varphi} \\
&= \BB\iota_\varphi + (-1)^{|p|} \iota_\varphi \BB + \bb \SSS_\varphi + (-1)^{|p|} \SSS_\varphi \bb  - \SSS_{\gd\varphi}.
\end{split}
\end{equation*}
Observe then that the statement in the cases $p > n+1$ and $p = n+1$ follows by definition. 
For $p < n+1$, let us write down \rmref{messagedenoelauxenfantsdefrance2}: 
\mycount=1
$$
\lie_\varphi = \ubs{(1)}{\sum^{n-|p|}_{i=1} (-1)^{\theta^{n,p}_{i}} \ttt^{n - |p| -i} \DD'_\varphi \ttt^{i+p}} 
+ \ubs{(2)}{\sum^{p}_{i=1} (-1)^{\xi^{n,p}_{i}} \ttt^{n-|p|} \DD'_\varphi \ttt^i},
$$
and also write with \rmref{parkraumbewirtschaftung1} and \rmref{parkraumbewirtschaftung2a}
on $\bar C^\cyc_\bull(U,M)$
\begin{eqnarray*}
\BB \iota_\varphi &=& {\sum^{n-p}_{k=0} (-1)^{k(n-p)} \sss_{-1} \ttt^k \dd_0 \DD'_\varphi} =: (3), \\
(-1)^{|p|}\iota_\varphi \BB &=& \sum^{n}_{k=0} (-1)^{|p|+nk} \ttt^{n-|p|} \DD'_\varphi \ttt^{k+1} \\
&=& \ubs{(4)}{\sum^{n}_{k=1} (-1)^{|p|+n|k|} \ttt^{n-|p|} \DD'_\varphi \ttt^{k}} + \ubs{(5)}{(-1)^{|p|} \ttt^{n-|p|} \DD'_\varphi}. 
\end{eqnarray*}
A lengthy computation using the simplicial and cyclic relations yields
\begin{footnotesize}
\begin{equation*}
\begin{split}
\bb \SSS_\varphi &= 
   \sum^{n-p}_{j=0} \, \sum^j_{i=0} (-1)^{\eta^{n,p}_{j,i}} \ttt^{n-p-i} \DD'_\varphi \ttt^{n+i-|j|} 
 + \sum^{n-p}_{j=0} \, \sum^j_{i=0} (-1)^{\eta^{n,|p|}_{|j|,|i|}+1} \sss_{-1} \ttt^{n-p-i} \dd_0 \DD'_\varphi \ttt^{n+i-|j|} \\
&\quad + \ubs{(6)}{\sum^{n-|p|}_{k = 2}  \, \sum^{k-1}_{{i=1}} \, \sum^{n-|p|}_{j=i} (-1)^{\eta^{n,p}_{|j|,|i|} + k -i} \sss_{-1}  \ttt^{n-p-i} \dd_k \DD'_\varphi \ttt^{n+i-|j|}} \\
&\quad+ \ubs{(7)}{\sum^{n-p}_{k = 1}  \, \sum^{n-p}_{j=k}  \, \sum^{j}_{i=k} (-1)^{\eta^{n,p}_{j,i} + k + n - |p| -i} \sss_{-1}  \ttt^{n-p-i} \dd_k \DD'_\varphi \ttt^{n+i-|j|}} \\
&\quad+ \ubs{(8)}{\sum^{n-p}_{j=0} (-1)^{\eta^{n,p}_{j,0} + n - p} \ttt^{n-|p|} \DD'_\varphi \ttt^{n-|j|}} 
 + \ubs{(9)}{\sum^{n-p-1}_{j=0} \, \sum^j_{i=0} (-1)^{\eta^{n,p}_{|j|,|i|}+n-p} \ttt^{n-p-i} \DD'_\varphi \ttt^{n+i-|j|}} \\
&= \ubs{(10)}{\sum^{n-p-1}_{j=0} \, \sum^j_{i=0} (-1)^{\eta^{n,p}_{j,i}} \ttt^{n-p-i} \DD'_\varphi \ttt^{n+i-|j|}} 
 +\ubs{(11)}{\sum^{n-|p|}_{i=1} (-1)^{\eta^{n,p}_{n-p,|i|}} \ttt^{n-|p|-i} \DD'_\varphi \ttt^{p+i}} \\ 
&\quad+ \ubs{(12)}{\sum^{n-p}_{j=1} \, \sum^{j-1}_{i=0} (-1)^{\eta^{n,|p|}_{|j|,|i|}+1} \sss_{-1} \ttt^{n-p-i} \dd_0 \DD'_\varphi \ttt^{n+i-|j|}} 
 + \ubs{(13)}{\sum^{n-p}_{i=0} (-1)^{\eta^{n,|p|}_{|i|,|i|}+1} \sss_{-1}  \ttt^{n-p-i} \dd_0 \DD'_\varphi} \\
&\quad+ (6) + (7) + (8) + (9).
\end{split}
\end{equation*}
\end{footnotesize}
Observe that by $(-1)^{\eta^{n,p}_{|j|,|i|}+n-p} = (-1)^{\eta^{n,p}_{j,i}+1}$ one has $(9) = - (10)$. Likewise, by 
$(-1)^{\eta^{n,p}_{n-p,|i|}} = (-1)^{\theta^{n,p}_{i}}$, we see that
$(11) = (1)$. By substitution $k:=n-p-i$, one obtains $(-1)^{k(n-p)} = (-1)^{\eta^{n,|p|}_{|i|,|i|}}$, and hence $(13) = -(3)$. Finally, $(2) = (4) + (5) + (8)$ by substitution of $i:=n-|j|$ in $(8)$.
We continue computing 
\begin{small}
\begin{equation*}
\begin{split}
(6) &= \sum^{n-|p|}_{k = 2} \, \sum^{k-1}_{i=1} \,  \sum^{n-p}_{j=i} (-1)^{\eta^{n,|p|}_{|j|,|i|} + k +1} \sss_{-1}  \ttt^{n-p-i} \dd_k \DD'_\varphi \ttt^{n+i-|j|} \\
& \qquad
+ \sum^{n-|p|}_{k = 2} \sum^{k-1}_{{i=1}} (-1)^{\eta^{n,|p|}_{n-p,|i|} + k +1} \sss_{-1}  \ttt^{n-p-i} \dd_k \DD'_\varphi \ttt^{p+i} \\
&= \ubs{(14)}{\sum^{n-p}_{k = 2}  \, \sum^{k-1}_{i=1} \, \sum^{n-p}_{j=i} (-1)^{\eta^{n,|p|}_{|j|,|i|} + k +1} \sss_{-1}  \ttt^{n-p-i} \dd_k \DD'_\varphi \ttt^{n+i-|j|}} \\
& \qquad
   + \ubs{(15)}{\sum^{n-p}_{k = 2} \sum^{k-1}_{{i=1}} (-1)^{\eta^{n,|p|}_{n-p,|i|} + k +1} \sss_{-1}  \ttt^{n-p-i} \dd_k \DD'_\varphi \ttt^{p+i}} \\
&\qquad + \ubs{(16)}{\sum^{n-p}_{j=1} \, \sum^{j}_{{i=1}} (-1)^{\eta^{n,|p|}_{|j|,|i|} + n + p} \sss_{-1}  \ttt^{n-p-i} \dd_{n-|p|} \DD'_\varphi \ttt^{n+i-|j|}} \\
 & \qquad
+ \ubs{(17)}{\sum^{n-p}_{{i=1}} (-1)^{\eta^{n,|p|}_{n-p,|i|} + n + p} \sss_{-1}  \ttt^{n-p-i} \dd_{n-|p|} \DD'_\varphi \ttt^{p+i}}. \\
\end{split}
\end{equation*}
\end{small}
With \rmref{parkraumbewirtschaftung-1} one sees 
%\begin{footnotesize}
$$
(15) = {\sum^{n-1}_{k = p+1} \sum^{k-p}_{{i=1}} (-1)^{\eta^{n,|p|}_{|p|,i} + k} \sss_{-1}  \ttt^{n-p-i} \DD'_\varphi \dd_k \ttt^{p+i}} =: (18),
$$
%\end{footnotesize}
and we also simplify 
\begin{footnotesize}
$$
(16) = \ubs{(19)}{\sum^{n-p-1}_{j=1} \, \sum^{j}_{{i=1}} (-1)^{\eta^{n,|p|}_{|j|,i}} \sss_{-1}  \ttt^{n-p-i} \dd_{n-|p|} \DD'_\varphi \ttt^{n+i-j}} 
+  \ubs{(20)}{\sum^{n-p}_{{i=1}} (-1)^{\eta^{n,|p|}_{i,i}} \sss_{-1}  \ttt^{n-p-i} \dd_{n-|p|} \DD'_\varphi}.
$$
\end{footnotesize}
Furthermore, 
\begin{footnotesize}
\begin{equation*}
\begin{split}
(7) &= \ubs{(21)}{\sum^{n-p}_{k = 2} \, \sum^{n-p}_{j=k} \, \sum^{j}_{i=k} (-1)^{\eta^{n,|p|}_{|j|,|i|} + k +1} \sss_{-1}  \ttt^{n-p-i} \dd_k \DD'_\varphi \ttt^{n+i-|j|}} \\
& \qquad
+ \ubs{(22)}{{\sum^{n-p}_{j=1} \, \sum^{j}_{i=1} (-1)^{\eta^{n,|p|}_{|j|,|i|}} \sss_{-1}  \ttt^{n-p-i} \dd_1 \DD'_\varphi \ttt^{n+i-|j|}}}.  
\end{split}
\end{equation*}
\end{footnotesize}
On the other hand, we have 
\begin{small}
\begin{equation*}
\begin{split}
(-1)^{|p|} \SSS_\varphi \bb &= \ubs{(23)}{\sum^{n-p}_{i=1}  (-1)^{\eta^{|n|,p}_{|i|,|i|} + n + |p|} \sss_{-1} \ttt^{n-p-i} \DD'_\varphi \dd_n} 
\\ & \qquad
+  \ubs{(24)}{\sum^{n-1}_{k = 0} \, \sum^{n-p}_{j=1} \, \sum^{j}_{i=1}  (-1)^{\eta^{|n|,p}_{|j|,|i|} + k + i - j + |p|} \sss_{-1}  \ttt^{n-p-i} \DD'_\varphi \dd_k \ttt^{n+i-|j|}} 
\\ & \qquad 
+  \ubs{(25)}{\sum^{n-p-1}_{i=1} \, \sum^{n-1}_{k=p+i}  (-1)^{\eta^{|n|,p}_{|n-p|,|i|} + k - |i|} \sss_{-1}  \ttt^{n-p-i} \DD'_\varphi \dd_k \ttt^{p+i}}, 
\end{split}
\end{equation*}
\end{small}
and we directly observe that $(23) = - (20)$ and $(25) = - (18)$, whereas
\begin{small}
\begin{equation*}
\begin{split}
(24) &= \ubs{(26)}{\sum^{p}_{k = 1} \, \sum^{n-p}_{j=1} \, \sum^{j}_{i=1}  (-1)^{\eta^{n,|p|}_{|j|,i} + k +1} \sss_{-1}  \ttt^{n-p-i} \DD'_\varphi \dd_k \ttt^{n+i-|j|}} \\
&\qquad +  \ubs{(27)}{\sum^{n-1}_{k = p+1} \, \sum^{n-p}_{j=1} \, \sum^{j}_{i=1}  (-1)^{\eta^{n,|p|}_{|j|,i} + k +1} \sss_{-1}  \ttt^{n-p-i} \DD'_\varphi \dd_k \ttt^{n+i-|j|}} \\
&\qquad +  \ubs{(28)}{\sum^{n-p}_{j=1} \, \sum^{j}_{i=1}  (-1)^{\eta^{n,|p|}_{|j|,i} +1} \sss_{-1}  \ttt^{n-p-i} \DD'_\varphi \dd_0 \ttt^{n+i-|j|}}, 
\end{split}
\end{equation*}
\end{small}
where by the cyclic relations 
\begin{footnotesize}
\begin{equation*}
(28) = \ubs{(29)}{\sum^{n-p-1}_{j=1} \, \sum^{j}_{i=1}  (-1)^{\eta^{n,|p|}_{|j|,i} +1} \sss_{-1}  \ttt^{n-p-i} \DD'_\varphi \dd_n \ttt^{n+i-j}} 
+  \ubs{(30)}{\sum^{n-p}_{i=1}  (-1)^{\eta^{n,|p|}_{p,i} +1} \sss_{-1}  \ttt^{n-p-i} \DD'_\varphi \dd_n \ttt^{p+i}}. 
\end{equation*}
\end{footnotesize}
By means of \rmref{parkraumbewirtschaftung-1}, one now sees that $(14) + (21) = - (27)$ and that $(29) = - (19)$, 
along with $(30) = - (17)$. 

To conclude the proof, we need to show that $\SSS_{\gd \varphi}$ equals the only remaining terms $(12)$, $(22)$, and $(26)$. 
Note first that from \rmref{mulhouse2}, \rmref{parkraumbewirtschaftung-1}, \rmref{parkraumbewirtschaftung1}, 
as well as from the cyclic and simplicial relations follows for the $(p+1)$-cochain $\gd \varphi$:
\begin{footnotesize}
\begin{equation*}
\begin{split}
\DD'_{\gd \varphi} &= \ttt\iota_{\gd \varphi} \sss_{-1} \ttt^n 
                 =  \ttt \bb \iota_{\varphi} \sss_{-1} \ttt^n + (-1)^{|p|}  \ttt\iota_{\varphi}\bb \sss_{-1} \ttt^n \\
                &= \sum^{n-|p|+1}_{k=1} (-1)^{|k|} \ttt \dd_0 \dd_k \DD'_\varphi \sss_{-1} \ttt^n + \sum^{n+1}_{k=0} (-1)^{k+|p|} \ttt \dd_0 \DD'_\varphi \dd_k \sss_{-1} \ttt^n \\
           &= \ttt \dd_0 \dd_1 \DD'_\varphi \sss_{-1} \ttt^n + \sum^{p}_{k=0} (-1)^{k+|p|} \ttt \dd_0 \DD'_\varphi \dd_k \sss_{-1} \ttt^n \\
           &= \ttt \dd_0 \iota_\varphi \sss_{-1} \ttt^n + (-1)^{|p|} \ttt \dd_0 \DD'_\varphi \ttt^n  + \sum^{p}_{k=1} (-1)^{k+|p|} \ttt \iota_\varphi \sss_{-1} \dd_{k-1} \ttt^n \\
           &= \ttt^{n-p+1} \dd_1 \DD'_\varphi  + (-1)^{|p|} \ttt \dd_0 \DD'_\varphi \ttt^n  + \sum^{p}_{k=1} (-1)^{k+|p|} \ttt^{n-p+1} \DD'_\varphi \dd_{k}. \\
\end{split}
\end{equation*}
\end{footnotesize}
Hence we have for the $(p+1)$-cochain $\gd \varphi$:
\begin{footnotesize}
\begin{equation*}
\begin{split}
\SSS_{\gd \varphi}  &= \sum^{n-(p+1)}_{j=0} \, \sum^{j}_{i=0}  (-1)^{\eta^{n,|p|}_{j,i}} \sss_{-1}  \ttt^{n-p-(i+1)} \dd_1 \DD'_\varphi \ttt^{n+i-|j|} \\
&\quad +  \sum^{p}_{k = 1} \, \sum^{n-(p+1)}_{j=0} \, \sum^{j}_{i=0}  (-1)^{\eta^{n,|p|}_{j,i} + k + |p|} \sss_{-1}  \ttt^{n-p-(i+1)} \DD'_\varphi \dd_k \ttt^{n+i-|j|} \\
& \quad + \sum^{n-(p+1)}_{j=0} \, \sum^{j}_{i=0}  (-1)^{\eta^{n,|p|}_{j,i} + |p|} \sss_{-1}  \ttt^{n-p-i} \dd_0 \DD'_\varphi \ttt^{n+i-j} \\
 &= \sum^{n-p}_{j=1} \, \sum^{j}_{i=1}  (-1)^{\eta^{n,|p|}_{|j|,|i|}} \sss_{-1}  \ttt^{n-p-i} \dd_1 \DD'_\varphi \ttt^{n+i-|j|} \\
&\quad +  \sum^{p}_{k = 1} \, \sum^{n-p}_{j=1} \, \sum^{j}_{i=1}  (-1)^{\eta^{n,|p|}_{|j|,|i|} + k + |p|} \sss_{-1}  \ttt^{n-p-i} \DD'_\varphi \dd_k \ttt^{n+i-|j|} \\
& \quad + \sum^{n-(p+1)}_{j=0} \, \sum^{j}_{i=0}  (-1)^{\eta^{n,|p|}_{j,|i|} +1} \sss_{-1}  \ttt^{n-p-i} \dd_0 \DD'_\varphi \ttt^{n+i-j}, 
\end{split}
\end{equation*}
\end{footnotesize}
and these summands are exactly the terms $(22)$, $(26)$, and $(12)$, which concludes the proof of \rmref{sacromonte2} and hence of \rmref{sacromonte1}.
\end{proof}

With the help of the homotopy formula, we can easily 
prove:

\begin{corollary}
\label{lacueva}
For any cochain $\varphi \in \bar C_M^\bull(U)$, 
we have on $\bar 
C^\cyc_\bull(U,M)$ 
\begin{equation}
\label{alles2}
        [\lie_\varphi, \BB] = 0.
\end{equation}
\end{corollary}

\begin{proof}
Using\label{hieralso}  
\rmref{sacromonte2}, \rmref{alles4}, and 
\rmref{mixedcomp}, 
we see by the graded Jacobi identity that
\begin{equation*}
\begin{split}
%[\bb, \lie_\varphi] &= [\bb, [\BB, \iota_\varphi]]  + [\bb, [\bb, \SSS_\varphi]] - [\bb, \SSS_{\gd\varphi}] \\
%           &= (-1)^p [\BB, [\iota_\varphi, \bb]]  - [\bb, \SSS_{\gd\varphi}] \\         
%           &= - [\BB, [\bb, \iota_\varphi]]  - [\bb, \SSS_{\gd\varphi}] \\         
%           &= - [\BB, \iota_{\gd\varphi}]  - [\bb, \SSS_{\gd\varphi}] \\         
%           &= - \lie_{\gd \varphi},
[\lie_\varphi, \BB] &= [[\BB, \iota_\gvf], \BB] + [[\bb, \SSS_\gvf], \BB] - [\SSS_{\gd\gvf}, \BB] \\
           &= [\BB, [\iota_\gvf, \BB]] - (-1)^p [\iota_\gvf, [\BB, \BB]] + [\bb, [\SSS_\gvf, \BB]] - (-1)^{p-2} [\SSS_\gvf, [\bb, \BB]] \\
           &= 0,
\end{split}
\end{equation*}
 where the fact that $ [\BB, [\iota_\gvf, \BB]] = 0$ directly follows from the graded Jacobi identity. 
%This proves \rmref{alles2}. 
%Similarly, \rmref{alles2} follows by considering the degrees in the graded Jacobi identity and using %\rmref{alles4}.
\end{proof}

\begin{rem}
With some more effort, it can be shown that \rmref{alles2} even holds on
the non-reduced complex, but we do not need this.
\end{rem}

\subsection{Proof of Theorem~\ref{pen1}}
If $ \varphi \in \bar C^\bull_M(U)$ is a cocycle, then for
the induced maps 
$$
        \lie_\varphi: \hhmu \to 
        H^M_{\bull-|p|}(U), \quad  \quad 
        \iota_\varphi: \hhmu \to 
        H^M_{\bull-p}(U), 
$$
the 
Rinehart homotopy formula 
\rmref{sacromonte1} simplifies to
$$
        \lie_\gvf = [\BB, \iota_\gvf].
$$
Using this and \rmref{mulhouse1} one has
\begin{corollary}
For cocycles $\varphi, \psi \in \bar C_M^\bull(U)$, 
the induced maps on ${\hhmu}$ obey
\begin{equation*}
\label{radicale2}
\lie_{\varphi \smallsmile \psi}  = \lie_\varphi  \iota_\psi + (-1)^{\deg \gvf} \iota_\varphi  \lie_\psi. 
\end{equation*}
\end{corollary}
\begin{proof}
This is now only one line:
\begin{equation*}
        \lie_{\varphi \smallsmile \psi}  = 
            [\BB, \iota_{\varphi \smallsmile \psi}]  = 
            [\BB, \iota_{\varphi}] \iota_\psi + 
            (-1)^{\deg\gvf} \iota_\gvf [\BB, \iota_\psi] 
            = 
            \lie_{\varphi} \iota_\psi + 
            (-1)^{\deg\gvf} \iota_\gvf \lie_\psi.
 \qedhere
  \end{equation*}
\end{proof}

We now sum up the results of Theorems \ref{feinefuellhaltertinte}, \ref{waterbasedvarnish}, and \ref{calleelvira},  
and state the main theorem 
(cf.~Theorem \ref{pen1}) of this paper:

\begin{theorem}
\label{pen3}
If $U$ is a left Hopf algebroid over $A$, 
and $M$ is a module-comodule,
then $\iota$ given in 
\rmref{alles4} and the 
Lie derivative $\lie$ given in
\rmref{messagedenoelauxenfantsdefrance2} 
turn 
${\hhmu}$ 
into a Batalin-Vilkovisky module 
over the Gerstenhaber algebra 
$\hmu$ 
defined by Theorem~\ref{pitandbull2}.
\end{theorem}

\begin{rem}
A natural question is to what extent and in which sense 
the above structures lift to the (co)chain level.   
For the Gerstenhaber algebra structure on Hochschild cohomology this
is the content of Deligne's conjecture, which asserts that the
Hochschild cochain complex $C^\bull(\Ae,A)$ is 
an algebra over an operad (in the category of cochain complexes of
$k$-modules) that is quasi-isomorphic to the chain little discs operad 
(see
e.g.~\cite{DwyHes:LKAMBO,GerVor:HGAAMSO,KonSoi:DOAOOATDC,McCSmi:ASODHCC,Tam:OPOMKFT}
or \cite[\S13.3.19]{LodVal:AO}).
Kontsevich and Soibelman have
extended the scope of Deligne's conjecture to the full differential
calculus structure on Hochschild (co)homology \cite[Theorem~11.3.1]{KonSoi:NOAA}.     
As the referee of the present paper remarked, one should expect our
(co)chain complexes to be in general algebras over the coloured
operad constructed therein, or over a
quasi-isomorphic one.  
\end{rem}

\section{Lie-Rinehart algebras and jet spaces}\label{ganzhuebsch}

This section contains a brief sketch of how to generalise the 
above results 
to complete left Hopf algebroids (the Hopf algebroid generalisation of
complete Hopf algebras, see e.g.~\cite{Qui:RHT}), and how this  
allows one to obtain the 
well-known calculus for Lie-Rinehart algebras (Lie algebroids)
given by the Lie derivative,
insertion operator, and the de Rham differential 
(cf.~the original reference \cite{Rin:DFOGCA} and also, for
example, \cite{Hue:DFLRAATMC,GraUrb:AGDCOVB,Hue:LRAGAABVA, Kos:EGAALBA,
  Xu:GAABVAIPG}), and in particular the classical  
Cartan calculus from differential geometry 
that arises as the special
case of the tangent Lie
algebroid (see \cite{Car:LTDUGDLEDUEFP}).

In \S\ref{magsein} we introduce the jet space 
$J\!L$ of a Lie-Rinehart
algebra (\cite{KowPos:TCTOHA}, see also \cite{CalRosVdB:HCFLA}), 
and explain
its complete Hopf algebroid structure. Then we sketch in 
\S\ref{cjl} 
how to adapt the constructions of this paper to this setting. 
Finally, in the last two sections we recall the definition of the 
generalised Hochschild-Kostant-Rosenberg morphisms and use them to
relate the differential calculus of
Theorem~\ref{pen1} to the standard one on the exterior
algebras of $L$ respectively $L^*$ that gives rise to
Lie-Rinehart cohomology.

%\subsection{$V\!L$ and $J\!L$}
\subsection{Universal enveloping algebras and jet spaces}\label{magsein}
 
Let $(A,L)$ be a Lie-Rinehart algebra over a commutative $k$-algebra
$A$ with anchor map $L \to \Der_k(A), X \mapsto \{a \mapsto X(a)\}$,  
and $V\!L$ be its universal enveloping algebra 
(see \cite{Rin:DFOGCA} for details). 
This is naturally a left Hopf algebroid, see 
e.g.~\cite{KowKra:DAPIACT};  
as therein, we denote by the same symbols elements $a \in A$ and $X \in L$ and the corresponding generators in $V\!L$.
The source and target maps $s = t$ are equal to the canonical injection $A \to V\!L$. The coproduct and the counit are given by 
\begin{equation}
\label{ayran}
\begin{array}{rclrcl}
        \gD(X) 
&:=& X \otimes_\ahha 1 + 1 \otimes_\ahha X, 
        & \qquad \gve(X) 
&:=& 0, \\
        \gD(a) &:=& 
a \otimes_\ahha 1, &                 \qquad \gve(a) &:=& a,
\end{array}
\end{equation}
whereas the inverse of the Hopf-Galois map is
\begin{equation}
\label{marmorkuchen}
X_+ \otimes_\Aopp X_- := X \otimes_\Aopp 1 - 1 \otimes_\Aopp X, \qquad a_+ \otimes_\Aopp a_- := a \otimes_\Aopp 1,
\end{equation}
where we retain the notation $\otimes_\Aopp$ for the tensor product $\due {V\!L} \blact {} \otimes_\Aopp \due {V\!L} {} \ract$ although $A$ is commutative.
By universality,  
these maps can be extended to $V\!L$.

\begin{definition}
The $A$-linear dual $J\!L := \Hom_\ahha(V\!L,A)$ is called 
the {\em jet space} of $(A,L)$.   
\end{definition}

By duality, $J\!L$  
carries a commutative $\Ae$-algebra structure with product
\begin{equation}
\label{kaesekuchen1}
        (fg) (u) = 
        f(u_{(1)}) g(u_{(2)}), \qquad 
        f, g \in J\!L, \ u \in V\!L,
\end{equation}
unit given by the counit $\gve$ of $V\!L$, 
and source and target maps given by 
\begin{equation}
\label{sarare}
        s(a)(u) := a \gve(u) = \varepsilon (au), \qquad 
        t(a)(u) := \gve(ua), \qquad 
        a \in A, u \in V\!L.
\end{equation}
Observe that these do not coincide although $A$ is commutative. 

The $\Ae$-algebra $J\!L$ is complete with respect to 
the (topology defined by the) decreasing 
filtration whose degree $p$ part consists of those 
functionals that vanish on the $A$-linear span $(V\!L)_{\le p} \,{\subseteq}\,
V\!L$ of all 
monomials in up to $p$ elements of $L$. For 
finitely generated projective $L$,
Rinehart's generalised PBW theorem \cite{Rin:DFOGCA}
identifies $J\!L$ with the completed symmetric algebra 
of the $A$-module $L^* = \Hom_\ahha(L,A)$. 

\begin{example}
The simplest example beyond Lie algebras is
$A=k[x],L=\mathrm{Der}_k(A)$, in which case 
$L$ is generated as an $A$-module by 
$p:=\frac{d}{dx}$. Then $V\!L$ is isomorphic to the first Weyl
algebra. In particular, there is an $A$-algebra isomorphism
$J\!L \simeq A \Lbrack h \Rbrack$ under which  
$h^i$ corresponds to the $A$-linear functional 
on $A[p]$ that maps $p^j$ to $ \delta _{ij} \in A$. 
Here $J\!L$ is considered as $A$-algebra via the source map $s$ which becomes 
under the isomorphism the standard unit map of $A \Lbrack h
\Rbrack$. However, the target map $t$ maps a
polynomial $a \in A$ to the power series given by its jet
$$
        t(a)=a+\frac{da}{dx} h+\frac{d^2a}{dx^2} h^2 + \cdots. 
$$    
\end{example}   

The filtration of $J\!L$ induces one
of $J\!L \otimes_\ahha J\!L$ and
if we denote by $J\!L
\hat\otimes_\ahha J\!L$ 
the completion, then  
the product of $V\!L$ yields a coproduct
$\gD: J\!L \to J\!L \hat\otimes_\ahha J\!L$ determined by 
\begin{equation}
\label{kaesekuchen2}
        f(uv) =: \gD(f)(u \otimes_\Aopp v) = f_{(1)}(u f_{(2)}(v)),
\end{equation}
see Lemma 3.16 in \cite[\S3.4]{KowPos:TCTOHA}. 
This is part of a {\em complete Hopf algebroid} structure on $J\!L$. 
We refer to \cite[Appendix A]{Qui:RHT} for
complete Hopf algebras, the Hopf algebroid generalisation is
straightforward. 
The counit of $J\!L$ is given by 
$f \mapsto f(1_{\scriptscriptstyle{V\!L}})$, and 
the antipode is 
\begin{equation}
\label{kaesekuchen3}
        (Sf)(u) := \varepsilon (u_+ f(u_-)), \qquad u \in V\!L, f \in J\!L,
\end{equation}
which for $u\in L \subseteq V\!L$ 
is known under the name {\em Grothendieck
  connection}. A short computation gives
$
S^2 = \mathrm{id}.
$
The translation map
\rmref{pm} is   
\begin{equation}
\label{apfelorangeingwersaft}
f_+ \hat\otimes_\Aopp f_- := f_{(1)} \hat\otimes_\Aopp S(f_{(2)}). 
\end{equation}

Note that
$J\!L$ is not only a left but a 
{\em full} complete Hopf algebroid 
in the sense of B\"ohm and
Szlach\'anyi \cite{Boe:HA}. Over noncommutative base algebras 
this would generally require {\em two}
bialgebroid structures that coincide here. In particular, 
$J\!L$ is also a commutative
Hopf algebroid %over a commutative base ring 
in the narrower sense studied already for decades
\cite{Hov:HTOCOAHA,Rav:CCASHGOS}.

%that the Lie-Rinehart algebra $(A,L)$ gives rise to 
%another (left) Hopf algebroid: define (cf.~\cite{NesTsy:DOSLADOHSSAIT%, CalVdB:HCAAC}) the {\em space of $p$-jets} of $(A,L)$ by
%$$
%J^p\!L := \Hom_\ahha(V\!L_{\leq p}, A),
%$$
%where $V\!L_{\leq p}$ denotes the elements of degree $p$ or less with %respect to the canonical PBW filtration \cite{Rin:DFOGCA}. The {\em jet% space} is then defined as 
%$$
%J\!L = \underset{\leftarrow}{\lim} \, J^p\!L,
%$$
%and we have $J\!L = \Hom_\ahha(V\!L,A)$ by $V\!L = \underset{\rightar%row}{\lim} \, V\!L_{\leq p}$. Hence, by definition, $J\!L$ is complete wit%h respect to the canonical PBW filtration. In the remainder of this sectio%n, we will therefore always complete tensor products using this filtration% (cf.~\cite{Qui:RHT}).

\subsection{$C^\bull(J\!L,A)$ and $C_\bull(J\!L,A)$}\label{cjl}
For complete Hopf algebroids such as $J\!L$, 
the theory developed in this paper 
needs to be modified as follows, in order for the structure maps (e.g.~the
cyclic operator $\ttt$) to be well-defined:
in $P_\bull$ and in the chain complex 
$C_\bull(J\!L,M)$, the completed tensor products have to be used. 
Similarly, in the definition of a module-comodule and of an SaYD module the
coaction might be given by maps $M \rightarrow J\!L
\hat\otimes_\ahha M$.

Dually, $C^\bull(J\!L,A)$ has to be defined as
$\Hom_\Aop^\mathrm{cont}({J\!L^{\hat\otimes_\Aopp \bull}}_\ract, A)$,
where $\mathrm{cont}$ means that the cochains 
have to be continuous ($A$
being discrete), as only the operators assigned to these cochains will
be well-defined on the completed tensor products.

Unlike for general left Hopf algebroids, 
we have for $J\!L$ canonical homology
coefficients: 
using that $J\!L$ is commutative, one easily 
verifies that $A$ carries a natural structure of an SaYD module 
over $J\!L$ whose action and coaction are given by
\begin{equation}
\label{solimando}
\begin{array}{rclrcl}
A \otimes J\!L & \to & A, & \quad (a,f) & \mapsto & a\varepsilon (f) , \\
A  & \to &  J\!L \otimes_\ahha A, & \quad a  & \mapsto & s(a) \otimes_\ahha 1_\ahha, 
\end{array}
\end{equation}
where $s$ is the source map from \rmref{sarare}. Hence 
Theorem~\ref{pen1} yields a canonical differential
calculus $(H^\bull(J\!L,A),H_\bull(J\!L,A))$ 
associated to any Lie-Rinehart algebra $(A,L)$ that we want to
discuss in more detail as an illustration of the abstract theory.

\subsection{Lie-Rinehart (co)homology}\label{lrc}
In order to do so, recall that the space 
$
        \Hom_\ahha(\textstyle\bigwedge^\bull_\ahha \!L,A) 
        %\simeq
        %\textstyle\bigwedge^\bull_\ahha \!L^*,\quad
        %L^* := \Hom_\ahha(L,A)
$
of alternating $A$-multilinear forms  
is a cochain complex of $k$-modules with respect to
% the {\em Rinehart-de Rham differential} 
$$
        \mathsf{d}: \Hom_\ahha(\textstyle\bigwedge^n_\ahha \!L, A) \to \Hom_\ahha(\textstyle\bigwedge^{n+1}_\ahha \!L, A)
$$ 
given by (where the terms $\hat X^i$ are omitted)
\begin{equation}
\label{fassbinder}
\begin{split}
\mathsf{d}\go(X^0, \ldots, X^n) &:= \sum^n_{i=0} (-1)^i X^i\big(\go(X^0, \ldots,  \hat{X}^i, \ldots, X^n)\big) \\
&\quad + \sum_{i < j} (-1)^{i+j} \go([X^i, X^j], X^0, \ldots, \hat{X}^i, \ldots, \hat{X}^j, \ldots,  X^n).
\end{split}
\end{equation}

In case $(A,L)$ arises from a Lie algebroid $E$, the above is the
complex of $E$-differential forms (see, for example,
\cite{CanWei:GMFNCA}), and in case $E$ is the tangent bundle of a
smooth manifold, these are the conventional differential forms that
appear in differential geometry. 

\begin{definition}
$H^\bull(\Hom_\ahha(\textstyle\bigwedge_\ahha \!L,A),\mathsf{d})$  
is called the 
{\em
  Lie-Rinehart cohomology} of $L$.  
\end{definition}
From \cite[Theorem 3.21]{KowPos:TCTOHA} we gather that there is a
morphism of chain complexes  
\begin{equation}
\label{schreibwarenhandlung1}
        F: \big(\bar C_\bull(J\!L,A), \bb\big) \to 
        \big(\Hom_\ahha(\textstyle\bigwedge^\bull_\ahha \!L,A), 0\big)
\end{equation}
given in degree $n$ by
$$
        F(f^1, \ldots, f^n)(X^1 \wedge \cdots \wedge X^n) 
:= (-1)^n \big(Sf^1 \wedge \cdots \wedge Sf^n\big)(X^1, \ldots, X^n).
$$
Here $Sf^1 \wedge \cdots \wedge Sf^n$ is the wedge product of
alternating multilinear forms. As
$C_\bull(J\!L,A)$ is defined via completed tensor
products, we have
\begin{equation}\label{inwers}
        C_n(J\!L,A) \simeq \varprojlim
        \Hom_\ahha\big((V\!L^{\otimes_\ahha n})_{\leq p},A \big),  
  \end{equation}
where $(V\!L^{\otimes_\ahha n})_{\leq p}$ is the degree $p$ part of
the filtration induced by that of $V\!L$.
The antipodes appear above as 
this isomorphism (\ref{inwers}) is given by
\begin{equation}\label{cityyildiz}
        (f^1, \ldots, f^n)(u^1, \ldots, u^n) := 
        Sf^1(u^1) \cdots Sf^n(u^n).
\end{equation}
That $F$ is well-defined on the reduced complex 
$\bar C_\bull(J\!L,A)$ follows 
since degenerate chains vanish under $F$ as
\rmref{Sch48} gives for $X \in L$
\begin{equation}
\label{schonwiederzuweniggegessenundzuvielKaffeegetrunken}
        \gve(X_+ 1_{\scriptscriptstyle{{J\!L}}}(X_-)) = 
        \gve(X_+ \gve(X_-)) = 
        \gve(X) = 0.
\end{equation}

When $L$ is finitely generated projective
over $A$, the wedge product of multilinear forms provides an
isomorphism 
$$
                \textstyle\bigwedge^\bull_\ahha \!L^* \rightarrow 
                \Hom_\ahha(\textstyle\bigwedge^\bull_\ahha \!L,A) 
$$
that we suppress in the sequel. Furthermore, the pairing
\rmref{cityyildiz} yields an isomorphism
(cf.~\cite[Eq.~(4.10)]{CalRosVdB:HCFLA})
\begin{equation}\label{zerolimits}
        C^n(J\!L,A) \simeq V\!L^{\otimes_\ahha n}.
\end{equation}
Finally, if we denote by $\pr: VL \to L$ the projection on $L$ resulting
from Rinehart's PBW theorem, we have:
\begin{prop}
Assume that $L$ is finitely generated projective over $A$ and 
define  
$$
        F'(\ga^1 \wedge \cdots \wedge \ga^n) := 
        \sum_{\gs \in S_n}(-1)^\gs \big(\pr^*\ga^{\gs(1)}, 
        \ldots, \pr^*\ga^{\gs(n)}\big)
$$
 for $\ga^1,\ldots,\ga^n \in L^*$. Then we have
$$
        FF'=n! \,\mathrm{id}_{\scriptscriptstyle \bigwedge^n_\ahha
          \!L^{\scriptscriptstyle \ast}}.
$$
In particular, if $ \mathbb{Q} \,{\subseteq}\,k$, then the morphism 
$F$ has a right inverse.
%$$
%        \frac{1}{\bull!} F : \big(C_\bull(J\!L,A), \bb,\BB\big) \to 
%        \big(\Hom_\ahha(\textstyle\bigwedge^\bull_\ahha \!L,A), 0,d\big)
%$$
%is a morphism of mixed complexes.
\end{prop} 
\begin{proof}
This follows by straightforward computation, using that 
\rmref{marmorkuchen} yields 
\begin{equation}
\label{tatue}
S(\pr^* \ga) = - \pr^* \ga
\end{equation}
for every $1$-form $\ga \in L^*$. 
%The second claim has been proven in \cite[Theorem 3.21]{KowPos:TCT%OHA}.   
\end{proof} 

%
%Combining this with \cite[Prop.~4.3]{KowKra:CSIACT}, we obtain
%\begin{equation}
%\label{ziamaria}
%\Tor^{J\!L}_\bull(A,A) \simeq \textstyle\bigwedge^\bull_\ahha \!L^*.
%\end{equation}

Dual to \rmref{schreibwarenhandlung1}, one has a
morphism
\begin{equation}
\label{schreibwarenhandlung2}
F^*: \big(\!\textstyle\bigwedge^\bull_\ahha\!L, 0\big)  \to 
(\bar C^\bull(J\!L,A),\delta ) 
%\big(\Hom_\ahha(J\!L^{\otimes_\Aopp \bull},A), \delta\big) 
\end{equation}
of cochain complexes explicitly given as 
$$
X^1 \wedge \cdots \wedge X^n \mapsto \big\{ (f^1, \ldots, f^n) \mapsto  (-1)^n \!\! \sum_{\gs \in S_n} (-1)^\gs (Sf^1)(X^{\gs(1)}) \cdots (Sf^n)(X^{\gs(n)}) \big\}.
$$
%and which, using a dualised version of \cite[Theorem 3.21]{KowPos:TCT%OHA},
%induces an isomorphism
%\begin{equation}
%\label{parkuhr}
%\textstyle\bigwedge^\bull_\ahha \!L \simeq \Ext^\bull_{J\!L}(A,A),
%\end{equation}
%where the left $J\!L$-action on $A$ is still given as in \rmref{solimando}. 

\subsection{The calculus structure for Lie-Rinehart algebras}
Our main aim is to use now $F,F^*$, and $F'$ to 
compare the calculus structure on 
$\big(H^\bull(J\!L,A),H_\bull(J\!L,A)\big)$ resulting from (the topological version of) 
Theorem~\ref{pen1} with the well-known 
calculus on $\big(\textstyle\bigwedge^\bull_\ahha \!L,
\textstyle\bigwedge^\bull_\ahha \!L^*\big)$ 
given by the exterior differential, the insertion operator, the Lie
derivative for differential forms, along with the classical Cartan
homotopy formula (see \cite{Rin:DFOGCA,
  Hue:DFLRAATMC, Hue:LRAGAABVA}, or
\cite{CanWei:GMFNCA, Kos:EGAALBA, Xu:GAABVAIPG} for the case of 
Lie algebroids and in particular the original reference
\cite{Car:LTDUGDLEDUEFP} for the tangent bundle of a smooth
manifold). First, recall that these operators, besides $\dd$ from \rmref{fassbinder}, are given by
\begin{eqnarray*}
%d: {\textstyle{\bigwedge^n_\ahha \!L^*}} &\to& {\textstyle{\bigwedge^{n+1}_\ahha \!L^*}}, \\
\mathsf{i}_X: {\textstyle{\bigwedge^n_\ahha \!L^*}} &\to& {\textstyle{\bigwedge^{n-1}_\ahha \!L^*}}, \quad \go \mapsto \go(\cdot, \ldots, X), \\
%\end{eqnarray*}
%and
%\begin{eqnarray*}
\mathsf{L}_X: {\textstyle{\bigwedge^n_\ahha \!L^*}} &\to& {\textstyle{\bigwedge^{n}_\ahha \!L^*}},  \quad \mathsf{L}_X\go(Y^1, \ldots, Y^n) := X\big(\go(Y^1, \ldots, Y^n)\big) \\
&& \hspace*{4.5cm} \quad - \sum^n_{i=1} \go(Y^1, \ldots, [X, Y^i] , \ldots,  Y^n), 
\end{eqnarray*}
where $Y^1, \ldots, Y^n \in L$.
%% $$
%% \mathsf{L}_X\go(Y^1, \ldots, Y^n) := X\big(\go(Y^1, \ldots,  \hat{Y}^i, \ldots, Y^n)\big) 
%% - \sum^n_{i=1} \go(Y^1, \ldots, [X, Y^i] , \ldots,  Y^n). 
%% $$

% \begin{corollary}
% If $\mathbb{Q} \,{\subseteq}\,k$, then
% $$
%        \frac{1}{\bull!} F : \big(C_\bull(J\!L,A), \bb,\BB\big) \to 
%        \big(\Hom_\ahha(\textstyle\bigwedge^\bull_\ahha \!L,A), 0,d\big)
% $$
% is a morphism of mixed complexes with right inverse $F'$.
% \end{corollary}

Let us then consider the Gerstenhaber bracket on
$
C^\bull(J\!L,A) \simeq V\!L^{\otimes_\ahha \bull}.
$
Now, $V\!L^{\otimes_\ahha n}$ 
carries a canonical comp algebra structure given by
\begin{equation}
\label{rosenthalerplatz}
\begin{split}
& (u^1 \otimes_\ahha \cdots \otimes_\ahha u^p) \circ_i^{\scriptscriptstyle{\rm tens}} 
(v^1 \otimes_\ahha \cdots \otimes_\ahha v^q) \\
& := (u^1 \otimes_\ahha \cdots \otimes_\ahha u^{i-1} \otimes_\ahha u^i_{(1)} v^1 \otimes_\ahha \cdots \otimes_\ahha  u^i_{(q)} v^q \otimes_\ahha u^{i+1} \otimes_\ahha \cdots \otimes_\ahha u^p, 
\end{split}
\end{equation}
for $i = 1, \ldots, p$, and
where $\Delta^q(u) = u_{(1)} \otimes_\ahha \cdots \otimes_\ahha u_{(q)}$ is the iterated coproduct (where $\Delta^0 := \gve$ and $\Delta^1 := \id$). This is a slight generalisation to bialgebroids from a statement in \cite[p.~65]{GerSch:ABQGAAD}, and the expression is well defined with \rmref{tellmemore}. 

In the first part of the following proposition we state that \rmref{rosenthalerplatz} corresponds to our general expression \rmref{maxdudler} of the Gerstenhaber products by means of the isomorphism \rmref{zerolimits}, and in particular that the resulting Gerstenhaber bracket
 corresponds to the classical Schouten-Nijenhuis bracket on the exterior algebra $\textstyle\bigwedge^\bull_\ahha \!L$. In the second part, we show how the relevant operators from the two mentioned calculi are connected to each other; for the sake of simplicity we restrict to the case where one acts with an element $X \in L = \textstyle\bigwedge^1_\ahha \!L$:

\begin{prop}
If $L$ is finitely generated projective over $A$, then 
for 
%$i = 1, \ldots, p$, 
$1 \leq i \leq p$
one has
$$
(u^1 \otimes_\ahha \cdots \otimes_\ahha u^p) \circ_i
(v^1 \otimes_\ahha \cdots \otimes_\ahha v^q) = (u^1 \otimes_\ahha \cdots \otimes_\ahha u^p) \circ_i^{\scriptscriptstyle{\rm tens}} 
(v^1 \otimes_\ahha \cdots \otimes_\ahha v^q),
$$
where the left hand side is the Gerstenhaber product from \rmref{maxdudler}. In particular, if $\mathbb{Q} \,{\subseteq}\,k$, then the Gerstenhaber bracket from \rmref{zugangskarte} corresponds to the classical Schouten-Nijenhuis bracket by means of the map $\frac{1}{n!}F^*$ from \rmref{schreibwarenhandlung2}.

Furthermore, 
for the operations $\mathsf{d}$, $\mathsf{i}_X$, and $\mathsf{L}_X$
%\begin{eqnarray*}
%d: {\textstyle{\bigwedge^n_\ahha \!L^*}} &\to& {\textstyle{\bigwedge^{n+1}_\ahha \!L^*}}, \\
%\mathsf{i}_X: {\textstyle{\bigwedge^n_\ahha \!L^*}} &\to& {\textstyle{\bigwedge^{n-1}_\ahha \!L^*}}, \quad \go \mapsto \go(\cdot, \ldots, X), \\
%\mathsf{L}_X: {\textstyle{\bigwedge^n_\ahha \!L^*}} &\to& {\textstyle{\bigwedge^{n}_\ahha \!L^*}} 
%\end{eqnarray*}
of differential, insertion, and Lie derivative %(see, for example, \cite{Rin:DFOGCA, CanWei:GMFNCA}) 
of (generalised) forms along a (generalised) vector field $X \in L$, one has on $\textstyle{\bigwedge^n_\ahha \!L^*}$
\begin{eqnarray}
\label{immaamt1}
       (n+1) \, \mathsf{d} &=& F \BB F', \\
\label{immaamt2}     
     (n-1) \, \mathsf{i}_X &=& F \iota_{F^*X} F', \\
\label{immaamt3}
         n \, \mathsf{L}_X &=& F \lie_{F^*X} F'.
\end{eqnarray}
\end{prop}
\begin{proof}
For the general Gerstenhaber product \rmref{maxdudler} one computes with the commutativity of $J\!L$,
\rmref{kaesekuchen1}--\rmref{kaesekuchen2}, \rmref{Sch3}, and using
the isomorphism \rmref{zerolimits}, 
\begin{footnotesize}
\begin{equation*}
\begin{split}
& \big((u^1 \otimes_\ahha \cdots \otimes_\ahha u^p) \circ_i (v^1 {\otimes_\ahha} \cdots {\otimes_\ahha} v^q)\big)\big(f^1, \ldots, f^{p+|q|}\big) \\
&= \big(u^1 \otimes_\ahha \cdots \otimes_\ahha u^p\big)\big(f^1, \ldots, f^{i-1}, \DD_{v^1 {\scriptscriptstyle{\otimes_\ahha}} \cdots {\scriptscriptstyle{\otimes_\ahha}} v^q} (f^i, \ldots, f^{i+|q|}), f^{i+q}, \ldots, f^{p+|q|}\big) \\
&= Sf^1(u^1) \cdots Sf^{i-1}(u^{i-1}) \big(S\big(s(Sf^i_{(1)}(v^1) \cdots  Sf^{i+|q|}_{(1)}(v^q)) f^i_{(2)} \cdots f^{i+|q|}_{(2)}\big)\big)\big(u^i\big) \\
& \hspace*{9cm}   Sf^{i+q}(u^{i+1})  \cdots Sf^{p+|q|}(u^{p}) \\
&= Sf^1(u^1) \cdots Sf^{i-1}(u^{i-1}) 
\,
\gve\big(u^i_{(1)+} \gve(v^1_+ f^i_{(1)}(v^1_-))f^i_{(2)}(u^i_{(1)-})\big) 
\cdots 
\\ & \hspace*{3cm}   
\gve\big(u^i_{(q)+} \gve(v^q_+ f^{i+|q|}_{(1)}(v^q_-))f^{i+|q|}_{(2)}(u^i_{(q)-})\big) 
Sf^{i+q}(u^{i+1})  \cdots Sf^{p+|q|}(u^{p}) \\
&= Sf^1(u^1) \cdots Sf^{i-1}(u^{i-1}) 
\,
\gve\big(u^i_{(1)+} v^1_+ f^i(v^1_- u^i_{(1)-})\big) 
\cdots 
\\ & \hspace*{3cm}   
\gve\big(u^i_{(q)+}v^q_+ f^{i+|q|}(v^q_- u^i_{(q)-})\big) 
Sf^{i+q}(u^{i+1})  \cdots Sf^{p+|q|}(u^{p}) \\
&=
 Sf^1(u^1) \cdots Sf^{i-1}(u^{i-1}) 
\,
Sf^i(u^i_{(1)} v^1) 
\cdots 
Sf^{i+|q|}(u^i_{(q)}v^q) 
Sf^{i+q}(u^{i+1})  \cdots Sf^{p+|q|}(u^{p}) \\
&=  \big((u^1 \otimes_\ahha \cdots \otimes_\ahha u^p) \circ_i^{\scriptscriptstyle{\rm tens}}  
(v^1 {\otimes_\ahha} \cdots {\otimes_\ahha} v^q)\big)\big(f^1, \ldots, f^{p+|q|}\big)
\end{split}
\end{equation*}
\end{footnotesize}
for $f^i \in J\!L$ and $u^j, v^k \in V\!L$. 
The fact that the Gerstenhaber bracket resulting from 
\rmref{rosenthalerplatz} corresponds to the (generalised) Schouten-Nijenhuis bracket on ${\textstyle\bigwedge^\bull_\ahha\!L}$ 
by means of the (generalised) Hochschild-Kostant-Rosenberg map was already shown in 
\cite[Theorem~1.4]{Cal:FFLA}. Hence, observing that the map
$\frac{1}{n!} F^*$ is the mentioned HKR morphism followed by
\rmref{zerolimits}, the first claim is proven.

 Concerning the identity \rmref{immaamt1}, 
as stated in \rmref{schonwiederzuweniggegessenundzuvielKaffeegetrunken}, 
the degenerate elements of $\BB$ vanish under $F$, whereas
the operator \rmref{alhambra} assumes the form
\begin{equation*}
\label{bohnen&speck}
\sss_{-1} \NN(f^1, \ldots, f^n) = \sum^n_{i=0} (-1)^{in} (f^{i+1}_+, \ldots, f^n_+, f^n_- \cdots f^1_-, f^1_+, \ldots, f^{i}_+) 
\end{equation*}
for an element $(f^1, \ldots, f^n) \in C_n(J\!L, A)$,
as is quickly revealed by a direct computation using \rmref{solimando}, \rmref{tellmemore}, and the commutativity of $J\!L$.
Hence, since $S$ is an involution and with \rmref{apfelorangeingwersaft}, \rmref{Sch3}, \rmref{ayran}, and \rmref{kaesekuchen1}--\rmref{kaesekuchen3} one has
\begin{footnotesize}
\begin{equation*}
\begin{split}
\big(& F \BB F'(\ga^1, \ldots, \ga^n)\big)\big(X^0 \wedge \cdots \wedge X^{n}\big) \\
&=  F\Big( \sum^n_{i=0} (-1)^{in} \sum_{\gs \in S_n}(-1)^{\gs} 
\big((\ga^{\gs(i+1)} \pr)_+, \ldots, (\ga^{\gs(n)} \pr)_+, \\
& \qquad \qquad (\ga^{\gs(n)} \pr)_- \cdots (\ga^{\gs(1)}\pr)_-, (\ga^{\gs(1)} \pr)_+, \ldots, (\ga^{\gs(i)} \pr)_+ \big)\Big)\Big(X^0 \wedge \cdots \wedge X^{n}\Big) \\
&= (n+1)\!\! \sum_{\gs \in S_n}\!\!(-1)^{\gs} 
S\big((\ga^{1} \pr)_{(1)}\big)\big(X^{\gs(1)}\big) \cdots S\big((\ga^{n} \pr)_{(1)}\big)\big(X^{\gs(n)}\big) 
\\ & \hspace*{5cm} 
\cdot \big((\ga^{n} \pr)_{(2)} \cdots (\ga^{1}\pr)_{(2)}\big)\big(X^{\gs(0)}\big) \\
&= (n+1)\!\! \sum_{\gs \in S_n}\!\!(-1)^{\gs} 
\gve\big(X^{\gs(1)}_+ (\ga^{1} \pr)(X^{\gs(1)}_- X^{\gs(0)}_{(1)})\big) 
\cdots
\gve\big(X^{\gs(n)}_+ (\ga^{n} \pr)(X^{\gs(n)}_- X^{\gs(0)}_{(n)})\big) \\
&=  (n+1) \sum^n_{i=1} \sum_{\gs \in S_n}\!\!(-1)^{\gs} 
\gve\big(X^{\gs(1)}_+ (\ga^{1} \pr)(X^{\gs(1)}_-)\big) 
\cdots \gve\big(X^{\gs(i)}_+ (\ga^{i} \pr)(X^{\gs(i)}_- X^{\gs(0)})\big) \\
& \hspace*{6cm} \cdots 
\gve\big(X^{\gs(n)}_+ (\ga^{n} \pr)(X^{\gs(n)}_-)\big) \\
&= (n+1) \sum^n_{i=1} \Big[ (-1)^{n-1} \! \sum_{\gs \in S_n}\!\!(-1)^{\gs} 
\ga^{1}(X^{\gs(1)}) 
\cdots X^{\gs(i)} \big(\ga^{i}(X^{\gs(0)})\big) \cdots \ga^{n} (X^{\gs(n)}) \\
& \quad +  
(-1)^{n} \! \sum_{\gs \in S_n}\!\!(-1)^{\gs} 
\ga^{1}(X^{\gs(1)}) 
\cdots (\ga^{i}\pr)(X^{\gs(i)} X^{\gs(0)}) \cdots \ga^{n} (X^{\gs(n)})\Big] \\
&= (n+1) \, \mathsf{d}(\ga^1 \wedge \cdots \wedge \ga^n)(X_0, \ldots, X_n),
\end{split}
\end{equation*}
\end{footnotesize}
where the last line follows from the fact that the vector fields are
derivations on $A$ and that 
 $\pr(XY - YX) = \pr([X,Y])=[X,Y]$.

As for the insertion operator, 
we compute with \rmref{tatue}, \rmref{kaesekuchen1}--\rmref{kaesekuchen3}, and $St = s$:
\begin{footnotesize}
\begin{equation*}
\begin{split}
&\big(F \iota_{F^*{X^n}} F'(\ga^1, \ldots, \ga^n)\big)\big(X^1 \wedge \cdots \wedge X^{n-1}\big) \\
&=  \sum_{\gs \in S_n}(-1)^\gs  F \Big(\pr^*\ga^{\gs(1)}, \ldots, \pr^*\ga^{\gs(n-2)}, \big((\pr^*\ga^{\gs(n)})(F^*X^n)\big) \blact \pr^*\ga^{\gs(n-1)}\Big)
\\ & \hspace*{9cm}  
\Big(X^1 \wedge \cdots \wedge X^{n-1}\Big) \\
&= (-1)^{n-1} (n-1) \! \sum_{\gs \in S_n}(-1)^\gs  
(S(\ga^{1} \pr))(X^{\gs(1)}) \cdots (S(\ga^{n-2} \pr))(X^{\gs(n-2)}) 
\\ & \hspace*{5cm}  
\Big(S\big((\ga^{n-1}\pr) t(\ga^n\pr(F^*X^{\gs(n)}))\big)\Big)\Big(X^{\gs(n-1)}\Big) \\
&= (n-1) \sum_{\gs \in S_n}(-1)^\gs  
\ga^{1}(X^{\gs(1)}) \cdots \ga^{n-2}(X^{\gs(n-2)}) s(\ga^n(X^{\gs(n)}))(X^{\gs(n-1)}_{(1)})
(\ga^{n-1}\pr) (X^{\gs(n-1)}_{(2)}) \\
&= (n-1) \sum_{\gs \in S_n}(-1)^\gs  
\ga^{1}(X^{\gs(1)}) \cdots  
\ga^{n-1} (X^{\gs(n-1)}) \ga^n(X^{\gs(n)}) \\
&= (n-1) \big(\mathsf{i}_{X^n}(\ga^1 \wedge \cdots \wedge \ga^n)\big)\big(X^1, \ldots, X^{n-1}\big),
\end{split}
\end{equation*}
\end{footnotesize}
hence \rmref{immaamt2} is proven.

In a similar way, one proves \rmref{immaamt3} the details of which we omit since the computation is similar to those of the two preceding identities.
\end{proof}

\section{Hochschild (co)homology and twisted Calabi-Yau algebras}

In this final section we discuss as an example  
the action of the 
Hochschild cohomology $H^\bull(A,A)$ 
of an associative algebra $A$ 
on the Hochschild homology 
$H_\bull(A,M)$ with coefficients in suitable $A$-bimodules $M$.
In particular, the differential calculus discussed in \cite{NesTsy:OTCROAA} 
is generalised towards nontrivial coefficients which are not even SaYD modules, and 
this is used to prove  
Theorem~\ref{ginzburger}.

\subsection{The Hopf algebroid $\Ae$ and the 
coefficients $A_\sigma $}
As said in the 
introduction, all the main results of this
paper were historically first obtained  
for the 
Hochschild cohomology $H^\bull(A,A)$ and homology 
$H_\bull(A,A)$ of an associative
$k$-algebra $A$.
This arises as the special 
case in which $U$ is the enveloping
algebra $\Ae$ of $A$, with 
$\eta=\mathrm{id}_{\Ae}$ and 
coproduct and counit given by
$$
        \Delta: U \to U \otimes_\ahha U, \> a
        \otimes_k b \mapsto (a \otimes_k 1) 
        \otimes_\ahha 
        (1 \otimes_k b),\quad 
        \varepsilon: U \to A, \> a \otimes_k
        b \mapsto ab.
$$
One then has
$$
 {}_\blact U \otimes_\Aopp  U_\ract = U
 \otimes_k U/{\rm span}_k\{(a \otimes_k cb) 
\otimes_k (a' \otimes_k b')
 -(a \otimes_k b) \otimes_k 
(a' \otimes_k b'c)\},
$$
where $cb$ and $b'c$ is understood to be
the product in $A$, and one
easily verifies that
$$
(a \otimes_k b)_+ \otimes_\Aopp  
(a \otimes_k b)_- := (a \otimes_k 1)
\otimes_\Aopp  (b \otimes_k 1)
$$
yields an inverse of the Galois map as was originally
pointed out by Schauenburg. 
For simplicity, we shall 
assume throughout this section that 
$k$ is a field which implies in particular that $U=\Ae$ 
is $A$-projective (in fact free) with respect to all four actions
$\lact,\ract,\blact,\bract$.

Like $J\!L$ in the previous section, 
$U=\Ae$ is an example of a full 
Hopf algebroid in the sense of
B\"ohm and Szlach\'anyi whose antipode
$S(a \otimes_k b):=b \otimes_k a$ is an involution.
We use this to identify   
left and right $U$-modules. Obviously, $U$-modules 
can also be identified with $A$-bimodules with symmetric action of
$k$, and in the sequel $M$ is such a bimodule that will be viewed 
freely as left or right $U$-module as necessary.

In particular, any algebra endomorphism 
$\gs: A \to A$ defines an $A$-bimodule
$A_\gs$ which is $A$ as
$k$-vector space with the $A$-bimodule respectively 
right $\Ae$-module structure 
$$
		  b \blact m  \bract a=m(a \otimes_k b) := 
		  bx\gs(a), \qquad a, m \in A, b \in \Aop.
$$
These bimodules are 
prototypical examples of the homology coefficients we are interested
in. They carry a left $\Ae$-comodule
structure given by 
$$
		  A_\gs \to \Ae \otimes_\ahha A_\gs,
		  \quad 
		  m \mapsto (m\otimes_k 1) \otimes_\ahha 1,
$$
for which the induced left $A$-module structure is 
${}_\blact A$. However, in general $A_\sigma $ is not a stable anti
Yetter-Drinfel'd module, see \cite{KowKra:CSIACT} for a discussion of
this fact.

Up to isomorphism, $A_\sigma$ only depends on the class of $\sigma $
in the outer automorphism group $ \mathrm{Out}(A)$ 
of $A$, and $ \sigma \mapsto A_\sigma $ yields 
an embedding of the latter 
into the Picard group of $\umod$ that appears to 
have been considered in detail for the first time by
Fr\"ohlich \cite{Fro:TPGONRIPOO}.   
The study of the (co)homology of $A$ with coefficients in these
bimodules
has many motivations. Nest and Tsygan suggested 
to view the Hochschild cohomology groups $H^\bull(A,A_\sigma)$ 
as defining a quantum analogue of the Fukaya category
\cite{NesTsy:OTCROAA, NesTsy:TFTCFAA} while Kustermans, Murphy and Tuset 
related $H_\bull(A,A_\sigma)$ to Woronowicz's concept of covariant
differential calculi over compact quantum groups
\cite{KusMurTus:DCOQGATCC}. Moreover, they arise naturally in the
description of the Hochschild (co)homology of the 
crossed product 
$A \rtimes_\sigma \mathbb{Z}$, see \cite{GetJon:TCHOCPA}.  

\subsection{The Hochschild (co)chain complex}
In this situation, the chain complex 
$C_\bull(U,M) = M \otimes_\Aopp U^{\otimes_\Aopp \bull}$ 
is isomorphic to the standard
Hochschild chain complex
$$
        C_\bull(A,M):=M \otimes_k A^{\otimes_k \bull}
$$  
by means of the map
$$
m \otimes_\Aopp (a_1 \otimes_k b_1) \otimes_\Aopp \cdots \otimes_\Aopp (a_n \otimes_k b_n) 
\mapsto 
b_n \cdots b_1 m \otimes_k a_1 \otimes_k \cdots \otimes_k a_n.  
$$

For $M=A_\sigma $, the para-cyclic structure 
on $C_\bull(U,A_\sigma )$ from Proposition~\ref{kamille} becomes 
under this isomorphism
\begin{equation*}
\!\!\!
\begin{array}{rcll}
        \dd_i(m \otimes_k y)  
&\!\!\!\!\! =& \!\!\!\!\!
        \left\{ \!\!\!\begin{array}{l}
        a_n m \otimes_k a_1  \otimes_k   \cdots   \otimes_k   a_{n-1}
\\
        m \otimes_k \cdots \otimes_k  a_{n-i} a_{n-i+1}
        \otimes_k  \cdots 
\\
        m\gs(a_1) \otimes_k a_2  \otimes_k   \cdots    \otimes_k  a_n 
        \end{array}\right.  
& \!\!\!\!\!\!\!\!\!\!\!\!   \begin{array}{l} 
        \mbox{if} \ i \!=\! 0, \\ \mbox{if} \ 1
\!  \leq \! i \!\leq\! n-1, \\ \mbox{if} \ i \! = \! n, \end{array} \\
\\
\sss_i(m \otimes_k y) &\!\!\!\!\! =&\!\!\!\!\!  \left\{ \!\!\!
\begin{array}{l} m  \otimes_k   a_1  
\otimes_k   \cdots   \otimes_k
 a_n  \otimes_k 
		  1
\\
m \otimes_k \cdots \otimes_k   a_{n-i} 
\otimes_k   1  
\otimes_k   a_{n-i+1}  \otimes_k
 \cdots  
\\
m \otimes_k 1
\otimes_k a_1  \otimes_k   \cdots    \otimes_k  a_n 
\end{array}\right.   & \!\!\!\!\!\!\!\!\!\!\!  \begin{array}{l} 
\mbox{if} \ i\!=\!0, \\ 
\mbox{if} \ 1 \!\leq\! i \!\leq\! n-1, \\  \mbox{if} \ i\! = \!n, \end{array} \\
\\
\ttt_n(m \otimes_k y) 
&\!\!\!\!\!=&\!\!\!\!\! 
\gs(a_1) \otimes_k a_2 \otimes_k \cdots
\otimes_k a_n \otimes_k m, 
& \\
\end{array}
\end{equation*}
where $m \in A$ and where we abbreviate $y:=a_1
\otimes_k \cdots \otimes_k a_n$.
In particular, one has
$$
		  \TT=\gs \otimes_k \cdots
		  \otimes_k \gs,
$$
so $C_\bull(A,A_\gs)$ is
cyclic
if and only if $\gs= \id$ (in
which case $A_\gs$ is an SaYD module).

Likewise, there is an isomorphism of cochain complexes of $k$-vector
spaces  
$$
        C^\bull(U,A) \rightarrow 
        C^\bull(A,A) := \Hom_k(A^{\otimes_k \bull},A),\quad
        \varphi \mapsto \tilde \varphi , 
$$
where the latter is the standard Hochschild cochain 
complex \cite{Hoch:OTCGOAAA} 
and $ \tilde \varphi $ is defined by 
$$
        \tilde\gvf(a_1 \otimes_k \cdots \otimes_k a_n) 
        :=       \gvf\big((a_1 \otimes_k 1)  \otimes_\Aopp \cdots 
        \otimes_\Aopp (a_n \otimes_k 1)\big) 
$$
so that 
$$
        \gvf\big((a_1 \otimes_k b_1)  \otimes_\Aopp \cdots 
        \otimes_\Aopp (a_n \otimes_k b_n)\big) =
        \tilde\gvf(a_1 \otimes_k \cdots \otimes_k a_n)b_n \cdots b_1. 
$$

The resulting operators involved in the calculus structure are 
given by
\begin{footnotesize}
\begin{eqnarray*}
       \BB(m \otimes_k y) &\!\!\!\!\!
=&\!\!\!\!\! \sum^n_{i=0} (-1)^{in} 1 \otimes_k a_{i+1} \otimes_k \cdots \otimes_k a_n \otimes_k m \otimes_k \gs(a_1) \otimes_k \cdots \otimes_k \gs(a_{i})
,\\
       \iota_{\tilde\varphi}(m \otimes_k y) &\!\!\!\!\!=&\!\!\!\!\! {\tilde\varphi}(a_{n-|p|}, \ldots, a_n)m \otimes_k a_1 \otimes_k \cdots \otimes_k a_{n-p}, \\
       \SSS_{\tilde\varphi}(m \otimes_k y) &\!\!\!\!\!=&\!\!\!\!\! \sum^{n-p}_{j=0} \sum^j_{i=0} (-1)^{\eta^{n,p}_{j,i}} 1 \otimes_k  
\gs(a_{n-|p|-j}) \otimes_k \cdots \otimes_k {\tilde\varphi}\big(\gs(a_{n-|p|+i-j}) \otimes_k \cdots \otimes_k \gs(a_{n+i-j})\big) \\
&& \qquad \qquad \otimes_k  \cdots \otimes_k \gs(a_n) \otimes_k \gs(m) \otimes_k \gs^2(a_1) \otimes_k \cdots \otimes_k \gs^2(a_{n-p-j}) \\
       \lie_{\tilde\varphi}(m \otimes_k y) &\!\!\!\!\!=&\!\!\!\!\! 
\sum^{n-|p|}_{i=1} (-1)^{\theta^{n,p}_i}  
\gs(m) \otimes_k \cdots \otimes_k {\tilde\varphi}\big(\gs(a_{i}) \otimes_k \cdots \otimes_k \gs(a_{i+|p|})\big) \otimes_k \cdots \otimes_k \gs(a_n) \\
&& \quad +
\sum^{p}_{i=1} (-1)^{\xi^{n,p}_i}  
\gs\Big({\tilde\varphi}\big(a_{n-|p|+i} \otimes_k \cdots \otimes_k a_n \otimes_k m \otimes_k \gs(a_{1}) \otimes_k \cdots \otimes_k \gs(a_{i-1})\big)\Big) \\ 
&& \qquad \qquad 
\otimes_k \gs(a_i) 
\otimes_k \cdots \otimes_k \gs(a_{n-p+i}), 
\end{eqnarray*}
\end{footnotesize}

Here we again work with the reduced complexes,
so ${\tilde\varphi} \in \bar C^p(A,A)$ and
$(m \otimes_k y)$ represents a class in 
$\bar C_\bull(A,A_\sigma)$.  
For $ \sigma = \mathrm{id} $ 
these operators 
appeared in \cite{Rin:DFOGCA,NesTsy:OTCROAA, Get:CHFATGMCICH}.

\subsection{The case of semisimple $ \sigma $}
A particularly well-behaved case is when the automorphism 
$ \sigma $ is semisimple (diagonalisable), 
that is, if there is a subset $\Sigma \,{\subseteq}\,k \setminus \{0\}$ and a decomposition 
of $k$-vector spaces
$$
        A=\bigoplus_{\lambda \in \Sigma} A_\lambda,\quad
        A_\lambda = \{a \in A \mid \sigma (a)=\lambda a\}.
$$
Note that we have $1 \in \Sigma $ because $\sigma (1)=1$, and also that an algebra $A$
equipped with such an automorphism is exactly the same as  
a $G$-graded algebra, where $G$ is a submonoid of the multiplicative group 
$k \setminus \{0\}$, as $\sigma (ab)=\sigma (a) \sigma (b)$ implies 
$A_\lambda A_\mu \,{\subseteq}\,A_{\lambda \mu}$ (thus the monoid $G \,{\subseteq}\,k \setminus \{0\}$ 
resulting from $ \sigma \in \mathrm{Aut}(A)$ is the one generated by 
$ \Sigma $).

This grading yields decompositions of $C^\bull(A,A)$ and
$C_\bull(A,A_\sigma)$. The chain complex $C_\bull(A,A_\sigma)$ becomes
$G$-graded by the total degree of a tensor,
$$
        C_\bull(A,A_\sigma)=\bigoplus_{\lambda \in G}
        C_\bull(A,A_\sigma)_\lambda,\quad
        C_n(A,A_\sigma)_\lambda=\bigoplus_{\lambda_0,\ldots,\lambda_n
          \in G \atop \lambda_0 \cdots \lambda_n=\lambda} 
        A_{\lambda_0} \otimes_k \cdots \otimes_k A_{\lambda_n},
$$
which is a decomposition of chain complexes of $k$-vector spaces. 
This coincides with the decomposition into eigenspaces of $\TT$, and
in particular we have
$$
        \mathrm{ker}(\mathrm{id}-\TT)=C_\bull(A,A_\sigma)_1,\quad
        \mathrm{im}(\mathrm{id}-\TT)=\bigoplus_{\lambda \in G \setminus \{1\}}
        C_\bull(A,A_\sigma)_\lambda.
$$
It is also immediately seen that this decomposition is in fact one of 
para-cyclic $k$-vector spaces, so we have:
\begin{lemma}\label{endspurt}
If $A$ is an algebra over a field $k$ and $ \sigma \in
\mathrm{Aut}(A)$ is a semisimple automorphism, then 
the para-cyclic $k$-vector space $C_\bull(A,A_\sigma)$ is
quasi-cyclic.  
\end{lemma}  

Unless $G$ is finite, the decomposition of the cochain complex
$C^\bull(A,A)$ is slightly more subtle. 
Given a cochain 
$\tilde \varphi \in C^p(A,A)$, we denote by 
$\tilde \varphi_\lambda $
its homogeneous component of degree $\lambda \in k \setminus
\{0\}$. That is,   
$\tilde \varphi_\lambda : A^{\otimes_k p} \rightarrow A$ 
is given 
on the homogeneous component 
$$
        (A^{\otimes_k p})_\mu :=
        \bigoplus_{\mu_1,\ldots,\mu_p
          \in G \atop \mu_1 \cdots \mu_p=\mu} 
        A_{\mu_1} \otimes_k \cdots \otimes_k A_{\mu_p}
$$ 
of elements of $A^{\otimes_k p}$ of total degree  $ \mu \in G$
by
$$
       \tilde \varphi_\lambda :=
       \pi_{\lambda \mu} \circ \tilde \varphi : 
       (A^{\otimes_k p})_\mu \rightarrow A_{\lambda \mu},
$$
where $ \pi_\nu : A \rightarrow A_\nu $ is the projection onto the
degree $ \nu $ part of $A$.
If we denote by 
$$
        C^p(A,A)_\lambda := \{ \tilde \varphi \in C^p(A,A)
        \mid 
        \tilde \varphi ((A^{\otimes_k p})_\mu) \,{\subseteq}\,
        A_{\lambda \mu}\}
$$ 
the set of all $\lambda$-homogeneous
$p$-cochains, then 
$\tilde \varphi \mapsto 
\{\tilde \varphi_\lambda\}_{\lambda \in k \setminus\{0\}}$
defines an embedding
$$
        C^\bull(A,A) \rightarrow \prod_{\lambda \in k \setminus \{0\}} 
        C^\bull(A,A)_\lambda
$$ 
of cochain complexes of $k$-vector spaces which is, however,
not a quasi-isomorphism in general. 
Still, we can split off the homogeneous part of degree $1$,
$$
        C^\bull(A,A) \simeq C^\bull(A,A)_1 \oplus 
        \Bigl(C^\bull(A,A) \cap \prod_{\lambda \in k \setminus \{0,1\}} 
        C^\bull(A,A)_\lambda\Bigr),
$$
and $C^\bull(A,A)_1$ consists precisely of those cochains $\tilde \varphi$
for which $\DD'_{\tilde \varphi}$ commutes with $\TT$.

Note that $C^\bull(A,A)_1$ is not equal 
to $C^\bull_{A_\sigma}(A,A)$ in general. We rather have:

\begin{lemma}
With the assumptions and notation as above,  
we have 
$$
        C^p_{A_\sigma}(A,A) = \big\{
        \tilde \varphi \in C^p(A,A) \mid 
        \forall \lambda \in k \setminus\{0,1\} \forall \mu \in G : 
        \tilde \varphi_\lambda |_{(A^{\otimes_k p})_{\lambda^{-1}
            \mu^{-1}}} =0 \big\}.
$$   
\end{lemma} 
\begin{proof}
This follows from the fact that the operator $\DD'_{\tilde
  \varphi_\lambda}$ maps a chain  
$x \otimes_k y \in C_{n+p}(A,A_\sigma)_{\lambda^{-1}}
\,{\subseteq}\,\mathrm{im}(\mathrm{id}-\TT)$ 
to 
$x \otimes_k \tilde \varphi _\lambda (y) \in
C_{n+1}(A,A_\sigma)_1
\,{\subseteq}\,\mathrm{ker}(\mathrm{id}-\TT)$. 
\end{proof}

From this it is clear that the projections onto the homogeneous parts
leave $C^\bull_{A_\sigma}(A,A) \,{\subseteq}\,C^\bull(A,A)$ invariant, so 
$C^\bull_{A_\sigma}(A,A)$ splits as well 
as a direct sum of cochain complexes
into 
$C^\bull(A,A)_1$ and 
$C^\bull_{A_\sigma}(A,A) \cap \prod_{\lambda \neq 1} 
C^\bull(A,A)_\lambda$. We therefore obtain:  

\begin{lemma}\label{schweinebucht}
If $A$ is an algebra over a field $k$ and $ \sigma \in
\mathrm{Aut}(A)$ is a semisimple automorphism, then 
$C^\bull(A,A)_1$ is a comp subalgebra of $C^\bull_{A_\sigma}(A,A)$,
and the induced morphisms
$$
        H^\bull(C(A,A)_1) \rightarrow H^\bull_{A_\sigma}(A,A),\quad 
        H^\bull(C(A,A)_1) \rightarrow H^\bull(A,A)
$$ 
are  
injective and split as maps 
of $H^\bull(C(A,A)_1)$-modules. 
\end{lemma} 

\begin{example}
Let $k$ be any field, $A$ be the polynomial ring 
$k[x]$, and $ \sigma $ be specified by 
$\sigma (x)=qx$ for some fixed $q \in k\setminus \{0\}$ 
which is assumed to be not a root of unity. Then 
we have $\Sigma =
\{q^n \mid n \in \mathbb{N} \}=G \simeq \mathbb{N}$, 
and $ \mathrm{ker} (\mathrm{id} - \TT)$ consists only of the
(degenerate) multiples of $1 \otimes_k \cdots \otimes_k 1$. 
Then $C^p(A,A)_1 \simeq k$ for all $p$ while 
$C^\bull_{A_\sigma}(A,A)$ consists of all cochains that do not 
decrease the degree (where ``decrease'' refers to the ordering of $G
\simeq \mathbb{N}$). In particular, 
$C^0(A,A)_1 \simeq k$ while 
$C^0_{A_\sigma}(A,A) \simeq A$, and as 
$A$ is commutative, we also have 
$H^0_{A_\sigma}(A,A) \simeq A$ while 
$H^0(C(A,A)_1) \simeq k$.   
\end{example}

\subsection{Twisted Calabi-Yau algebras}
More recently, the Hochschild homology groups with 
coefficients in $A_\sigma $ have been studied
intensively for the fact that large classes of algebras have
been recognised to
be what is nowadays called a twisted Calabi-Yau algebra: 

\begin{definition}\label{mathar}
An algebra $A$ is a \emph{twisted Calabi-Yau algebra} 
with \emph{modular
automorphism} $\sigma \in \mathrm{Aut} (A)$ if 
the $\Ae$-module $A$ has (as an $\Ae$-module) 
a finitely generated projective resolution of finite length
and there exists $d \in \mathbb{N} $ and isomorphisms 
of right $\Ae$-modules
$$
        \mathrm{Ext}^i_\Ae(A,\Ae) \simeq 
        \left\{
          \begin{array}{ll}
            0 & i \neq d,\\
            A_\sigma & i=d.\end{array}\right.
$$    
\end{definition}

The number $d$ is then necessarily the \emph{dimension} of $A$ 
in the sense
of \cite{CarEil:HA}, that is, the 
projective dimension of $A \in \amoda$, and  
the Ischebeck spectral sequence 
\cite{Isch:EDZDFEUT} leads to a Poincar\'e-type
duality 
\begin{equation}\label{puankare}
        H^\bull(A,A) \simeq H_{d-\bull}(A,A_\sigma).  
  \end{equation}

We refer to
\cite{BerSol:ACFHPTBTCY, Bic:HHOHAAFYDROTC, BroZha:DCATHCFNHA, Gin:CYA, Kel:DCYC, Kra:OTHCOQHS, LiuWu:RDCOQHS, VdB:ARBHHACFGR, VdB:CYAAS, VdBdTdV:CYDANCH}   
and the references therein for more information and
background, and in particular plenty of
examples.

%In \cite{KowKra:CSIACT} we explained that 
%$A_\sigma$ carries a natural left coaction of 
%$U=\Ae$ such that 
%the para-cyclic object $C_\bull(U,A_\sigma)$ is quasi-cyclic if 
%$\sigma $ is semisimple, so that under this condition 
%and if $A$ is $k$-projective Theorem~\ref{pen} applies 
%to $H_\bull(A,A_\sigma)=\mathrm{Tor}_\bull^U(A_\sigma,A)$. 

It had been our aim in \cite{KowKra:DAPIACT} to understand 
the duality (\ref{puankare}) in the wider context of Hopf algebroids 
and to observe that 
(\ref{puankare}) is an isomorphism of graded 
$H^\bull(A,A)$-modules.
%
%
%Inspired by the results we then suspected 
%it could also be used to generalise the 
%results of \cite{Rin:DFOGCA,Ger:TCSOAAR,Goo:CHDATFL, 
%Get:CHFATGMCICH,NesTsy:OTCROAA} from $A$ to $A_\sigma $.   
From that point of view, the essence of the present paper 
is that
(\ref{puankare}) is 
even compatible with the Gerstenhaber structure
which implies Theorem~\ref{ginzburger}. 
For $ \sigma = \mathrm{id} $ 
this theorem has been proven by Ginzburg in \cite{Gin:CYA} and just as
therein, the fact is more or less immediate once the full differential
calculus structure is established: 

\begin{proof}[Proof of Theorem~\ref{ginzburger}]
First we need to observe that in the case of a twisted Calabi-Yau
algebra, we have $H^\bull(A,A) \simeq H^\bull(C(A,A)_1)$.
Indeed, we know already that the duality isomorphism 
(\ref{puankare}) is an isomorphism of 
$H^\bull(A,A)$-modules, see, for instance, Theorem~1 
in \cite{KowKra:DAPIACT}. By Lemma~\ref{endspurt} we know that the
homology is in fact concentrated in degree $1$ with respect to the
$G$-grading. Hence the cohomology is also concentrated in degree
$1$, that is, the embedding $C^\bull(A,A)_1 \rightarrow 
C^\bull(A,A)$ is a quasi-isomorphism.

Now Theorem~\ref{pen1} states in combination
with Theorem~1 in \cite{KowKra:DAPIACT} precisely that 
$H^\bull(A,A)$ 
and $H_\bull(A,A_\sigma)$ form for
a twisted Calabi-Yau algebra with 
semisimple modular automorphism 
$\sigma $ 
what Lambre calls a differential calculus with duality 
\cite[D\'efinition~1.2]{Lam:VDBDABVSOCYA}.  
Hence \cite[Corollaire~1.6]{Lam:VDBDABVSOCYA} 
directly implies Theorem~\ref{ginzburger}. 
\end{proof}

\providecommand{\bysame}{\leavevmode\hbox to3em{\hrulefill}\thinspace}
\providecommand{\MR}{\relax\ifhmode\unskip\space\fi MR }
% \MRhref is called by the amsart/book/proc definition of \MR.
\providecommand{\MRhref}[2]{%
  \href{http://www.ams.org/mathscinet-getitem?mr=#1}{#2}
}
\providecommand{\href}[2]{#2}

\end{document}